\documentclass[a4paper,11pt]{article}

\usepackage{a4wide}
\usepackage[utf8x]{inputenc}
\usepackage[T1]{fontenc}
\usepackage{amsfonts}
\usepackage{mathrsfs}
\usepackage{amsthm,amsrefs}
\usepackage{thmtools}

\usepackage{mathtools}%an extension package to amsmath
\mathtoolsset{showonlyrefs}% for hiding unused labels

\usepackage[all]{xy}
\usepackage{xcolor}
\usepackage{enumitem}
\usepackage[pdftex]{graphicx}
\usepackage[pdftex,
colorlinks=true,
linkcolor=blue,
citecolor=violet,
hyperindex,
plainpages=false,
bookmarks=true,
bookmarksopen=true,
bookmarksnumbered,]{hyperref}

\theoremstyle{plain}
\newtheorem{theorem}{Theorem}[section]
\newtheorem{corollary}[theorem]{Corollary}
\newtheorem{lemma}[theorem]{Lemma}

\newtheorem{remark}[theorem]{Remark}

\theoremstyle{definition}
\newtheorem{definition}[theorem]{Definition}
\theoremstyle{remark}

\newtheorem*{theorem*}{Theorem \ref{mainthm}}

\numberwithin{equation}{section}

\newcommand{\HM}{\mathcal{H}}
\newcommand{\vect}[1]{{\rm span}\{#1\}}
\newcommand{\midd}{\;\middle|\;}
\newcommand{\id}{{\rm id}}
\newcommand{\supp}{{\rm supp}}

\newcommand{\scale}[1]{\boldsymbol{\mu}_{#1}}
\newcommand{\trans}[1]{\boldsymbol{\tau}_{#1}}
\newcommand{\mr}{\mathbin{\vrule height 1.6ex depth 0pt width 0.13ex
		\vrule height 0.13ex depth 0pt width 1.3ex}}
\makeatletter
\DeclareFontFamily{U}{tipa}{}
\DeclareFontShape{U}{tipa}{m}{n}{<->tipa10}{}
\newcommand{\arc@char}{{\usefont{U}{tipa}{m}{n}\symbol{62}}}%

\newcommand{\arc}[1]{\mathpalette\arc@arc{#1}}

\newcommand{\arc@arc}[2]{%
  \sbox0{$\m@th#1#2$}%
  \vbox{
    \hbox{\resizebox{\wd0}{\height}{\arc@char}}
    \nointerlineskip
    \box0
  }%
}
\makeatother

\DeclareMathOperator{\ap}{ap}
\DeclareMathOperator{\Tan}{Tan}
\DeclareMathOperator{\Lip}{Lip}
\DeclareMathOperator{\card}{card}
\DeclareMathOperator{\dist}{dist}
\DeclareMathOperator{\diam}{diam}
\DeclareMathOperator{\length}{length}

\begin{document}
\title{Local $C^{1,\beta}$-regularity at the boundary of two \protect\\ 
dimensional sliding almost minimal sets in $\mathbb{R}^{3}$}
\author{Yangqin Fang}
\renewcommand{\thefootnote}{\fnsymbol{footnote}} 
\footnotetext{\emph{2010 Mathematics Subject Classification:} 49K99, 49Q20,
49J99.}
\footnotetext{\emph{Key words:} Almost minimal sets, sliding boundary
conditions, regularity, blow-up limit, Plateau's problem, Hausdorff
measure, normalized Hausdorff distance.}
\footnotetext{ I would like to thank Guy David for giving me the thinking to
the project, I would also like to thank Ulrich Menne for his constant support 
and help.}
\renewcommand{\thefootnote}{\arabic{footnote}} 

\date{}
\maketitle
\begin{abstract}
	In this paper, we will give a $C^{1,\beta}$-regularity result on the 
	boundary for two dimensional sliding almost minimal sets in $\mathbb{R}^3$. 
	This effect may lead to the existence of a solution to the Plateau problem with
	sliding boundary conditions proposed by Guy David in \cite{David:2014p} in
	the case that the boundary is a 2-dimensional smooth manifold.
\end{abstract}
\section{Introduction}
Jean Taylor, in \cite{Taylor:1976}, proved a celebrated regularity result of
Almgren almost minimal sets, that gives a complete classification of the local
structure of 2-dimensional (almost) minimal sets. This result may apply to many
actual surfaces, soap films are considered as typical examples.  Guy David, in
\cite{David:2008}, gave a new proof of this result and generalized it to any
codimension. That is, every $2$-dimensional almost minimal set, in an open set
$U\subseteq \mathbb{R}^n$ with gauge function $h(t)\leq Ct^{\alpha}$, is local 
$C^{1,\beta}$ equivalent to a $2$-dimensional minimal cone. 

In \cite{David:2014p}, Guy David proposed to consider the Plateau Problem with
sliding boundary conditions, since it is very natural to soap films and Jean
Taylor's regularity also applies for sliding almost minimal sets away from the
boundary, and it also has some advantages to consider the local structure at the 
boundary. Motivated by these, regularity at the boundary would be well worth
our considering. In fact, a result similar to Jean Talyor's will be a 
satisfactory conclusion, for which together with Jean Taylor's theorem will 
imply the local Lipschitz retract property of sliding (almost) minimal sets,
and the existence of minimizers for the sliding Plateau Problem easily
follows. 

One of advantages of the sliding boundary conditions is that we have chance to
determine the possibility of minimal cones in the upper half space $\Omega_0$ 
of $\mathbb{R}^{3}$, where minimal cone is a cone but minimal, and minimal is
understood with sliding on the boundary $\partial \Omega_0$. Indeed, there no
more than seven kinds of cones which are minimal, they are $\partial \Omega_0$, 
cones of type $\mathbb{V}$, cones of type $\mathbb{P}_+$, cones of type 
$\mathbb{Y}_+$, cones of type $\mathbb{T}_+$ and cones $\partial \Omega_0 \cup
Z$ where $Z$ are cones of type $\mathbb{P}_+$ or $\mathbb{Y}_+$, see Section 3
in \cite{Fang:2015} for the precise definition of cones of type $\mathbb{P}_+$, 
$\mathbb{Y}_+$, $\mathbb{T}_+$ and $\mathbb{V}$, and also Remark 3.11 for the
claim. We ascertain that there are only there kinds of cones which are minimal
and contains the boundary $\partial \Omega_0$, they are $\partial \Omega_0$ 
and $\partial \Omega_0 \cup Z$ where $Z$ is cone of type $\mathbb{P}_+$ or
$\mathbb{Y}_+$, see Theorem 3.10 in \cite{Fang:2015} for the statement.

Another advantages of the sliding boundary conditions is that we can easily 
establish a monotony density property at the boundary, see Theorem
\ref{thm:ANDD} for precise statement. In fact, the monotony density property is
not enough, we have estimated the decay of the almost density, and that is
also possible with sliding on the boundary, see Corollary \ref{co:dendecay}. 

In \cite{Fang:2015}, we proved a H\"older regularity of two dimensional
sliding almost minimal set at the boundary. That is, suppose that
$\Omega\subseteq\mathbb{R}^{3}$ is a closed domain with boundary $\partial 
\Omega$ a $C^1$ manifold of dimension 2, $E\subseteq \Omega$ is a 2 dimensional 
sliding almost minimal set with sliding boundary $\partial \Omega$, and that
$\partial\Omega\subseteq E$. Then $E$, at the boundary, is locally biH\"older
equivalent to a sliding minimal cone in the upper half space $\Omega_0$. In 
this paper, we will generalized the biH\"older equivalence to a $C^{1,\beta}$
equivalence when the gauge function $h$ satisfies that $h(t)\leq Ct^{\alpha_1}$
and $\partial \Omega$ is a 2 dimensional $C^{1,\alpha}$ manifold.
Let us refer to Theorem \ref{mainthm} for details. Where the sliding minimal 
cones always contain the boundary $\partial \Omega_0$, namely only there kinds
of cones can appear: $\partial\Omega_0$ and $\partial \Omega_0\cup Z$, where $Z$ 
are cones of type $\mathbb{P}_+$ or $\mathbb{Y}_+$. 

Let us introduce some notation and definitions before state our main theorem.
A gauge function is a nondecreasing function $h:[0,\infty)\to[0,\infty]$ with
$\lim_{t\to 0} h(t)=0$. Let $\Omega$ be a closed domain of $\mathbb{R}^{3}$, 
$L$ be a closed subset in $\mathbb{R}^{3}$, $E\subseteq \Omega$ be a given set.
Let $U\subseteq \mathbb{R}^3$ be an open set. A family of mappings $\{ \varphi_t 
\}_{0\leq t\leq 1}$, from $E$ into $\Omega$, is called a sliding 
deformation of $E$ in $U$, while $\varphi_1(E)$ is called a competitor of
$E$ in $U$, if following properties hold:
\begin{itemize}[topsep=0pt,noitemsep]
\item $\varphi_t(x)=x$ for $x\in E\setminus U$, $\varphi_t(x)\subseteq U$ for
	$x\in E\cap U$, $0\leq t\leq 1$,
\item $\varphi_t(x)\in L$ for $x\in E\cap L$, $0\leq t \leq 1$,
\item the mapping $[0,1]\times E\to \Omega, (t,x)\mapsto \varphi_{t}(x)$ is
	continuous,
\item $\varphi_1$ is Lipschitz and $\psi_0=\id_E$.
\end{itemize}
\begin{definition}
We say that an nonempty set $E\subseteq \Omega$ is locally sliding almost
minimal at $x\in E$ with sliding boundary $L$ and with gauge function $h$, 
called $(\Omega,L,h)$ locally sliding almost at $x\in E$ for short, if $\HM^2\mr E$
is locally finite, and for any sliding deformation $\{ \varphi_t \}_{0\leq
t\leq 1}$ of $E$ in $B(x,r)$, we have that 
\[
	\HM^2(E\cap B(x,r))\leq \HM^2(\varphi_1(E)\cap B(x,r)) +h(r)r^2.
\]

We say that $E$ is sliding almost minimal with sliding boundary $L$ and gauge
function $h$, denote by $SAM(\Omega,L,h)$ the collection of all such sets, if
$E$ is locally sliding almost minimal at all points $x\in E$.
\end{definition}
 
For any $x\in \mathbb{R}^3$, we let $\trans{x}: \mathbb{R}^3\to 
\mathbb{R}^3$ be the translation defined by $\trans{x}(y)=y+x$, and let 
$\scale{r}:\mathbb{R}^{3}\to\mathbb{R}^{3}$
be the mapping defined by $\scale{r}(y)=ry$ for any $r>0$.
For any $S\subseteq \mathbb{R}^3$ and $x\in S$, a blow-up limit of $S$ at
$x$ is any closed set in $\mathbb{R}^3$ that can be obtained as the Hausdorff
limit of a sequence $\scale{1/r_k}\circ\trans{-x}(S)$ with $\lim_{k\to
\infty}r_k=0$. A set $X$ in $\mathbb{R}^3$ is 
called a cone centered at the origin $0$ if for any $\scale{t}(X)=X$ for any
$t\geq 0$; in general, we call a cone $X$ centered at $x$ if $\trans{-x}(X)$ 
is a cone centered at $0$. We denote by $\Tan(S,x)$ the tangent cone of
$S$ at $x$, see Section 2.1 in \cite{Allard:1972}.
We see that if there is unique blow-up limit of $S$ at $x$, then it
coincide with the tangent cone $\Tan(S,x)$. Our main theorem is the following.

\begin{theorem}\label{mainthm}
	Let $\Omega\subseteq\mathbb{R}^{3}$ be a closed set such that the 
	boundary $\partial\Omega$ is a $2$-dimensional manifold of class $C^{1,
	\alpha}$ for some $\alpha>0$ and $\Tan(\Omega,z)$ is a half space for any
	$z\in \partial \Omega$. Let $E\subseteq \Omega$ be a closed set such that
	$E\supseteq \partial \Omega$ and $E$ is a sliding almost minimal set with
	sliding boundary $\partial\Omega$ and with gauge function $h$ satisfying that 
	\[
		h(t)\leq C_h t^{\alpha_1},\ 0<t\leq t_0, \text{ for some }C_h>0,
		\alpha_1>0 \text{ and }t_0>0.
	\]
	Then for any $x_0\in \partial \Omega$, there is unique blow-up limit of
	$E$ at $x_0$; moreover, there exist a radius $r>0$, a sliding minimal cone
	$Z$ in $\Omega_0$ with sliding boundary $\partial\Omega_0$, and a mapping
	$\Phi:\Omega_0\cap B(0,1)\to \Omega$ of class $C^{1,\beta}$, which is a 
	diffeomorphism  between its domain and image, such that $\Phi(0)=x_0$, 
	$|\Phi(x)-x_0-x|\leq 10^{-2}r$ for $x\in B(0,2r)$, and 
	\[
		E\cap B(x_0,r)=\Phi(Z)\cap B(x_0,r).
	\]
\end{theorem}
Theorem \ref{mainthm} and Jean Taylor's theorem imply that any set $E$ as in
above theorem is Lipschitz neighborhood retract. This effect gives the
existence of a solution to the Plateau problem with sliding boundary
conditions in a special case, see Theorem \ref{thm:espp}.
\section{Lower bound of the decay for the density}
 In this section, we will consider a simple case that $\Omega$ is a half
 space and $L$ is its boundary; without loss of generality, we assume that
 $\Omega$ is the upper half space, and change the notation to be $\Omega_0$ for
 convenience, i.e.
\[
	\Omega_0=\{(x_1,x_2,x_3)\in \mathbb{R}^{3}\mid x_3\geq 0\}, L_0=\partial \Omega_0.
\]

It is well known that for any 2-rectifiable set $E$, there exists an approximate 
tangent plane $\Tan(E,y)$ of $E$ at $y$ for $\HM^2$-a.e. $y\in E$. We will 
denote by $\theta(y)\in [0,\pi/2]$ the angle between the segment $[0,y]$ and 
the plane $\Tan(E,y)$, by $\theta_x(y)\in [0,\pi/2]$ the angle between the 
segment $[x,y]$ and the plane $\Tan(E,y)$, for $x\in \mathbb{R}^{3}$.

In this section, we assume that there is a number $r_h>0$ such that  
\begin{equation}
\label{eq:AG}
\int_{0}^{r_h}\frac{h(2t)}{t}dt<\infty,
\end{equation}
and put 
\begin{equation}
\label{eq:gauge}
h_1(t)=\int_{0}^{t}\frac{h(2s)}{s}ds ,\text{ for }0\leq t\leq r_h.
\end{equation}
\begin{lemma}\label{le:DE}
	Let $E\subseteq \Omega_0$ be any 2-rectifiable set.
	Then, by putting $u(r)=\HM^2(E\cap B(x,r))$, we have that $u$ is 
	differentiable almost every $r>0$, and for such $r$,
	\begin{equation}\label{eq:DE1}
		\HM^1(E\cap \partial B(x,r))\leq u'(r).
	\end{equation}
\end{lemma}
\begin{proof}
	Considering the function $\psi:\mathbb{R}^{3}\to \mathbb{R}$ defined by
	$\psi(y)=|y-x|$, we have that, for any $y\neq x$ and $v\in \mathbb{R}^{3}$, 
	\[
		D\psi(y)v=\left\langle \frac{y-x}{|y-x|},v \right\rangle,
	\]
	thus
	\begin{equation}\label{eq:DE10}
		\ap J_1(\psi\vert_{E})(y)=\sup\{ |D\psi(y)v|: v\in \Tan(E,x),|v|=1 \}=\cos
		\theta_x(y).
	\end{equation}

	Employing Theorem 3.2.22 in \cite{Federer:1969}, we have that, for any
	$0<r<R<\infty$,
	\[
		\int_{r}^{R}\HM^1(E\cap \partial B(x,t))dt=\int_{E\cap
		B(x,R)\setminus B(x,r)}\cos_x(y)d\HM^2(y)\leq u(R)-u(r),
	\]
	we get so that, for almost every $r\in (0,\infty)$,
	\[
		\HM^1(E\cap \partial B(x,t))\leq u'(r).
	\]
\end{proof}
\begin{lemma}\label{le:DI}
	Let $E$ be a 2-rectifiable $(\Omega_0,L_0,h)$ locally sliding almost 
	minimal at $x\in E$. 
	\begin{itemize}[topsep=0pt,itemsep=0pt]
		\item If $x\in E\cap L_0$, then for $\HM^1$-a.e. 
			$r\in (0,\infty)$,
			\begin{equation}\label{eq:DI1}
				\HM^2(E\cap B(x,r))\leq \frac{r}{2}\HM^1(E\cap \partial
				B(x,r))+h(2r)(2r)^2.
			\end{equation}
		\item If $x\in E\setminus L_0$, then inequality \eqref{eq:DI1} holds for
			$\HM^1$-a.e. $r\in (0,\dist(x,L_0))$.
	\end{itemize}
\end{lemma}
\begin{proof}
	If $\HM^2(E\cap \partial B(x,r))>0$, then $\HM^1(E\cap \partial
	B(x,r))=\infty$, and nothing need to do. We assume so that $\HM^2(E\cap \partial
	B(x,r))=0$.

	Let $f:[0,\infty)\to [0,\infty)$ be any Lipschitz function, we let 
	$\phi:\Omega_0\to\Omega_0$ be defined by 
	\[
		\phi(y)=f(|y-x|)\frac{y-x}{|y-x|}.
	\]
	Then, for any $y\neq x$ and any $v\in \mathbb{R}^{3}$, by putting
	$\tilde{y}=y-x$, we have that 
	\begin{equation}\label{eq:Diff}
		D\phi(y)v=\frac{f(|\tilde{y}|)}{|\tilde{y}|}v+\frac{|\tilde{y}|f'(|\tilde{y}|)-f(|\tilde{y}|)}{|\tilde{y}|^{2}}
		\left\langle \frac{\tilde{y}}{|\tilde{y}|},v \right\rangle \tilde{y}
	\end{equation}

	If the tangent plane $\Tan^{2}(E,y)$ of $E$ at $y$ exists, we take 
	$v_1,v_2\in \Tan^{2}(E,y)$ such that $|v_1|=|v_2|=1$, $v_1$ is perpendicular
	to $y=x$, and that $v_2$ is perpendicular to $v_1$, let $v_3$ be a vector in
	$\mathbb{R}^{3}$ which is perpendicular to $\Tan^{2}(E,y)$ and $|v_3|=1$, then 
	\[
		\tilde{y}=\langle \tilde{y},v_2 \rangle v_2+\langle \tilde{y},v_3 \rangle
		v_3=|\tilde{y}|\cos \theta_x(y)
		v_2+|\tilde{y}|\sin \theta_x(y) v_3,
	\]
	and
	\[
		D\phi(y)v_1\wedge D\phi(y)v_2= \frac{f(|\tilde{y}|)^2}{|\tilde{y}|^2}v_1\wedge
		v_2+\frac{|\tilde{y}|f'(|\tilde{y}|)f(|\tilde{y}|)
		-f(|\tilde{y}|)^2}{|\tilde{y}|^3}\cos \theta_x(y) v_1\wedge \tilde{y},
	\]
	thus
	\begin{equation}\label{eq:Jacobian}
		\begin{aligned}
			\ap J_{2}(\phi\vert_{E})(y)	&=\|D\phi(y)v_1\wedge D\phi(y)v_2\|\\
			&=\frac{f(|\tilde{y}|)}{|\tilde{y}|}\left(f'(|\tilde{y}|)^2 \cos^2
			\theta_x(y)+\frac{f(|\tilde{y}|)^2}{|\tilde{y}|^2}\sin^2\theta_x(y)
			\right)^{1/2}.
		\end{aligned}
	\end{equation}

	We consider the function $\psi:\mathbb{R}^{3}\to \mathbb{R}$ defined by
	$\psi(y)=|y-x|$. Then, by \eqref{eq:DE10}, we have that 
	\[
		\ap J_{1}(\psi\vert_{E})(y)=\cos\theta_x(y).
	\]

	For any $\xi\in (0,r/2)$, we consider the function $f$ defined by 
	\[
		f(t)=\begin{cases}
			0,& 0\leq t\leq r-\xi\\
			\frac{r}{\xi}(t-r+\xi),& r-\xi< t\leq r\\
			t,&t>r.
		\end{cases}
	\]
	Then we have that 
	\[
		\ap J_{2}(\phi\vert_{E})(y)\leq
		\frac{f(|\tilde{y}|)f'(|\tilde{y}|)}{|\tilde{y}|}\cos\theta_x
		(y)+\frac{f(|\tilde{y}|)^2}{|\tilde{y}|^2}\sin\theta_x(y).
	\]
	Applying Theorem 3.2.22 in \cite{Federer:1969}, by putting 
	$A_{\xi}=E\cap B(0,r)\setminus B(0,r-\xi)$, we get that 
	\[
		\begin{aligned}
			\HM^2(\phi(E\cap B(0,r)))&\leq
			\int_{A_{\xi}}\frac{r^2}{\xi^2}\cdot 
			\frac{|\tilde{y}|-r+\xi}{|\tilde{y}|}\cos\theta_x(y)d\HM^2(y)
			+\frac{r^2}{(r-\xi)^2}\HM^2(A_{\xi})\\
			&=\int_{r-\xi}^{r}\frac{r^2(t-r+\xi)}{\xi^2 t}\HM^1(E\cap \partial
			B(x,t))dt+4\HM^2(A_{\xi}),
		\end{aligned}
	\]
	thus 
	\[
		\HM^2(E\cap B(0,r))\leq (2r)^2h(2r)+\lim_{\xi\to 0+}r^2
		\int_{r-\xi}^r\frac{t-r+\xi}{t\xi^2}\HM^1(E\cap \partial B(x,t))dt.
	\]
	Since the function $g(t)=\HM^1(E\cap B(x,t))/t$ is a measurable function, we
	have that, for almost every $r$,
	\[
		\lim_{\xi\to
		0+}\int_{0}^{\xi}\frac{tg(t-r+\xi)}{\xi^2}dt=\frac{1}{2}g(r),
	\]
	thus for such $r$,
	\[
		\HM^2(E\cap B(x,r))\leq (2r)^2h(2r)+\frac{r}{2}\HM^1(E\cap B(x,r)).
	\]
\end{proof}

For any set $E\subseteq \mathbb{R}^3$, we set 
\[
	\Theta_E(x,r)=r^{-2}\HM^2(E\cap B(x,r)),\ \mbox{for any } r>0,
\]
and denote by $\Theta_E(x)=\lim_{r\to 0+}\Theta_E(x,r)$ if the limit exist, we
may drop the script $E$ if there is no danger of confusion.
\begin{theorem}\label{thm:ANDD}
	Let $E$ be a 2-rectifiable $(\Omega_0,L_0,h)$ locally sliding almost 
	minimal at $x\in E$. 
	\begin{itemize}[nolistsep]
		\item If $x\in L_0$, then $\Theta(x,r)+8 h_1(r)$ 
			is nondecreasing as  $r\in (0,r_h)$.
		\item If $x\not\in L_0$, then
			$\Theta(x,r)+8 h_1(r)$ is nondecreasing as $ r\in (0,\min\{r_h,\dist(x,L)\})$.
	\end{itemize}

\end{theorem}
\begin{proof}
	From Lemma \ref{le:DI} and Lemma \ref{le:DE}, by putting $u(r)=\HM^2(E\cap
	B(x,r))$, we get that, if $x\in L$, 
	\begin{equation}\label{eq:ANDD2}
		u(r)\leq \frac{r}{2}u'(r)+h(2r)(2r)^2,
	\end{equation}
	for almost every $r\in (0,\infty)$; if $x\not \in L$, then \eqref{eq:ANDD2}
	holds for almost every $r\in (0,\min\{r_h,\dist(x,L)\})$. 

	We put $v(r)=r^{-2}u(r)$, then $v'(r)\geq -8r^{-2}h(2r)$, we get that 
	$	\Theta(x,r)+8 h_1(r)$ is nondecreasing.
\end{proof}
\begin{remark}
	Let $E$ be a 2-rectifiable $(\Omega_0,L_0,h)$ locally sliding almost 
	minimal at some point $x\in E$. Then by Theorem \ref{thm:ANDD}, we get
	that $\Theta_E(x)$ exists. 
\end{remark}
\section{Estimation of upper bound}
Let $\mathcal{Z}$ be a collection of cones. We say that a set $E\subseteq \mathbb{R}^{3}$ is locally $C^{k,\alpha}$-equivalent (resp. $C^{k}$-equivalent)
to a cone in $\mathcal{Z}$ at $x\in E$ for some nonnegative integer $k$ and 
some number $\alpha\in (0,1]$, if there exist $\varrho_0>0$ and $\tau_0>0$ 
such that for any $\tau\in (0,\tau_0)$ there is $\varrho\in (0,\varrho_0)$, a
cone $Z\in \mathcal{Z}$ and a mapping $\Phi:B(0,2\varrho)\to\mathbb{R}^{3}$,
which is a homeomorphism of class $C^{k,\alpha}$ (resp. $C^k$) between 
$B(0,2\varrho)$ and its image $\Phi(B(0,2\varrho))$ with $\Phi(0)=x$, satisfying that
\begin{equation}\label{eq:apps}
\| \Phi-\id-\Phi(0) \|_{\infty}\leq \varrho\tau
\end{equation}
and
\begin{equation}\label{eq:imcone}
	E\cap B(x,\varrho)\subseteq \Phi\left(Z\cap B\left( 0,2\varrho
	\right)\right)\subseteq E\cap B(x,3\varrho).
\end{equation}
Similarly, if $\Omega\subseteq\mathbb{R}^{3}$ is a closed set with the boundary
$\partial\Omega$ is a 2-dimensional manifold, a set $E\subseteq \Omega$
is called locally $C^{k,\alpha}$-equivalent to a sliding minimal cone $Z$ in
$\Omega_0$ at $x\in E\cap \partial\Omega$, if there exist $\varrho_0>0$
and $\tau_0>0$ such that for any $\tau\in (0,\tau_0)$ there is $\varrho
\in (0,\varrho_0)$ and a mapping $\Phi: B(0,2\varrho)\cap\Omega_0\to\Omega$, 
which is a diffeomorphism of class $C^{k,\alpha}$ between its domain and image
with $\Phi(0)=x$ satisfying that $\Phi(L_0\cap B(0,2\varrho))\subseteq 
\partial\Omega$ and \eqref{eq:apps} and \eqref{eq:imcone}. 

Suppose that $\Omega\subseteq \mathbb{R}^3$ is closed set with the boundary
$\partial \Omega$ is a $2$-dimensional $C^1$ manifold. Suppose that $E\subseteq 
\Omega$ is sliding almost minimal with sliding boundary $\partial \Omega$ and
gauge function $h$. Then, by putting $U=\Omega\setminus \partial \Omega$, we 
see that $E\cap U$ is almost minimal in $U$, applying Jean Taylor's theorem, 
$E$ is locally $C^{1,\beta}$-equivalent to a minimal cone at each point $x\in
E\cap U$ for some $\beta>0$ in case $h(r)\leq cr^{\alpha}$ for some $c>0$,
$\alpha>0$, $r_0>0$ and $0<r<r_0$.  We see from \cite[Theorem 6.1]{Fang:2015} 
that, at $x\in E\cap \partial\Omega$, $E$ is locally $C^{0,\beta}$-equivalent 
to a sliding minimal cone in $\Omega_0$ in case the gauge function $h$ 
satisfying \eqref{eq:AG}.

\subsection{Approximation of $E\cap\partial B(0,r)$ by rectifiable curves}
For any sets $X,Y\subseteq\mathbb{R}^{3}$, any $z\in\mathbb{R}^{3}$ and any
$r>0$, we denote by $d_{z,r}$ the normalized local Hausdorff distance defined
by
\[
	d_{z,r}(X,Y)=\frac{1}{r}\sup\{\dist(x,Y):x\in X\cap B(z,r)\}+
	\frac{1}{r}\sup\{\dist(y,X):y\in Y\cap B(z,r)\}.
\]
A cone in $\mathbb{R}^{3}$ is called of type $\mathbb{Y}$ if it is the union 
of three half planes with common boundary line and that make $120^{\circ}$ 
angles along the boundary line.
A cone $Z\subseteq \Omega_0$ is called of type $\mathbb{P}_+$ is if it is a half
plane perpendicular to $L_0$; a cone $Z\subseteq \Omega_0$ is called of type
$\mathbb{Y}_+$ is if $Z=\Omega_0\cap Y$, where $Y$ is a cone of type 
$\mathbb{Y}$ perpendicular to $L_0$; for convenient, we will also use the 
notation $\mathbb{P}_+$, to denote the collection of all of cones of
type $\mathbb{P}_+$, and $\mathbb{Y}_{+}$ to denote the collection of all
of cones of type $\mathbb{Y}_{+}$.

For any set $E\subseteq\Omega_0$ with $0\in E$, and any $r>0$, we set 
\begin{equation}\label{eq:ePY}
	\begin{gathered}
		\varepsilon_{P}(r)=\inf\{ d_{0,r}(E,Z): Z \in \mathbb{P}_{+}\},\\
		\varepsilon_{Y}(r)=\inf\{ d_{0,r}(E,Z): Z \in \mathbb{Y}_{+}\}.
	\end{gathered}
\end{equation}
If $E$ is 2-rectifiable and $\HM^2(E)<\infty$, then $E\cap \partial B(0,r)$ is
1-rectifiable and $\HM^1(E\cap \partial B(0,r))<\infty$ 
for $\HM^1$-a.e. $r\in (0,\infty)$; we consider the function $u:(0,\infty)\to
\mathbb{R}$ which is defined by $u(r)=\HM^2(E\cap B(0,r))$, it is quite easy
to see that $u$ is nondecreasing, thus $u$ is differentiable for $\HM^1$-a.e.; 
we will denote by $\mathscr{R}$ the set $r\in (0,\infty)$ such that
\begin{equation}\label{eq:R1}
	\HM^1(E\cap \partial B(0,r))<\infty,\ u \text{ is differentiable at }r,
\end{equation}
\begin{equation}\label{eq:R2}
	\lim_{\xi\to 0+}\frac{1}{\xi}\int_{t\in (r-\xi,r)}\int_{E\cap \partial
	B(0,t)}f(z)d\HM^1(z)dt=\int_{E\cap \partial B(0,r)}f(z)d\HM^1(z),		
\end{equation}
and 
\begin{equation}\label{eq:R3}
	\sup_{\xi>0}\frac{1}{\xi}\int_{t\in (r-\xi,r)}\HM^1(E\cap \partial
	B(0,t))dt<+\infty.
\end{equation}
It is not hard to see that $\HM^1((0,\infty)\setminus \mathscr{R})=0$, see for
example Lemma 4.12 in \cite{David:2008}.
\begin{lemma}\label{le:FHMCIPC}
	Let $E\subseteq\mathbb{R}^{3}$ be a connected set. If $\HM^1(E)<\infty$, then
	$E$ is path connected.
\end{lemma}
For a proof, see for example Lemma 3.12 in \cite{Falconer:1986}, so we omit it
here.
\begin{lemma}\label{le:CoSe}
	Let $\mathbb{X}$ be a locally connected and simply connected compact metric space.
	Let $A$ and $B$ be two connected subsets of $\mathbb{X}$. If $F$ is a closed
	subset of $\mathbb{X}$ such that $A$ and $B$ are contained in two different
	connected components of $\mathbb{X}\setminus F$, then there exists a 
	connected closed set $F_0\subseteq F$ such that $A$ and $B$ still lie in two
	different connected components of $\mathbb{X}\setminus F_0$. 
\end{lemma}
\begin{proof}
	See for example 52.\uppercase\expandafter{\romannumeral3}.1 on page 335 in
	\cite{Kuratowski:1992}, so we omit the proof here.
\end{proof}
For any $r>0$, we put $\mathfrak{Z}_r=(0,0,r)\in \mathbb{R}^{3}$.
\begin{lemma}\label{le:1CuSp}
	Let $E\subseteq \Omega_0$ be a $2$-rectifiable set with $\HM^2(E)<\infty$. 
	Suppose that $0\in E$, and that $E$ is locally $C^0$-equivalent to a sliding
	minimal cone of type $\mathbb{P}_+$ at $0$. Then for any $\tau\in (0,\tau_0)$
	there exist $\mathfrak{r}=\mathfrak{r}(\tau)>0$ such that, for any $r\in 
	(0,\mathfrak{r})$ and $\varepsilon>\varepsilon_{P}(r)$, we can find
	$y_r\in E\cap \partial B(0,r)\setminus L$ , $\mathfrak{X}_{r,1},
	\mathfrak{X}_{r,2}\in E\cap L\cap \partial B(0,r)$  and two simple curves
	$\gamma_{r,1},\gamma_{r,2}\subseteq E\cap \partial B(0,r)$ satisfying that 
	\begin{enumerate}[nolistsep,label=\emph{(\arabic*)},leftmargin=2\parindent]
		\item $|y_r-\mathfrak{Z}_r|\leq \varepsilon r$ and $|z_{r,1}-z_{r,2}|\geq
			(2-2\varepsilon)r$;
		\item $\gamma_{r,i}$ joins $y_r$ and $\mathfrak{X}_{r,i}$, $i=1,2$;
		\item $\gamma_{r,1}$ and $\gamma_{r,2}$ are disjoint except for point
			$y_r$.
	\end{enumerate}
\end{lemma}
\begin{proof}
	Since $E$ is locally $C^0$-equivalent to a sliding minimal cone of type
	$\mathbb{P}_+$ at $0$, for any $\tau\in (0,\tau_0)$, there exist 
	$\varrho>0$, sliding minimal cone $Z$ of type $\mathbb{P}_+$, and a mapping
	$\Phi:\Omega_0\cap B(0,2\varrho)\to \Omega_0$ which is a homeomorphism
	between $\Omega_0\cap B(0,2\varrho)$ and $\Phi(\Omega_0\cap B(0,2\varrho))$
	with $\Phi(0)=0$ and $\Phi(\partial\Omega_0\cap B(0,2\varrho))\subseteq
	\partial\Omega_0$ such that \eqref{eq:apps} and \eqref{eq:imcone} hold.
	We new take $\mathfrak{r}=\varrho$. Then for any $r\in (0,\mathfrak{r})$, 
	\[
		\Phi^{-1}\left[ E\cap \partial B(0,r) \right]\subseteq Z \cap B(0,3\varrho).
	\]
	Without loss of generality, we
	assume that $Z=\{(x_1,0,x_3)\mid x_1\in \mathbb{R}, x_3\geq 0 \}$. 
	Applying Lemma \ref{le:CoSe} with $\mathbb{X}=Z\cap
	\overline{B(0,3\varrho)}$, $F=\Phi^{-1}\left[ E\cap \partial B(0,r)\right]$,
	$A=\{ 0 \}$ and $B=Z\cap \partial B(0,3\varrho)$, we get that there is 
	a connected closed set $F_0\subseteq F$ such that $A$ and $B$ lie in two 
	different connected components of $A\setminus F_0$, thus $\phi(F_0)\subseteq 
	E\cap \partial B(0,r)$ is connected. We put $a_1=\{ (x_1,0,0)\mid x_1< 0 \}$ 
	and $a_2=\{ (x_1,0,0)\mid x_1> 0 \}$. Then $F_0\cap a_i\neq \emptyset$, $i=1,2$;
	otherwise $A$ and $B$ are contained in a same connected component of
	$X\setminus F_0$. We take $z_{r,i}\in F_0\cap a_i$, and let
	$\mathfrak{X}_{r,i}=\phi(z_{r,i})\in E\cap \partial B(0,r)$. Then 
	$|\mathfrak{X}_{r,1}-\mathfrak{X}_{r,2}|\geq (2-2\varepsilon)r$.

	Since $F_0$ is connected and $\HM^1(F_0)<\infty$, by Lemma \ref{le:FHMCIPC},
	$F_0$ is path connected. Let $\gamma$ be a simple curve which joins $z_{r,1}$ 
	and $z_{r,2}$. We see that ${B(\mathfrak{Z}_r,\varepsilon r)}\cap
	\gamma\neq \emptyset$, because $\varepsilon_P(r)<\varepsilon$ and
	$\mathfrak{Z}_r\in Z$ for 
	sliding minimal cone $Z$ of type $\mathbb{P}_{+}$. We take 
	$y_r\in {B(\mathfrak{Z}_r,\varepsilon r)}\cap \gamma$.
\end{proof}
\begin{lemma}\label{le:3CuSp}
	Let $E\subseteq \Omega_0$ be a $2$-rectifiable set with $\HM^2(E)<\infty$. 
	Suppose that $0\in E$, and that $E$ is locally $C^0$-equivalent to a sliding
	minimal cone of type $\mathbb{Y}_+$ at $0$. Then for any $\tau\in (0,\tau_0)$
	there exist $\mathfrak{r}=\mathfrak{r}(\tau)>0$ such that, for any $r\in 
	(0,\mathfrak{r})$ and $\varepsilon>\varepsilon_{Y}(r)$, we can find
	$y_r\in E\cap \partial B(0,r)\setminus L$ , 
	$\mathfrak{X}_{r,1},\mathfrak{X}_{r,2},\mathfrak{X}_{r,3}\in E\cap L\cap 
	\partial B(0,r)$ and three simple curves $\gamma_{r,1},\gamma_{r,2},
	\gamma_{r,3}\subseteq E\cap \partial B(0,r)$ satisfying that
	\begin{enumerate}[nolistsep,label=\emph{(\arabic*)},leftmargin=2\parindent]
		\item $|\mathfrak{Z}_r-y_r|\leq \pi r/6$, and there exists 
			$Z\in \mathbb{Y}_{+}$ with $\dist(x,Z)\leq \varepsilon r$ for $x\in \gamma$;
		\item $\gamma_{r,i}$ join $y_r$ and $\mathfrak{X}_{r,i}$;
		\item $\gamma_{r,i}$ and $\gamma_{r,j}$ are disjoint except for point
			$y_r$.
	\end{enumerate}
\end{lemma}
\begin{proof}
	Since $E$ is locally $C^0$-equivalent to a sliding minimal cone of type
	$\mathbb{Y}_+$ at $0$, for any $\tau\in (0,\tau_0)$, there exist $\tau>0$,
	$\varrho>0$, sliding minimal cone $Z$ of type $\mathbb{Y}_+$, and a mapping
	$\Phi:\Omega_0\cap B(0,2\varrho)\to \Omega_0$ which is a homeomorphism
	between $\Omega_0\cap B(0,2\varrho)$ and $\Phi(\Omega_0\cap B(0,2\varrho))$
	with $\Phi(0)=0$ and $\Phi(\partial\Omega_0\cap B(0,2\varrho))\subseteq
	\partial\Omega_0$ such that \eqref{eq:apps} and \eqref{eq:imcone} hold.
	We new take $\mathfrak{r}=\varrho$. Then for any $r\in (0,\mathfrak{r})$, 
	\[
		\Phi^{-1}\left[ E\cap \partial B(0,r) \right]\subseteq Z \cap B(0,3\varrho).
	\]
	Applying Lemma \ref{le:CoSe} with $\mathbb{X}=Z\cap
	\overline{B(0,3\varrho)}$, $F=\Phi^{-1}\left[ E\cap \partial B(0,r)\right]$,
	$A=\{ 0 \}$ and $B=Z\cap \partial B(0,3\varrho)$, we get that there is 
	a connected closed set $F_0\subseteq F$ such that $A$ and $B$ lie in two 
	different connected components of $A\setminus F_0$, thus $\phi(F_0)\subseteq 
	E\cap \partial B(0,r)$ is connected. We let $a_i$, $i=1,2,3$, be the there
	component of $Z\cap L_0\setminus A$. Then $F_0\cap a_i\neq \emptyset$,
	$i=1,2,3$; otherwise $A$ and $B$ are contained in a same connected component
	of $X\setminus F_0$. We take $z_{r,i}\in F_0\cap a_i$, and let
	$\mathfrak{X}_{r,i}=\phi(z_{r,i})\in E\cap \partial B(0,r)$. Then 
	$|\mathfrak{X}_{r,1}-\mathfrak{X}_{r,2}|\geq (\sqrt{3}-2\varepsilon)r$.

	Since $F_0$ is connected and $\HM^1(F_0)<\infty$, by Lemma \ref{le:FHMCIPC},
	$F_0$ is path connected. 
\end{proof}

\subsection{Approximation of rectifiable curves in $\mathbb{S}^2$ by Lipschitz graph}
We denote by $\mathbb{S}^2$ the unit sphere in $\mathbb{R}^{3}$. We say that a
simple rectifiable curve $\gamma\subseteq\mathbb{S}^2$ is a Lipschitz graph with
constant at most $\eta$, if it can be parametrized by 
\[
	z(t)=\left(\sqrt{1-v(t)^2}\cos\theta(t),\sqrt{1-v(t)^2}\sin\theta(t),v(t)\right),
\]
where $v$ is Lipschitz with $\Lip(v)\leq \eta$.

\begin{lemma}
	\label{le:smallh}
	Let $T\in [\pi/3,2\pi/3]$ be a number, and $\gamma:[0,T]\to \mathbb{S}^2$ a
	simple rectifiable curve given by  
	\[
		\gamma(t)=\left(\sqrt{1-v(t)^2}\cos\theta(t),\sqrt{1-v(t)^2}\sin\theta(t),
		v(t)\right),
	\]
	where $v$ is a continuous function with $v(0)=v(T)=0$, $\theta$ is a
	continuous function with $\theta(0)=0$ and $\theta(T)=T$. Then there is a
	small number $\tau_0\in (0,1)$ such that whenever $|v(t)|\leq \tau_0$, we 
	have that
	\[
		|v(t)|\leq 10\sqrt{\HM^1(\gamma)-T}.
	\]
\end{lemma}
\begin{proof}
	We let $A=\gamma(0)=(1,0,0)$, $B=\gamma(T)=(\cos T,\sin T,0)$,
	and let $C=\gamma(t_0)$ be a point in $\gamma$ such that
	\[
		|v(t_0)|=\max\{ |v(t)|: t\in [0,T] \}.
	\]
	We let $\gamma_i$, $i=1,2$, be two curve such that $\gamma_1(0)=A$,
	$\gamma_1(1)=C$, $\gamma_2(0)=B$ and $\gamma_2(1)=C$, and let $s\in [0,1]$
	be the smallest number such that $\gamma_1(s)\not\in \gamma_2$, and put
	$D=\gamma_1(s)$. Then, by setting $\mathfrak{C}_1$, $\mathfrak{C}_2$ and
	$\mathfrak{C}_3$ the arc $AD$, $BD$ and $CD$ respectively, we have that 
	\[
		\HM^1(\gamma)\geq \HM^1(\gamma_1\cup \gamma_2)\geq
		\HM^1(\mathfrak{C}_1)+\HM^1(\mathfrak{C}_2)+\HM^1(\mathfrak{C}_3).
	\]

	We see that $\mathfrak{C}_1\cup \mathfrak{C}_2$ is a simple Lipschitz curve
	joining $A$ and $B$, and let $\gamma_3:[0,\ell]\to \mathbb{S}^2$ giving by 
	\[
		\gamma_3(t)=\left( \sqrt{1-w(t)^2}\cos \theta(t),
		\sqrt{1-w(t)^2}\sin\theta(t),
		w(t)\right)
	\]
	be its parametrization by length. 
	We assume that $\gamma_3(t_1)=D$, then $w'(t)>0$ on $(0,t_1)$, or $w'(t)<0$
	on $(0,t_1)$, thus $|w(t)|=\int_0^{t_1}|w'(t)|d t$.

	We let the number $\tau_0\in (0,1)$ to be the small number $\tau_1$ in Lemma
	7.8 in \cite{David:2008}. If $\HM^1(\gamma)-T\leq \tau_0$, then we have that  
	\[
		\int_0^{\ell} |w'(t)|^2dt\leq 14 (\ell-T),
	\]
	thus 
	\[
		|w(t_1)|=\int_0^{t_1}|w'(t)|d t\leq \left( t_1 \int_0^{t_1}
		|w'(t)|^2dt\right)^{1/2}\leq \sqrt{14\ell(\ell-T)}.
	\]
	We get so that 
	\[
		\begin{aligned}
			|v(t_0)|&\leq \HM^1(\mathfrak{C}_3)+|w(t_1)|\leq
			(\HM^1(\gamma)-\ell)+\sqrt{14\ell(\ell-T)}\\
			&\leq \sqrt{14\HM^1(\gamma)(\HM^1(\gamma)-T)}\leq 10 \sqrt{\HM^1(\gamma)-T}.
		\end{aligned}
	\]

	If $\HM^1(\gamma)-T> \tau_0$, then $v(t)\leq \tau_0\leq 10\sqrt{\tau_0}\leq
	10 \sqrt{\HM^1(\gamma)-T}$.
\end{proof}

\begin{lemma}\label{le:SmCuAp}
	Let $a$ and $b$ be two points in $\Omega_0\cap \partial B(0,1)$ satisfying 
	\begin{equation}\label{eq:SmCu1}
		\frac{\pi}{3}\leq\dist_{\mathbb{S}^2}(a,b)\leq \frac{2\pi}{3}.
	\end{equation}
	Let $\gamma$ be a simple rectifiable curve in $\Omega_0\cap
	\partial B(0,1)$ which joins $a$ and $b$, and satisfies 
	\begin{equation}\label{eq:SmCu2}
		\length(\gamma)\leq \dist_{\mathbb{S}^2}(a,b)+\tau_0,
	\end{equation}
	where $\tau_0>0$ is as in Lemma \ref{le:smallh}.
	Then there is a constant $C>0$ such that, for any $\eta>0$, we can find a 
	simple curve $\gamma_{\ast}$ in $\Omega_0\cap
	\partial B(0,1)$ which is a Lipschitz graph with constant at most $\eta$
	joining $a$ and $b$, and satisfies that 
	\[
		\HM^1(\gamma_{\ast}\setminus \gamma)\leq
		\HM^1(\gamma\setminus\gamma_{\ast})\leq
		C\eta^{-2}(\length(\gamma)-\dist_{\mathbb{S}^2}(a,b)).
	\]
\end{lemma}
The proof will be the same as in \cite[p.875-p.878]{David:2008}, so we omit
it.
\subsection{Compare surfaces}
Let $\Gamma$ be a Lipschitz curve in $\mathbb{S}^2$. We
assume for simplicity that its extremities $a$ and $b$ lie in the horizontal
plane. Let us assume that $a=(1,0,0)$ and $b=(\cos T, \sin T, 0)$ for some
$T\in [\pi/3,2\pi/3]$. We also assume that $\Gamma$ is a Lipschitz graph with
constant at most $\eta$, i.e. there is a Lipschitz
function $s:[0,T]\to\mathbb{R}$ with $s(0)=s(T)=0$ and $\Lip(s)\leq \eta$,
such that $\Gamma$ is parametrized by 
\[
	z(t)=(w(t)\cos t,w(t)\sin t, s(t)) \text{ for }t\in [0,T],
\]
where $w(t)=(1-|s(t)|^2)^{1/2}$.

We set 
\[
	D_T=\{ (r\cos t, r\sin t)|\mid 0< r <1 ,0 <t < T\},
\]
and consider the function $v:\overline{D}_T\to \mathbb{R}$ defined by 
\[
	v(r\cos t,r\sin t)=\frac{rs(t)}{w(t)} \text{ for } 0\leq r\leq 1 \text{ and
	} 0\leq t\leq T.
\]

For any function $f:\overline{D}_T\to\mathbb{R}$, we denote by $\Sigma_f$ 
the graphs of $f$ over $\overline{D}_T$. 

\begin{lemma}\label{le:CoSu}
	There is a universal constant $\kappa>0$ such that we can find a Lipschitz 
	function $u$ on $\overline{D}_T$ satisfying that 
	\[
		\Lip(u)\leq C\eta,
	\]
	\[
		u(r,0)=u(r\cos T,r\sin T)=0, \text{ for  }0\leq r\leq 1, 0\leq t\leq T,
	\]
	\[
		u(r\cos t,r\sin t)=v(r\cos t,r\sin t)\text{ for  }0\leq r\leq 1, 0\leq t\leq T,
	\]
	\[
		u(r\cos t,r\sin t)=0,\text{ for }0\leq r\leq 2\kappa, 0\leq t\leq T
	\]
	and 
	\[
		\HM^2(\Sigma_v)-\HM^2(\Sigma_u)\geq 10^{-4}(\HM^1(\Gamma)-T).
	\]
\end{lemma}
\begin{proof}
	The proof is the same as Lemma 8.8 in \cite{David:2008}, we omit it.
\end{proof}

\subsection{Retractions}
We let $\Pi:\mathbb{R}^{3}\setminus \{ 0 \}
\to \mathbb{S}^2$ be the projection defined by $\Pi(x)=x/|x|$.
In this subsection, we assume that $E\subseteq \Omega_0$ is a $2$-rectifiable
set satisfying that  
\begin{enumerate}
		[topsep=0pt,itemsep=0pt,leftmargin=2\parindent,label=\textbf{(\alph*)}]
	\item \label{AS1} $\HM^2(E)<\infty$, $0\in E$, 
	\item \label{AS2} $E$ is locally $(\Omega_0,L_0,h)$ sliding almost minimal
		at $0$,
	\item \label{AS3} $E$ is locally $C^0$-equivalent to a sliding 
		minimal cone of type $\mathbb{P}_+$ or $\mathbb{Y}_+$.
\end{enumerate}
For convenient, we put 
\[
	j(r)=\frac{1}{r}\HM^1(E\cap \partial B(0,r))-\HM^1(X\cap \partial B(0,1)),
\]
and denote by $\mathscr{R}_1$ the set $\{ r\in \mathscr{R}:j(r) \leq
\tau_1\}$, where $\tau_1$ is the small number considered as in Lemma
\ref{le:smallh}.

For any $r\in (0,\mathfrak{r})\cap \mathscr{R}_1$, we take $\mathcal{X}_r
\subseteq E\cap B(0,r)\cap L_0$ as following:
if $E$ is locally $C^0$-equivalent to a sliding minimal cone
of type $\mathbb{P}_+$, we let $\mathfrak{X}_{r,1}$ and $\mathfrak{X}_{r,2}$ 
be the same as in Lemma \ref{le:1CuSp}, and let $\mathcal{X}_r=\{
\mathfrak{X}_{r,1},\mathfrak{X}_{r,2} \}$; if
$E$ is locally $C^0$-equivalent to a sliding minimal cone
of type $\mathbb{Y}_+$, we let $\mathfrak{X}_{r,1}$, $\mathfrak{X}_{r,2}$ and
$\mathfrak{X}_{r,3}$ be the same as in
Lemma \ref{le:3CuSp}, and let $\mathcal{X}_r=\{
\mathfrak{X}_{r,1},\mathfrak{X}_{r,2},\mathfrak{X}_{r,3} \}$.

We take $y_r$ as in Lemma \ref{le:1CuSp} or Lemma \ref{le:3CuSp}. For any 
$x\in \mathcal{X}_r$, we let $\gamma^x$ be the curve which joins $x$
and $y_r$ as in Lemma \ref{le:1CuSp} and Lemma \ref{le:3CuSp}, let
$D_{x,y_r}$ be the sector determined by points $0$, $y_r$ and $x$.
We denote by $P_{x,y_r}$ the plane that contains $0$, $x$ and $y_r$, let
$\mathcal{R}_{x,y_r}$ be a rotation such that
$\mathcal{R}_{x,y_r}(y_r)=(r,0,0)$ and $\mathcal{R}_{x}(y_r)=(r\cos T_x,r\sin
T_x,0)$, where $T_x\in [\pi/3,2\pi/3]$.

For any $x\in \mathcal{X}_r$, $\gamma^x$ is a simple rectifiable curve in
$\Omega_0\cap\partial B(0,r)$, thus the curve $\Gamma^x=\Pi(\gamma^x)$ is a
simple rectifiable curve in 
$\Omega_0\cap \partial B(0,1)$, let $\Gamma^x_{\ast}$ be the corresponding curve 
with respect to $\Gamma^{x}$ as in Lemma \ref{le:SmCuAp}. Let 
$z(t)=(w(t)\cos t, w(t)\sin t, s(t))$ be a parametrization of
$\mathcal{R}_{x,y_r}(\Gamma_{\ast}^{x})$, where
$w(t)=\sqrt{1-s(t)^2}$. Let $\Sigma^x_v$ and $\Sigma^x_u$ be the same as in Lemma 
\ref{le:CoSu}. We put
$T=\sum_{x\in \mathcal{X}_r}T_x$, and put 
\begin{equation}\label{eq:CoLip}
	X=\bigcup_{x\in \mathcal{X}_r}D_{x,y_r},\ 
	\Gamma_{\ast}=\bigcup_{x\in \mathcal{X}_r}\Gamma^x_{\ast},\
	\mathcal{M}=\bigcup_{x\in \mathcal{X}_r}\Sigma^x_v,\text{ and }
	\Sigma=\bigcup_{x\in \mathcal{X}_r}\Sigma^x_u.
\end{equation}

By Lemma \ref{le:CoSu}, we have that 
\begin{equation}\label{eq:CoToSu}
	\HM^2(\mathcal{M})-\HM^2(\Sigma)\geq 10^{-4}\left(
	\HM^1(\Gamma_{\ast})-T\right),
\end{equation}
and by Lemma \ref{le:smallh}, we have that 
\begin{equation} \label{eq:XM}
	d_{0,1}(X,\mathcal{M})\leq 10 j(r)^{1/2}.
\end{equation}
\begin{lemma} \label{le:EM}
	If $\varepsilon(r)<1/2$, then for any $\varepsilon(r)<\varepsilon <1/2$, 
	there is a sliding minimal cone $Z=Z_r$ such that 
	\[
		d_{0,1}(X,Z)\leq 4 \varepsilon.
	\]
	Moreover 
	\[
		d_{0,r}(E,X)\leq 5 \varepsilon(r).
	\]
\end{lemma}
\begin{proof}
	There exists sliding minimal cone $Z$ such that $d_{0,r}(E,Z)\leq
	\varepsilon$, thus for any $x\in \mathcal{X}_r$,
	there is $x_z\in Z\cap (L_0\cap \partial B_r)$ satisfying that
	$|x-x_z|\leq 2 \varepsilon r$. We get so that  
	\[
		d_H([x,y_r],[x_z,\mathfrak{Z}_r])\leq 2 \varepsilon r.
	\]
	Since $\dist(0,[x,y_r])>r/2$ for any $x\in \mathcal{X}_r$, we have that
	\[
		d_H(X\cap B(0,r/2) ,Z\cap B(0,r/2))\leq 2 \varepsilon r.
	\]
	Thus
	\[
		d_{0,1}(X,Z)=d_{0,r/2}(X,Z)\leq 4 \varepsilon,
	\]
	and 
	\[
		d_{0,r}(E,X)\leq d_{0,r}(E,Z)+d_{0,r}(Z,X)\leq 5 \varepsilon.
	\]
\end{proof}
\begin{lemma}\label{le:3vect}
	Let $0<\delta,\varepsilon< 1/2$ be positive numbers. Let $v_1,v_2,v_3\in
	\mathbb{R}^{3}$ be three unit vectors.  
	\begin{itemize}[nolistsep]
		\item If $ |\langle v_2,v_i \rangle|\leq \delta$ for $i=1,3$, then for any
			$v\in\mathbb{R}^{3}$ with $\langle v,v_2 \rangle=0$ and
			$\dist(v,\vect{v_1,v_2})\leq \varepsilon |v|$, we have that 
			\begin{equation}\label{eq:3vect1}
				|\langle v,v_3 \rangle-\langle v_1,v_3 \rangle\langle v,v_2 \rangle|\leq
				(\varepsilon+\delta)|v|, \text{ and }|\langle v,v_1 \rangle|\geq
				(1-\varepsilon-\delta)|v|.
			\end{equation}
		\item If $\langle v_1,v_3
			\rangle<1$ and $0<\delta<10^{-2}(1-\langle v_1,v_3 \rangle)^2$, then for
			any $w_1,w_3\in \mathbb{R}^{3}$ with $\langle v_i,w_i \rangle
			\geq (1-\delta) |w_i| $, $i=1,3$, we have that 
			\begin{equation}\label{eq:3vect2}
				|w_1|+|w_3|\leq \sqrt{2}\cdot \left(1-\langle v_1,v_3
				\rangle-4\sqrt{2\delta}\right)^{-1/2}|w_1-w_3|.
			\end{equation}
	\end{itemize}
\end{lemma}
\begin{proof}
	We write $v=v^{\perp}+\lambda_1 v_1+\lambda_2 v_2$, $\lambda_i\in
	\mathbb{R}$, $\langle v^{\perp},v_i \rangle=0$. 
	Since  $\langle v,v_2 \rangle=0$, we have that $\lambda_2=-\lambda_1\langle v_1,v_2 \rangle$, thus
	\[
		\lambda_1=\frac{\langle v,v_1 \rangle}{1-\langle v_1,v_2 \rangle^2},\
		\lambda_2=-\frac{\langle v,v_1 \rangle\langle v_1,v_2 \rangle}{1-\langle
		v_1,v_2 \rangle^2},
	\]
	we get so that 
	\begin{equation}\label{eq:3vectdecom}
		v=v^{\perp}+\frac{\langle v,v_1 \rangle v_1- \langle v,v_1 \rangle\langle
		v_1,v_2 \rangle v_2}{1-\langle v_1,v_2 \rangle^2},
	\end{equation}
	and then
	\[
		\langle v,v_3 \rangle=\langle v^{\perp},v_3 \rangle+\frac{\langle v_1,v_3
		\rangle-\langle v_2,v_3 \rangle\langle v_1,v_2 \rangle}{1-\langle v_1.v_2
		\rangle^2}\langle v,v_1 \rangle, 
	\]
	thus
	\[
		|\langle v,v_3 \rangle-\langle v_1,v_3 \rangle\langle v,v_1 \rangle|\leq
		\varepsilon |v|+\frac{\delta^2+\delta}{1-\delta^2}|v|\leq
		(\varepsilon+2\delta)|v|.
	\]
	We get also, from \eqref{eq:3vectdecom}, that 
	\[
		|v|\leq |v^{\perp}|+\frac{1+|\langle v_1.v_2 \rangle|}{1-\langle v_1,v_2
		\rangle^2}|\langle v,v_1 \rangle|\leq \varepsilon
		|v|+\frac{1}{1-\delta}|\langle v,v_1 \rangle|,
	\]
	thus 
	\[
		|\langle v,v_1 \rangle|\geq (1-\varepsilon)(1-\delta) |v|\geq
		(1-\varepsilon-\delta)|v|.
	\]

	We can certainly assume $w_i\neq 0$, otherwise the inequality
	\eqref{eq:3vect2} will be trivial true. Since $\langle v_i,w_i \rangle 
	\geq (1-\delta)|w_i|$, we have that $\langle v_i,w_i/|w_i| \rangle \geq
	1-\delta$, and
	\[
		\Big|v_i-w_i/|w_i|\Big|^2=2-2 \langle v_i,w_i/|w_i| \rangle \leq 2 \delta.
	\]
	Thus
	\[
		\begin{aligned}
			\left| \frac{w_1}{|w_1|}-\frac{w_2}{|w_2|}\right|^2&=\left|
			\left(\frac{w_1}{|w_1|}-v_1\right)-\left(\frac{w_2}{|w_2|}-v_2\right)+
			(v_1-v_2)\right|^2\\
			&\geq |v_1-v_2|^2-2|v_1-v_2|\left(\left|\frac{w_1}{|w_1|}-v_1\right|+
			\left|\frac{w_2}{|w_2|}-v_2\right|\right)\\
			&\geq 2-2 \langle  v_1,v_2\rangle-8\sqrt{2 \delta},
		\end{aligned}
	\]
	and 
	\[
		\langle w_1,w_2 \rangle
		=|w_1||w_2|\left\langle \frac{w_1}{|w_1|},\frac{w_2}{|w_2|}
		\right\rangle\leq
		|w_1||w_2| \left(\langle  v_1,v_2\rangle+4\sqrt{2 \delta}\right).
	\]
	Hence 
	\[
		|w_1-w_2|^2\geq |w_1|^2+|w_2|^2-2|w_1||w_2|\left(\langle
		v_1,v_2\rangle+4\sqrt{2 \delta}\right)\geq (1-s)(|w_1|+|w_2|)^2,
	\]
	where $s=(1+\langle v_1,v_2 \rangle+ 4\sqrt{2 \delta})/2\in (0,1)$.
\end{proof}
\begin{lemma} \label{le:BLNR}
	For any $r\in (0,\mathfrak{r})\cap \mathscr{R}_1$, we let $\Sigma$ be
	as in \eqref{eq:CoLip}. Then there is a Lipschitz mapping $p:\Omega_0\to 
	\Sigma$ with $\Lip(p)\leq 50$, such that $p(z)\in L$ for $z\in L$, and that 
	$p(z)=z$ for $z\in\Sigma$.
\end{lemma}
\begin{proof}
	By definition, we have that 
	\begin{equation}\label{eq:BLNR1}
		\Sigma\setminus B\left( 0,9/10 \right)=\mathcal{M}\setminus B\left(
		0,9/10 \right),
	\end{equation}
	and that 	
	\begin{equation}\label{eq:BLNR2}
		\Sigma\cap B(0,2\kappa)=X\cap B(0,2\kappa).
	\end{equation}
	For any $z\in\Omega_0\setminus \{0\}$, we denote by $\ell(z)$ the 
	line which is through $0$ and $z$. Then $\partial
	D_{x,y_r}=\ell(x)\cup \ell(y_r)$. 
	We fix any $\sigma\in (0,10^{-3})$, put 
	\[
		\begin{gathered}
			R^x=\{ z\in \Omega_0\mid \dist(z,D_{x,y_r})\leq
			\sigma\dist(z,\partial D_{x,y_r}) \}, \\
			R_1^x=\{ z\in \Omega_0\mid \dist(z,D_{x,y_r})\leq
			\sigma\dist(z,\ell(y_r)) \},
		\end{gathered}
	\]
	and 
	\[
		R=\bigcup_{x\in \mathcal{X}_r}R^x,
		R_1=\bigcup_{x\in \mathcal{X}_r}R_1^x.
	\]
	Then we see that $R^x\subseteq R_1^x$, and that both of them are cones,  
	\[
		R^{x_i}\cap R^{x_j}=R_1^{x_i}\cap R_1^{x_j}=\ell(y_r) \text{ for } 
		x_i,x_j\in \mathcal{X}_r, x_i\neq x_j.
	\]

	Since $\Sigma^x_{u}$ is a small Lipschitz graph over $D_{x,y_r}$ bounded by 
	two half lines of $\partial D_{x,y_r}$ with constant at most $\eta$,
	there is a constant $\bar{\eta}$ such that 
	\[
		\Sigma^{x}_{u}\subseteq R^x,
	\]
	when $0<\eta< \bar{\eta}$.

	We will construct a Lipschitz retraction $p_0:\Omega_0\to R_1$ such that
	$p_0(z)=z$ for $z\in R_1$, $p_0(z)\in L$ for $z\in L$, and $\Lip(p_0)\leq 3$.
	We now distinguish two cases, depending on cardinality of $\mathcal{X}_r$.

	Case 1: $\card(\mathcal{X}_r)=2$. We assume that $\mathcal{X}_r=\{ x_1,x_2\}$.
	Then $|y_r|=|x_1|=|x_2|=r$, and 
	\[
		0\leq \langle x_1,x_2 \rangle +r^2\leq 2\varepsilon^2r^2.
	\]
	Since $|y_r-\mathfrak{Z}_r|\leq
	\varepsilon r$, we have that $|\langle y_r,x \rangle|\leq \varepsilon r^2$
	for any $x\in L\cap\partial B(0,r)$. 

	We now let $e_1$ and $e_2$ be two unit vectors in $L$ such that $\langle 
	x_1,e_1 \rangle =\langle x_2,e_1 \rangle\geq 0$ and $e_2=-e_1$. Then 
	\[
		0\leq \langle x_i,e_1 \rangle\leq \varepsilon r.
	\]
	We let $\Omega_1'$ and $\Omega_2'$ be the two connected components of
	$\Omega_0\setminus (\cup_{i} D_{x_i,y_r})$ such that $e_i\in \Omega_i'$.
	We put $\Omega_i=\Omega_i'\setminus R_1$.
	We claim that
	\begin{equation}\label{eq:BLNR11}
		|\langle  z_1-z_2, e_i\rangle|\leq 5(\sigma+\varepsilon)|z_1-z_2| 
	\end{equation}
	whenever  $z_1,z_2\in\partial\Omega_i$, $z_1\neq z_2$, $i\in \{1,2\}$.

	Without loss of generality, we assume $z_1,z_2\in \partial\Omega_1$, because
	for another case we will use the same treatment. We see that
	\[
		\dist(z_i,D_{x_j,y_r})=\sigma\dist(z_i,\ell(y_r)).
	\]
	\begin{figure}[!htb]
		\centering
		\includegraphics[width=1in]{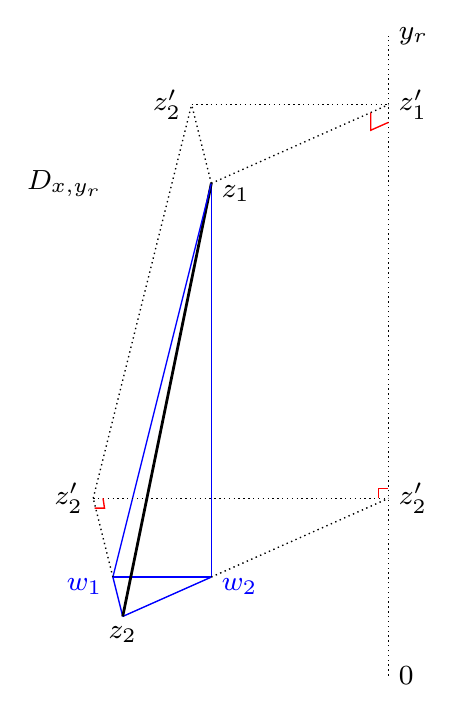}
		\caption{the angle between $z_1-z_2$ and $D_{x,y_r}$ is small.}
		\label{fig:z1z2}
	\end{figure}
	\begin{enumerate}[nolistsep,label=\emph{(\arabic*)}]
		\item In case $z_1,z_2\in \partial R_1^{x_i}\cap \Omega_1$, without loss of
			generality, we assume that $z_1,z_2\in \partial R_1^{x_1}\cap \Omega_1$. 
			We let $\widetilde{z}_i\in D_{x_1,y_r}$ be 
			such that 
			\[
				z_i-\widetilde{z}_i=\dist(z_i,D_{x_1,y_r}),\ i=1,2,
			\]
			and let $z_i'\in \ell(y_r)$ be such that 
			\[
				|z_i-z_i'|=\dist(z_i,\ell(y_r)),
			\]
			and put 
			\[
				w_1=z_1-\widetilde{z}_1+\widetilde{z}_2,\ w_2=z_1-z_1'+z_2',
			\]
			then we get that $z_1-z_2=(z_1-w_2)+(w_2-z_2)$. Moreover, we have that
			$z_1-w_2$ is perpendicular to $w_2-z_2$ and parallel to $y_r$. 
			Thus $|w_2-z_2|\leq |z_1-z_2|$, $|z_1-w_2|\leq |z_1-z_2|$ and 
			\[
				\dist(w_2-z_2,\vect{x_1,y_r})=\sigma|w_2-z_2|.
			\]
			We apply Lemma \ref{le:3vect} to get that 
			\[
				|\langle z_1-w_2,e_1\rangle|\leq \varepsilon|z_1-w_2 |
			\]
			and 
			\[
				|\langle w_2-z_2,e_1 \rangle|\leq
				\left( \sigma+3\varepsilon \right)|w_2-z_2|, 
			\]
			thus
			\[
				|\langle z_1-z_2,e_1 \rangle|\leq |\langle z_1-w_2,e_1
				\rangle|+|\langle w_2-z_2,e_1 \rangle|\leq 
				\left(  \sigma+4\varepsilon\right)|z_1-z_2|.  
			\]

		\item In case $z_1\in \partial R^{x_1}\cap \Omega_1$, $z_2\in \partial R^{x_2}\cap
			\Omega_1$. We let $\widetilde{z}_i\in D_{x_i,y_r}$ be such that 
			\[
				|z_i-\widetilde{z}_i|=\dist(z_i,D_{x_i,y_r}),\ i=1,2,
			\]
			and let $z_{i}'\in
			\ell(y_r)$ be such that 
			\[
				|z_i-z_i'|=\dist(z_i,\ell(y_r)),\ i=1,2.
			\]
			Then by Lemma \ref{le:3vect}, we have that
			\[
				\left\langle z_i-z_i',\frac{x_i}{|x_i|} \right\rangle\geq
				(1-\sigma-\varepsilon)|z_i-z_i'|,\ i=1,2.
			\]
			Since $z_1-z_2=(z_1-z_1')+(z_2'-z_2)+(z_1'-z_2')$, 
			\[
				|\langle z_1'-z_2',e_1
				\rangle|\leq \varepsilon |z_1'-z_2'|\leq \varepsilon |z_1-z_2|
			\]
			and 
			\[
				|\langle z_i-z_i',e_1 \rangle|\leq \left( \sigma+\varepsilon
				\right)|z_i-z_i'|,
			\]
			we get that 
			\begin{equation}\label{eq:BLNR92}
				\begin{aligned}
					|\langle z_1-z_2,e_1 \rangle|&\leq |\langle z_1-z_1',e_1 \rangle|+|\langle
					z_2'-z_2,e_1 \rangle|+|\langle z_1'-z_2',e \rangle|\\
					&\leq 2\cdot \left(  \sigma+\varepsilon\right)
					(|z_1-z_1'|+|z_2-z_2'|)+\varepsilon |z_1-z_2|.
				\end{aligned}
			\end{equation}
			Since $z_1'-z_2'$ is perpendicular to $z_1-z_1'$ and $z_2-z_2'$, and 
			\[
				\left\langle z_i-z_i',\frac{x_i}{|x_i|} \right\rangle\geq
				(1-\sigma-\varepsilon)|z_i-z_i'|,\ i=1,2,
			\]
			and 
			\[
				\left\langle \frac{x_1}{|x_1|},\frac{x_2}{|x_2|}  \right\rangle\leq
				-1+2\varepsilon^2,
			\]
			we get, by Lemma \ref{le:3vect}, that  
			\[
				|z_1-z_1'|+|z_2-z_2'|\leq \left(
				1-\varepsilon^2-5\sqrt{\sigma+\varepsilon}
				\right)^{-1/2}|(z_1-z_1')-(z_2-z_2')|\leq 2 |z_1-z_2|.
			\]
			Thus 
			\[
				\langle z_1-z_2,e \rangle\leq (4\sigma +5\varepsilon) |z_1-z_2|.
			\]
	\end{enumerate}

	We now define $p_0:\Omega_0\to R_1$ as follows:
	for any $z\in \Omega_i$, we let $p_0(z)$ be the unique point in
	$\partial \Omega_i$ such that $p_0(z)-z$
	parallels $e$; and for any $z\in R_1$, we let $p_0(z)=z$. Since $p_0(z)-z$
	parallels $e$, we see that $p_0(L)\subseteq L$. We will check that 
	\begin{equation}\label{eq:BLNR12}
		p_0\text{ is Lipschitz with }\Lip(p_0)\leq \frac{2}{1-5(\sigma+\varepsilon)}.
	\end{equation}
	Indeed, for any $z_1,z_2\in \Omega_0$, we put 
	\[
		p_0(z_i)=z_i+t_ie,\ t_i\in \mathbb{R},
	\]
	then 
	\[
		\begin{aligned}
			|t_1-t_2|&=|\langle (t_1-t_2)e,e \rangle|\\
			&\leq |\langle p_0(z_1)-p_0(z_2),e \rangle|+|\langle z_1-z_2,e
			\rangle|\\
			&\leq 5(\sigma+\varepsilon)|p_0(z_1)-p_0(z_2)|+|z_1-z_2|,
		\end{aligned}
	\]
	and 
	\[
		|p_0(z_1)-p_0(z_2)|\leq |z_1-z_2|+|t_1-t_2|\leq
		5(\sigma+\varepsilon)|p_0(z_1)-p_0(z_2)|+2|z_1-z_2|,
	\]
	thus 
	\[
		|p_0(z_1)-p_0(z_2)|\leq \frac{2}{1-5(\sigma+\varepsilon)}|z_1-z_2|.
	\]

	Case 2: $\card(\mathcal{X}_r)=3$. We assume that $\mathcal{X}_r=\{
	x_1,x_2,x_3\}$, then 
	\[
		|\langle x_i,y_r \rangle|\leq \varepsilon r^2, \left(-\sqrt{3}\varepsilon
		-\frac{1}{2}\right) r^2\leq  \langle x_i,x_j \rangle\leq \left(-\frac{1}{2}+
		2\varepsilon\right) r^2.
	\]

	We put 
	\begin{equation}\label{eq:BLNR20}
		e_1=\frac{x_{2}+x_{3}}{|x_{2}+x_{3}|},
		e_2=\frac{x_{1}+x_{3}}{|x_{1}+x_{3}|},
		e_3=\frac{x_{2}+x_{1}}{|x_{2}+x_{1}|},
	\end{equation}
	and let $\Omega_1'$, $\Omega_2'$ and $\Omega_3'$ be the three connected 
	components of $\Omega_0\setminus (\cup_{i} D_{x_i,y_r})$ such that $e_i\in
	\Omega_i'$. By putting $\Omega_i=\Omega_i'\setminus R_1$,  
	we claim that
	\begin{equation}\label{eq:BLNR21}
		\left( \frac{1}{2}-5(\sigma+\varepsilon) \right)|z_1-z_2|\leq |\langle
		z_1-z_2,e_i\rangle|\leq \left( \frac{1}{2}+5(\sigma+\varepsilon) \right)|z_1-z_2| 
	\end{equation}
	whenever $z_1,z_2\in\partial\Omega_i$, $z_1\neq z_2$, $i\in \{1,2,3\}$.

	Indeed, we only need to check the case $z_1,z_2\in \partial\Omega_1$, and
	the other two cases will be the same. Since $-\sqrt{3}\varepsilon
	-1/2\leq  \langle x_i,x_j \rangle\leq 1/2+ 2\varepsilon$, we have that
	$(1/2-\varepsilon)r\leq \langle x_i,e_1 \rangle\leq (1/2+\varepsilon)r$ for
	$i=2,3$.

	If $z_1,z_2\in \partial R_1^{x_2}\cap \Omega_1$ or $z_1,z_2\in \partial
	R_1^{x_3}\cap \Omega_1$, we assume that $z_1,z_2\in \partial R_1^{x_2}\cap
	\Omega_1$, and let $\widetilde{z}_i\in D_{x_2,y_r}$ be such that 
	\[
		z_i-\widetilde{z}_i=\dist(z_i,D_{x_2,y_r}),\ i=1,2,
	\]
	and let $z_i'\in \ell(y_r)$ be such that 
	\[
		|z_i-z_i'|=\dist(z_i,\ell(y_r)),
	\]
	and put 
	\[
		w_1=z_1-\widetilde{z}_1+\widetilde{z}_2,\ w_2=z_1-z_1'+z_2',
	\]
	then we get that $z_1-w_2$ is perpendicular to $w_2-z_2$ and parallel to
	$y_r$. Since $z_1-z_2=(z_1-w_2)+(w_2-z_2)$, we have that 
	$|w_2-z_2|\leq |z_1-z_2|$, $|z_1-w_2|\leq |z_1-z_2|$ and 
	\[
		\dist(w_2-z_2,\vect{x_1,y_r})=\sigma|w_2-z_2|.
	\]
	We apply Lemma \ref{le:3vect} to get that 
	\[
		|\langle z_1-w_2,e_1\rangle|\leq \varepsilon|z_1-w_2 |
	\]
	and 
	\[
		|\langle w_2-z_2,e_1 \rangle|\leq
		\left( \frac{1}{2}+\varepsilon+\sigma+\varepsilon \right)|w_2-z_2|, 
	\]
	thus
	\[
		|\langle z_1-z_2,e_1 \rangle|\leq |\langle z_1-w_2,e_1
		\rangle|+|\langle w_2-z_2,e_1 \rangle|\leq 
		\left( \frac{1}{2}+ \sigma+3\varepsilon\right)|z_1-z_2|.  
	\]

	If $z_1\in \partial R^{x_2}\cap \Omega_1$, $z_2\in \partial 
	R^{x_3}\cap \Omega_1$, we let $\widetilde{z}_i\in D_{x_i,y_r}$ be such that 
	\[
		|z_1-\widetilde{z}_1|=\dist(z_1,D_{x_2,y_r}),\
		|z_2-\widetilde{z}_2|=\dist(z_2,D_{x_3,y_r})
	\]
	and let $z_{i}'\in
	\ell(y_r)$ be such that 
	\[
		|z_i-z_i'|=\dist(z_i,\ell(y_r)),\ i=1,2.
	\]
	Since $z_1-z_2=(z_1-z_1')+(z_2'-z_2)+(z_1'-z_2')$, 
	\[
		|\langle z_1'-z_2',e_1
		\rangle|\leq \varepsilon |z_1'-z_2'|\leq \varepsilon |z_1-z_2|
	\]
	and 
	\[
		|\langle z_i-z_i',e_1 \rangle|\leq \left(\frac{1}{2}+\varepsilon+ \sigma+\varepsilon
		\right)|z_i-z_i'|,
	\]
	we get that 
	\begin{equation}\label{eq:BLNR123}
		\begin{aligned}
			|\langle z_1-z_2,e_1 \rangle|&\leq |\langle z_1-z_1',e_1 \rangle|+|\langle
			z_2'-z_2,e_1 \rangle|+|\langle z_1'-z_2',e \rangle|\\
			&\leq \left( \frac{1}{2}+ \sigma+2\varepsilon\right)
			(|z_1-z_1'|+|z_2-z_2'|)+\varepsilon |z_1-z_2|.\\
		\end{aligned}
	\end{equation}
	By Lemma \ref{le:3vect}, we have that 
	\[
		\left\langle z_1-z_1',\frac{x_2}{|x_2|}\right\rangle\geq 
		(1-\sigma-\varepsilon)|z_1-z_1'|
	\]
	and 
	\[
		\left\langle z_2-z_2',\frac{x_3}{|x_3|}\right\rangle\geq 
		(1-\sigma-\varepsilon)|z_2-z_2'|.
	\]
	Applying Lemma \ref{le:3vect} with $\langle x_2/|x_2|,x_3/|x_3| \rangle\leq
	-1/2+2\varepsilon$, we get that 
	\[
		\begin{aligned}
			|z_1-z_1'|+|z_2-z_2'|&\leq \left( \frac{2}{1+1/2-2\varepsilon
			-4\sqrt{2\sigma+2\varepsilon}} \right)^{1/2}|(z_1-z_1')-(z_2-z_2')| \\
			&\leq \frac{2}{\sqrt{3}}\left(
			1-\frac{2\varepsilon+4\sqrt{2\sigma+2\varepsilon}}{3} \right)|z_1-z_2|.
		\end{aligned}
	\]
	We get, from \eqref{eq:BLNR123}, that 
	\begin{equation}\label{eq:BLNR134}
		|\langle z_1-z_2,e_1 \rangle|\leq \frac{2}{3}|z_1-z_2|.
	\end{equation}

	For any $z\in \Omega_i$, we now let $p_0(z)$ be the unique point in
	$\partial \Omega_i$ such that $p_0(z)-z$ parallels $e$; and for $z\in R_1$,
	we let $p_0(z)=z$. Then $p_0(L)\subseteq L$. We will check that 
	\begin{equation}\label{eq:BLNR23}
		p_0\text{ is Lipschitz with }\Lip(p_0)\leq 6.
	\end{equation}

	For any $z_1,z_2\in \Omega_i$, we put 
	\[
		p_0(z_j)=z_j+t_je_i,\ t_i\in \mathbb{R},\ j=1,2,
	\]
	then 
	\[
		\begin{aligned}
			|t_1-t_2|&=|\langle (t_1-t_2)e_i,e_i \rangle|\\
			&\leq |\langle p_0(z_1)-p_0(z_2),e_i \rangle|+|\langle z_1-z_2,e_i
			\rangle|\\
			&\leq \frac{2}{3}|p_0(z_1)-p_0(z_2)|+|z_1-z_2|,
		\end{aligned}
	\]
	and 
	\[
		|p_0(z_1)-p_0(z_2)|\leq |z_1-z_2|+|t_1-t_2|\leq
		\frac{2}{3}|p_0(z_1)-p_0(z_2)|+2|z_1-z_2|,
	\]
	thus 
	\[
		|p_0(z_1)-p_0(z_2)|\leq 6|z_1-z_2|.
	\]

	By the definition of $R^x$ and $R_1^x$, we have that 
	\[
		R^x=\{ z\in R_1^x \mid \dist(z,D_{x,y_r})\leq \sigma\dist(z,\ell(x))\}.
	\]
	Similar as above, we can that, for any $z_1,z_2\in R_1^x \cap \partial R^{x}$ with
	$[z_1,z_2]\cap D_{x,y_r}=\emptyset$, if $\card(\mathcal{X}_r)=2$ then
	\[
		|\langle z_1-z_2,e_i \rangle|\leq 5(\sigma+\varepsilon)|z_1-z_2|;
	\]
	if $\card(\mathcal{X}_r)=3$ then 
	\[
		|\langle z_1-z_2,e_i \rangle|\leq \left( \frac{1}{2}+\sigma+3\varepsilon
		\right)|z_1-z_2|,
	\]
	where $e_i$ is the vector in \ref{eq:BLNR20} such that
	$z_1,z_2\in \Omega_i$. 

	We now consider the mapping $p_1:R_1\to R$ defined by 
	\[
		p_1(z)=
		\begin{cases}
			z,& \text{ for }z\in R,\\
			z-te_i\in \partial R \cap \Omega_i,&\text{ for }z\in \Omega_i.
		\end{cases}
	\]
	By the same reason as above, we get that 
	\begin{equation}\label{eq:BLNR50}
		\Lip(p_1)\leq \frac{2}{1-1/2-\sigma-3\varepsilon}\leq 5.
	\end{equation}

	We define a mapping $p_2:R\cap \overline{B(0,1)}\to \Sigma$ as follows:
	we know $\Sigma^x_{u}$ is the graph of $u$ over $D_{x,y_r}$, thus for any
	$z\in R^x$, there is only one point in the intersection of $\Sigma^x_u$ and
	the line which is perpendicular to $D_{x,y_r}$ and through $z$, we define
	$p_2(z)$ to be the unique intersection point. That is, $p_2(z)$ is the
	unique point in $\Sigma^x_u$ such that $p_2(z)-z$ is perpendicular to $
	D_{x,y_r}$. We will show that $p_2$ is Lipschitz and $\Lip(p_2)\leq
	1+10^{4}\eta$. Indeed, for any points $z_1,z_2\in R^{x}$, we let
	$\widetilde{z}_i$, $i=1,2$, be the points in $D_{x,y_r}$ such that
	$z_i-\widetilde{z}_i$ is perpendicular to $D_{x,y_r}$, then 
	\[
		|(p_2(z_1)-z_1)-(p_2(z_2)-z_2)|=\left |u\left(
		\widetilde{z}_1\right)-u\left( \widetilde{z}_2 \right)\right |\leq
		\Lip(u)|\widetilde{z}_1-\widetilde{z}_2|\leq \Lip(u)|z_1-z_2|,
	\]
	thus
	\[
		|p_2(z_1)-p_2(z_2)|\leq (1+\Lip(u))|z_1-z_2|\leq (1+10^{4}\eta)|z_1-z_2|.
	\]

	Let $p_3:\mathbb{R}^{3}\to \mathbb{R}^{3}$ be the mapping defined by
	\[
		p_3(x)=\begin{cases}
			x,& |x|\leq 1\\
			\frac{x}{|x|},&|x|>1.
		\end{cases}
	\]
	Then $p=p_3\circ p_2\circ p_3\circ p_1\circ p_0$ is our desire
	mapping.
\end{proof}

\begin{lemma}\label{le:CotoCo}
	For any $r\in (0,\mathfrak{r})\cap \mathscr{R}_1$, we let $\Sigma$ be
	as in \eqref{eq:CoLip}, and let $\Sigma_r$ be given by $\scale{r}(\Sigma)$. 
	Then we have that 
	\[
		\HM^2(E\cap B(0,r))\leq \HM^2(\Sigma_r)+2550\int_{E\cap \partial
		B(0,r)}\dist(z,\Sigma_r)d\HM^1(z) + (2r)^2h(2r).
	\]
\end{lemma}
\begin{proof}
	For any $\xi>0$, we consider the function $\psi_{\xi}:[0,\infty)\to\mathbb{R}$
	defined by 
	\[
		\psi_{\xi}(t)=\begin{cases}
			1,& 0\leq t\leq 1-\xi\\
			-\frac{t-1}{\xi},&1-\xi<t\leq 1\\
			0,&t> 1,
		\end{cases}
	\]
	and the mapping $\phi_{\xi}:\Omega_0\to \Omega_0$ defined by 
	\[
		\phi_{\xi}(z)=\psi_{\xi}(|z|)p(z)+(1-\psi_{\xi}(|z|))z.
	\]
	Then we get that $\phi_{\xi}(L)\subseteq L$. For any $t\in [0,1]$, we put
	\[
		\varphi_t(z)=tr\phi_{\xi}\left( z/r
		\right)+(1-t)z,\text{ for } z\in \Omega_0.
	\]
	Then $\{ \varphi_t \}_{0\leq t\leq 1}$ is a sliding deformation, and we get that 
	\begin{equation}\label{eq:CotoCo1}
		\HM^2(E\cap \overline{B(0,r)})\leq \HM^2(\varphi_1(E)\cap
		\overline{B(0,r)})+(2r)^2h(2r).
	\end{equation}
	Since $	\psi_{\xi}(t)=1$ for $t\in [0,1-\xi]$, we get that 
	\[
		\varphi_1(E\cap B(0,(1-\xi)r))=p(E\cap B(0,(1-\xi)r))\subseteq \Sigma_r.
	\]
	We set $A_{\xi}=B(0,r)\setminus B(0,(1-\xi)r)$. By Theorem 3.2.22 in
	\cite{Federer:1969}, we get that 
	\begin{equation}\label{eq:CotoCo3}
		\HM^2(\varphi_1(E\cap A_{\xi}))\leq \int_{E\cap
		A_{\xi}}\ap J_{2}(\varphi_1\vert_{E})(z)d\HM^2(z).
	\end{equation}

	For any $z\in A_{\xi}$ and $v\in \mathbb{R}^{3}$, we have, by setting
	$z'=z/r$, that 
	\[
		D\varphi_1(z)v=\psi_{\xi}(|z'|)Dp(z')v+(1-\psi_{\xi}(|z'|))v+
		\psi_{\xi}'(|z'|)\langle z/|z|,v\rangle (rp(z')-z).
	\]
	For any $z\in A_{\xi}\cap E$, we let $v_1,v_2\in T_zE$ be such that 
	\[
		|v_1|=|v_2|=1,\ v_1\perp z\text{ and } v_2\perp v_1,
	\]
	then we have that $\langle z/|z|,v\rangle=\cos\theta(z)$, and that
	\[
		|\psi_{\xi}(|z'|)Dp(z')v_i+(1-\psi_{\xi}(|z'|))v_i|\leq |Dp(z')v_i|\leq
		\Lip(p),
	\]
	thus
	\begin{equation}\label{eq:CotoCo10}
		\begin{aligned}
			\ap J_{2}(\varphi_1\vert_{E})(z)&=|D\varphi_1(z)v_1\wedge D\varphi_1(z)v_2|\\
			&\leq
			\Lip(p)^2+\frac{1}{\xi}\Lip(p)\cos\theta(z) |rp(z')-z|.
		\end{aligned}
	\end{equation}
	Since $p(\widetilde{z})=\widetilde{z}$ for any $\widetilde{z}\in \Sigma$, we
	have that 
	\[
		|p(z')-z'|=|p(z')-p(\widetilde{z})+\widetilde{z}-z'|\leq
		(\Lip(p)+1)|\widetilde{z}-z'|,
	\]
	then we get that 
	\[
		|p(z')-z'|\leq (\Lip(p)+1)\dist(z,\Sigma).
	\]
	We now get, from \eqref{eq:CotoCo10}, that
	\begin{equation}\label{eq:CotoCo11}
		\ap J_{2}(\varphi_1\vert_{E})(z)\leq
		\Lip(p)^2+\frac{1}{\xi}\Lip(p)(\Lip(p)+1)
		\dist(z,\Sigma_r)\cos\theta(z),
	\end{equation}
	plug that into \eqref{eq:CotoCo3} to get that 
	\[
		\begin{aligned}
			\HM^2(\varphi_1(E\cap A_{\xi}))&\leq 2500\HM^2(E\cap
			A_{\xi})+\frac{2550}{\xi}\int_{E\cap
			A_{\xi}}\dist(z,\Sigma_r)\cos\theta(z)d\HM^2(z)\\
			&\leq 2500\HM^2(E\cap
			A_{\xi})+\frac{2550}{\xi}\int_{(1-\xi)r}^r\int_{E\cap\partial 
			B(0,t)}\dist(z,\Sigma_r)d\HM^1(z)dt,
		\end{aligned}
	\]
	we let $\xi\to 0+$, then we get that, for such $r$, 
	\[
		\lim_{\xi\to 0+}\HM^2(\varphi_1(E\cap A_{\xi}))\leq 2550r\int_{E\cap\partial 
		B(0,r)}\dist(z,\Sigma_r)d\HM^1(z),
	\]
	thus
	\[
		\HM^2(E\cap B(0,r))\leq \HM^2(\Sigma_r)+2550r\int_{E\cap \partial
		B(0,r)}\dist(z,\Sigma_r)d\HM^1(z) + (2r)^2h(2r).
	\]
\end{proof}
\subsection{The main comparison statement}
For any $x,y\in \Omega_0\cap \partial B(0,1)$, if $|x-y|<2$, we denote by
$g_{x,y}$ the unique geodesic on $\Omega_0\cap \partial B(0,1)$ which join
$x$ and $y$.

We will denote by $B_t$ the open ball $B(0,t)$ sometimes for short.
\begin{lemma}\label{le:CLLHP}
	Let $\tau\in (0,10^{-4})$ be a given. Then there is a constant $\vartheta>0$ such
	that the following hold. Let $a\in \partial B(0,1)$ and 
	$b,c\in L_0\cap \partial B(0,1)$ be such that $\dist(a,(0,0,1))\leq \tau$,
	$\dist(b,(1,0,0))\leq \tau$ and $\dist(c,(-1,0,0))\leq \tau$.
	Let $X$ be the cone over $g_{a,b}\cup g_{a,c}$. Then there is a 
	Lipschitz mapping $\varphi:\Omega_0\to \Omega_0$ with $\varphi(
	L_0)\subseteq L_0$, $|\varphi(z)|\leq 1$ when $|z|\leq 1$, and $\varphi(z)=z$
	when $|z|>1$, such that 
	\[
		\HM^2(\varphi(X)\cap \overline{B(0,1)})\leq
		(1-\vartheta)\HM^2(X \cap B(0,1))+\frac{\vartheta\pi}{2}.
	\]
\end{lemma}
\begin{proof}
	We let $b_0$ a unit vector in $L_0$ which is perpendicular to $b$, and let 
	$c_0$ be a unit vector in $L_0$ which is perpendicular to $c$, such that
	$b_0+c_0$ is parallel to $b+c$, and take 
	\[
		u_i=\frac{a-\langle a,i \rangle i}{|a-\langle a,i \rangle i|},\
		e_i=\frac{i-\langle i,a \rangle a}{|i-\langle i,a \rangle a|},\
		\text{ for }i\in \{b,c\},
	\]
	$v_a=\lambda_a (e_b+e_c)$, $v_b=\lambda_b b_0$ and $v_c=\lambda_c c_0$,
	where $\lambda_j\in \mathbb{R}$, $j\in \{a,b,c\}$, will be chosen later.
	We let $\psi_1:\mathbb{R}\to \mathbb{R}$ be a function of class $C^1$ such 
	that $0\leq \psi_1\leq 1$, $\psi_1(x)=0$ for $x\in (-\infty,1/4)\cup
	(3/4,+\infty)$, $\psi_1(x)=1$ for $x\in [2/5,3/5]$, and $|\psi_1'|\leq 10$.
	We let $\psi_2:\mathbb{R}\to \mathbb{R}$ be a non increasing function of 
	class $C^1$ such that $0\leq \psi_2\leq 1$, $\psi_2(x)=1$ for $x\in 
	(-\infty,0]$, $\psi_2(x)=0$ for $x\in [1/5,+\infty)$, and $|\psi_2'|\leq
	10$. We let $\psi: \mathbb{R}^3 \times \mathbb{R}^3\to \mathbb{R}$ be a
	function defined by 
	\begin{equation}\label{eq:FP}
		\psi(z,v)=\psi_1(\langle z,v \rangle) \psi_2(|z-\langle
		z,v \rangle v|).
	\end{equation}

	We now consider the mapping $\varphi:\mathbb{R}^3\to
	\mathbb{R}^3$ defined by 
	\[
		\varphi(z)=z+\psi(z,a)v_a+\psi(z,b)v_b+\psi(z,c)v_c.
	\]
	We see that $\supp(\psi(\cdot,a))$, $\supp(\psi(\cdot,b))$ and
	$\supp(\psi(\cdot,c))$ are mutually disjoint, and that
	\[
		\overline{\{z\in \mathbb{R}^3:\varphi(z)\neq z\}}\subseteq B(0,1),\ 
		\varphi(\Omega_0)\subseteq \Omega_0,\ \varphi(L_0)\subseteq L_0 .
	\]
	We have that
	\[
		D \varphi(z)w=w+\langle D\psi(\cdot,a),w \rangle v_a+\langle
		D\psi(\cdot,b),w \rangle v_b+\langle D\psi(\cdot,c),w \rangle v_c.
	\]
	By setting $z_v^{\perp}=z-\langle z,v \rangle v$ for convenient, if $w\neq
	0$ and $z_v^{\perp}\neq 0$, we have that 
	\[
		D\psi(\cdot,v)w=\psi_1'(\langle z,v \rangle)\psi_2(|z_v^{\perp}|)\langle w/|w|,v
		\rangle +\psi_1(\langle z,v \rangle)\psi_2'(|z_v^{\perp}|)\langle
		w_v^{\perp},z_v^{\perp} /|z_v^{\perp}|\rangle.
	\]
	If $w$ is perpendicular to $v$, then $w_v^{\perp}=w$; if $w$ is parallel
	to $v$ and $|v|=1$ , then $w_v^{\perp}=0$. We denote by
	$W_j=\supp(\psi(\cdot,j))$ for $j\in \{a,b,c\}$. Then 
	\[
		D\psi(\cdot,v)w=\begin{cases}
			w,&z\notin W_a\cup W_b\cup W_c,\\
			w+\langle D\psi(\cdot,v),w \rangle v_j, & z\in W_a\cup W_b\cup W_c.
		\end{cases}
	\]
	But 
	\[
		\langle D\psi(\cdot,j),j \rangle =
		\psi_1'(\langle z,j \rangle) \psi_2(|z_j^{\perp}|), \ j\in \{a,b,c\},
	\]
	\[
		\langle D\psi(\cdot,i),u_i \rangle =
		\psi_1(\langle z,i \rangle) \psi_2'(|z_i^{\perp}|)
		\langle u_i, z_i^{\perp} /|z_i^{\perp}|\rangle, \ i\in \{b,c\},
	\]
	and
	\[
		\langle D\psi(\cdot,a),e_i \rangle =
		\psi_1(\langle z,a \rangle) \psi_2'(|z_a^{\perp}|)
		\langle e_i, z_a^{\perp} /|z_a^{\perp}|\rangle, \ i\in \{b,c\},
	\]
	by putting 
	\[
		g_j(z)=\psi_1'(\langle z,j \rangle) \psi_2(|z_j^{\perp}|),\ j\in
		\{a,b,c\},
	\]
	\[
		g_{a,i}(z)=\psi_1(\langle z,a \rangle) \psi_2'(|z_a^{\perp}|) 
		\langle e_i, z_a^{\perp} /|z_a^{\perp}|\rangle, \ i\in \{b,c\}
	\]
	and 
	\[
		g_{i,i}(z)=\psi_1(\langle z,i \rangle) \psi_2'(|z_i^{\perp}|) 
		\langle v_i, z_i^{\perp} /|z_i^{\perp}|\rangle, \ i\in \{b,c\},
	\]
	and denote by $X_i$ the cone over $g_{a,i}$, $i\in \{b,c\}$, we have that 
	\[
		D \varphi(z)a \wedge D \varphi(z)e_i = a \wedge e_i + g_a(z) v_a \wedge e_i
		+ g_{a,i}(z) a \wedge v_a, \ z\in X_i\cap W_a
	\]
	and 
	\[
		D \varphi(z)i \wedge D \varphi(z)u_i =i \wedge u_i + g_i(z) v_i \wedge u_i
		+ g_{i,i}(z) i \wedge v_i,\ z\in X_i\cap W_i.
	\]
	If $z\in X_i\cap W_a$, $i\in \{b,c\}$, we have that
	\[
		\begin{aligned}
			J_2 \varphi\vert_{X}(z)&=\|D \varphi(z)a \wedge D \varphi(z)e_i\|\\
			&\leq 1+\left\langle a \wedge e_i,g_a(z) v_a \wedge e_i
			+ g_{a,i}(z) a \wedge v_a\right\rangle + \frac{1}{2}\|g_a(z) v_a \wedge e_i
			+ g_{a,i}(z) a \wedge v_a\|^2\\
			&= 1+g_a(z)\langle a,v_a \rangle + g_{a,i}(z) \langle e_i,v_a \rangle +
			\frac{1}{2}\left( g_a(z)^2\|v_a\wedge
			e_i\|^2+g_{a,i}(z)^2|v_a|^2\right)\\
			&\leq 1+  g_{a,i}(z) \langle e_i,v_a \rangle + 100|v_a|^2.
		\end{aligned}
	\]
	Similarly, we have that, for $z\in X_i\cap W_i$,
	\[
		J_2 \varphi\vert_{X}(z)=\|D \varphi(z)i \wedge D \varphi(z)u_i\|\leq 1+ 
		g_{i,i}(z) \langle u_i,v_i \rangle + 100|v_i|^2.
	\]

	We see that $z_a^{\perp}/|z_a^{\perp}|=e_i$ when $z\in X_i\setminus \vect{a}$, and
	$z_i^{\perp}/|z_i^{\perp}|=u_i$ in case $z\in X_i\setminus \vect{i}$, thus 
	\[
		g_{a,i}(z)=\psi_1(\langle z,a \rangle) \psi_2'(|z_a^{\perp}|) \text{ and }
		g_{i,i}(z)=\psi_1(\langle z,i \rangle) \psi_2'(|z_i^{\perp}|).
	\]
	Hence, for $j=a$ or $i$, we have that 
	\[
		\begin{aligned}
			\int_{z\in X_i \cap W_j} g_{j,i}(z)d\HM^2(z)&= 
			\int_{z\in X_i\cap W_j} \psi_1(\langle z,j
			\rangle)\psi_2'(|z_j^{\perp}|)d\HM^2(z)\\
			&= \int_{0}^{+\infty}\int_{0}^{+\infty}\psi_1(t)
			\psi_2'(s) dt ds\\
			&= -\int_{0}^{+\infty}\psi_1(t) dt <-\frac{1}{5},
		\end{aligned}
	\]
	Thus
	\[
		\begin{aligned}
			\HM^2(\varphi(X\cap B_1))&=
			\int_{z\in X\cap B(0,1)}J_2 \varphi\vert_{X}(z) d\HM^2(z)\\
			&\leq (1+100\sum_j|v_j|^2)\HM^2(X\cap B_1)-\frac{1}{5}(\langle
			v_a,e_b+e_c \rangle +\sum_{i}\langle u_i,v_i \rangle)
		\end{aligned}
	\]
	If we take $\lambda_a=10^{-3}\HM^2(X\cap B_1)^{-1}$ and 
	$\lambda_i=10^{-3}\HM^2(X\cap B_1)^{-1}\langle u_i,i_0 \rangle$, $i\in
	\{b,c\}$, then 
	\[
		\HM^2(\varphi(X\cap B_1)) \leq \HM^2(X\cap B_1)-
		10^{-4}(|e_b+ e_c|^2+ \langle u_b,b_0 \rangle^2 + \langle u_c,c_0
		\rangle^2).
	\]

	Since $|\langle a,w \rangle|\leq \tau|w|$ for $w\in L_0$,  and 
	$-1\leq \langle b,c \rangle\leq -1+2\tau^2$, we get that 
	\[
		\begin{aligned}
			|e_b+e_c|^{2}&=2(1+\langle e_b,e_c \rangle)= \frac{2}{1-\langle e_b,e_c
			\rangle}(1-\langle e_b,e_c \rangle^2)\\
			&\geq 1-\frac{(\langle b,c \rangle-\langle a,b \rangle \langle a,c
			\rangle)^2}{(1-\langle a,b \rangle^2)(1-\langle a,c \rangle^2)}\\
			&\geq 1-\langle a,b \rangle^2-\langle a,c \rangle^2-\langle b,c
			\rangle^2+2\langle a,b \rangle\langle b,c \rangle\langle c,a \rangle\\
			&=(1-\langle b,c \rangle +2 \langle a,b \rangle \langle a,c
			\rangle)(1+\langle b,c \rangle) -\langle a,b+c \rangle^2\\
			&\geq (1-3\tau^2)|b+c|^2.
		\end{aligned}
	\]

	Since $\arcsin x= x+\sum_{n\geq 1}C_nx^{2n+1}$ for $|x|\leq 1$, where
	$C_n=\frac{(2n)!}{4^n(n!)^2(2n+1)}$, we have that 
	\[
		\begin{aligned}
			\HM^2(X\cap B_1)-\frac{\pi}{2}&=\frac{1}{2}(\arccos \langle a,b \rangle+
			\arccos \langle a,c \rangle)-\frac{\pi}{2}\\
			&=-\frac{1}{2} (\arcsin \langle a,b \rangle+ \arcsin \langle a,c
			\rangle) \leq \frac{1}{2}(1+\tau)|\langle a,b+c \rangle|.
		\end{aligned}
	\]

	If $b+c\neq 0$, then $|b_0+c_0|\geq 1$, and we have that 
	\[
		\left\langle a,\frac{b+c}{|b+c|}  \right\rangle^2=\left\langle
		a,\frac{b_0+c_0}{|b_0+c_0|}  \right\rangle^2 \leq 2 \left(\langle a,b_0
		\rangle^2 + \langle a,c_0 \rangle^2\right).
	\]
	We get so that in any case 
	\[
		|\langle a,b+c \rangle| \leq \frac{1}{2} \left(|b+c|^2+ 2 \langle a,b_0
		\rangle^2 + 2\langle a,c_0 \rangle^2\right).
	\]

	Since 
	\[
		\langle u_b,b_0 \rangle^2+\langle u_c,c_0 \rangle^2=
		\frac{ \langle a,b_0 \rangle^2} {1-\langle a,b \rangle^2}+
		\frac{ \langle a,c_0 \rangle^2} {1-\langle a,c \rangle^2}
		\geq \langle a,b_0 \rangle^2 +  \langle a,c_0 \rangle^2,
	\]
	we get that 
	\[
		\begin{aligned}
			\HM^2(\varphi(X\cap B_1)) &\leq \HM^2(X\cap B_1)-
			10^{-4}\left(\frac{1}{2}|b+c|^2+ \langle a,b_0 \rangle^2 +
			\langle a,c_0 \rangle^2 \right)\\
			&\leq\HM^2(X\cap B_1)-10^{-4} \left(\HM^2(X\cap B_1)-\frac{\pi}{2}\right) .
		\end{aligned}
	\]
\end{proof}
\begin{lemma}\label{le:CLLHY}
	Let $\tau\in (0,10^{-4})$ be a given. Then there is a constant $\vartheta>0$ such
	that the following hold. Let $a\in \partial B(0,1)$ and 
	$b,c,d\in L_0\cap \partial B(0,1)$ be such that $\dist(a,(0,0,1))\leq \tau$,
	$\dist(b,(-1/2,\sqrt{3}/2,0))\leq \tau$,
	$\dist(c,(-1/2,-\sqrt{3}/2,0))\leq \tau$ and $\dist(d,(1,0,0))\leq \tau$.
	Let $X$ be the cone over $g_{a,b}\cup g_{a,c}\cup g_{a,d}$. 
	Then there is a Lipschitz mapping $\varphi:\Omega_0\to \Omega_0$ with 
	$\varphi(E\cap L)\subseteq L$, $|\varphi(z)|\leq 1$ when $|z|\leq 1$, and
	$\varphi(z)=z$ when $|z|>1$, such that 
	\[
		\HM^2(\varphi(X)\cap \overline{B(0,1)})\leq
		(1-\vartheta)\HM^2(X\cap B(0,1))+\vartheta\frac{3\pi}{4}.
	\]
\end{lemma}
\begin{proof}
	We let $b_0$, $c_0$ and $d_0$ be unit vectors in $L_0$ such that 
	\[
		b_0 \perp b, c_0\perp c, d_0\perp d.
	\]
	For $i\in \{b,c,d\}$, we put 
	\[
		u_i=\frac{a-\langle a,i \rangle i}{|a-\langle a,i \rangle i|},\
		e_i=\frac{i-\langle i,a \rangle a}{|i-\langle i,a \rangle a|}.
	\]
	We take $v_a=\lambda_a (e_b+e_c+e_d)$ and $v_i=\lambda_i i_0$,
	where $\lambda_i>0$, $i\in \{b,c,d\}$, will be chosen later.
	We let $\psi$ be the same as in \eqref{eq:FP}, and consider the mapping 
	$\varphi:\mathbb{R}^3\to \mathbb{R}^3$ defined by 
	\[
		\varphi(z)=z+\psi(z,a)v_a+\psi(z,b)v_b+\psi(z,c)v_c+\psi(z,d)v_d.
	\]
	We see that $\supp(\psi(\cdot,a))$, $\supp(\psi(\cdot,b))$, 
	$\supp(\psi(\cdot,c))$ and $\supp(\psi(\cdot,d))$ are mutually disjoint,
	and that
	\[
		\overline{\{z\in \mathbb{R}^3:\varphi(z)\neq z\}}\subseteq B(0,1),\ 
		\varphi(\Omega_0)\subseteq \Omega_0,\ \varphi(L_0)\subseteq L_0 .
	\]
	By putting $W_j=\supp(\psi(\cdot,j))$ for $j\in \{a,b,c,d\}$, we have that
	\[
		D\psi(\cdot,v)w=\begin{cases}
			w,&z\notin W_a\cup W_b\cup W_c \cup W_d,\\
			w+\langle D\psi(\cdot,v),w \rangle v_j, & z\in W_a\cup W_b\cup W_c\cup
			W_d,
		\end{cases}
	\]
	and
	\[
		\begin{gathered}
			\langle D\psi(\cdot,j),j \rangle =
			\psi_1'(\langle z,j \rangle) \psi_2(|z_j^{\perp}|), \ j\in \{a,b,c,d\},\\
			\langle D\psi(\cdot,i),u_i \rangle =
			\psi_1(\langle z,i \rangle) \psi_2'(|z_i^{\perp}|)
			\langle u_i, z_i^{\perp} /|z_i^{\perp}|\rangle,\\
			\langle D\psi(\cdot,a),e_i \rangle =
			\psi_1(\langle z,a \rangle) \psi_2'(|z_a^{\perp}|)
			\langle e_i, z_a^{\perp} /|z_a^{\perp}|\rangle, \ i\in \{b,c,d\},
		\end{gathered}
	\]
	where $z_w=z-\langle z,w \rangle w$. By putting 
	\[
		\begin{gathered}
			g_j(z)=\psi_1'(\langle z,j \rangle) \psi_2(|z_j^{\perp}|),\ j\in
			\{a,b,c,d\},\\
			g_{a,i}(z)=\psi_1(\langle z,a \rangle) \psi_2'(|z_a^{\perp}|) 
			\langle e_i, z_a^{\perp} /|z_a^{\perp}|\rangle,\\
			g_{i,i}(z)=\psi_1(\langle z,i \rangle) \psi_2'(|z_i^{\perp}|) 
			\langle v_i, z_i^{\perp} /|z_i^{\perp}|\rangle, \ i\in \{b,c,d\},
		\end{gathered}
	\]
	and denote by $X_i$ the cone over $g_{a,i}$, $i\in \{b,c,d\}$, we have that 
	\[
		\begin{gathered}
			D \varphi(z)a \wedge D \varphi(z)e_i = a \wedge e_i + g_a(z) v_a \wedge e_i
			+ g_{a,i}(z) a \wedge v_a, \ z\in X_i\cap W_a,\\
			D \varphi(z)i \wedge D \varphi(z)u_i =i \wedge u_i + g_i(z) v_i \wedge u_i
			+ g_{i,i}(z) i \wedge v_i,\ z\in X_i\cap W_i.
		\end{gathered}
	\]
	We have that, for $i\in \{b,c,d\}$,
	\[
		\begin{gathered}
			J_2 \varphi\vert_{X}(z)=\|D \varphi(z)a \wedge D \varphi(z)e_i\|\leq
			1+  g_{a,i}(z) \langle e_i,v_a \rangle + 100|v_a|^2, z\in X_i\cap W_a,\\
			J_2 \varphi\vert_{X}(z)=\|D \varphi(z)i \wedge D \varphi(z)u_i\|\leq 1+ 
			g_{i,i}(z) \langle u_i,v_i \rangle + 100|v_i|^2,z\in X_i\cap W_i.
		\end{gathered}
	\]

	Since $z_a^{\perp}/|z_a^{\perp}|=e_i$ when $z\in X_i\setminus \vect{a}$, and
	$z_i^{\perp}/|z_i^{\perp}|=u_i$ in case $z\in X_i\setminus \vect{i}$, we
	have that 
	\[
		g_{a,i}(z)=\psi_1(\langle z,a \rangle) \psi_2'(|z_a^{\perp}|) \text{ and }
		g_{i,i}(z)=\psi_1(\langle z,i \rangle) \psi_2'(|z_i^{\perp}|).
	\]
	Thus, for $j=a$ or $i$, 
	\[
		\int_{z\in X_i \cap W_j} g_{j,i}(z)d\HM^2(z)
		= -\int_{0}^{+\infty}\psi_1(t) dt <-\frac{1}{5}.
	\]
	Hence
	\[
		\begin{aligned}
			\HM^2(\varphi(X\cap B_1))&=
			\int_{z\in X\cap B_1}J_2 \varphi\vert_{X}(z) d\HM^2(z)\\
			&\leq \left(1+100(|v_a|^2+|v_b|^2+|v_c|^2+|v_d|^2)\right)\HM^2(X\cap B_1)\\
			&\quad -\frac{1}{5}\left(\langle
			v_a,e_b+e_c+e_d \rangle +\langle u_b,v_b \rangle+\langle u_c,v_c
			\rangle+\langle u_d,v_d \rangle\right).
		\end{aligned}
	\]
	If we take $\lambda_a=10^{-3}\HM^2(X\cap B_1)^{-1}$ and 
	$\lambda_i=10^{-3}\HM^2(X\cap B_1)^{-1}\langle u_i,i_0 \rangle$, $i\in
	\{b,c,d\}$, then 
	\begin{equation}\label{eq:CLLHY40}
		\HM^2(\varphi(X\cap B_1)) \leq \HM^2(X\cap B_1)-
		10^{-4}\left(|e_b+ e_c+e_d|^2+ \sum_{i}\langle u_i,i_0 \rangle^2 \right).
	\end{equation}

	Since $|\langle a,w \rangle|\leq \tau |w|$, for $w\in L_0$, and
	$-1/2-\sqrt{3}\tau\leq \langle i_1,i_2 \rangle\leq -1/2+\sqrt{3}\tau$, 
	$i_1,i_2\in \{b,c,d\}$, $i_1\neq i_2$, we get that
	$ \langle i,j \rangle- \langle a,i \rangle \langle a,j \rangle < 0$.
	By putting $e=(0,0,1)$, it is evident that
	\begin{equation}\label{eq:CLLHY50}
		\langle a,w \rangle^2\leq 1-\langle a,e \rangle^2, \text{ for any }w\in
		L_0 \text{ with }|w|=1.
	\end{equation}

	We put $N=\langle a,b \rangle^2+\langle a,c \rangle^2+\langle a,d \rangle^2$,
	and we claim that 
	\begin{equation}\label{eq:CLLHY60}
		N\leq  \left(3/2+25\tau \right) \left(1- \langle a,e \rangle^2\right).
	\end{equation}
	Indeed, for any $w=\lambda b + \mu c$ with $\lambda,\mu \geq 0$, we have that 
	\[
		|w|^2=\lambda^2+\mu^2+2 \lambda\mu \langle b,c \rangle\geq
		\lambda^2+\mu^2-(1+4\tau) \lambda\mu,
	\]
	\[
		\langle w,d \rangle^2\leq (1/2+\sqrt{3}\tau)^2(\lambda+\mu)^2\leq
		(1/4+2\tau)(\lambda+\mu)^2
	\]
	and 
	\[
		\begin{aligned}
			\langle w,b \rangle^2+\langle w,b \rangle^2+\langle w,b
			\rangle^2&=(\lambda^2+\mu^2)(1+\langle b,c \rangle^2)+4 \lambda\mu
			\langle b,c \rangle +\langle w,d \rangle^2\\
			&\leq
			\left(3/2+4\tau\right)(\lambda^2+\mu^2)-(3/2-10\tau)\lambda\mu\\
			&\leq (3/2+25\tau)|w|^2.
		\end{aligned}
	\]
	Hence, for any $w\in L_0$, we have that 
	\[
		\langle w,b \rangle^2+\langle w,b \rangle^2+\langle w,b \rangle^2
		\leq (3/2+25\tau)|w|^2,
	\]
	we now take $w=a-\langle a,e \rangle e$, then
	\[
		N\leq (3/2+25\tau)|a-\langle a,e \rangle e|^2=(3/2+25\tau)(1-\langle a,e
		\rangle^2),
	\]
	the claim \eqref{eq:CLLHY60} follows.

	Since $(1-x)^{1/2}\leq 1-x/2-x^2/8$ for any $x\in (0,1)$, and 
	\[
		(1-\langle a,b \rangle^2)(1-\langle a,c \rangle^2)(1-\langle a,d \rangle^2)
		\geq 1-N,
	\]
	we have that, for $\{i,j,k\}=\{b,c,d\}$, 
	\[
		\begin{aligned}
			\langle e_i,e_j \rangle
			&=\frac{\langle i,j \rangle- \langle a,i \rangle \langle a,j \rangle}
			{(1-\langle a,i \rangle^2)^{1/2}(1-\langle a,j \rangle^2)^{1/2}}\\
			&\geq \frac{(\langle i,j \rangle- \langle a,i \rangle \langle a,j
			\rangle)(1-\langle a,k \rangle^2/2-\langle a,k \rangle^4/8)} {(1-N)^{1/2}}.
		\end{aligned}
	\]
	Note that 
	\[
		\langle a,b \rangle^4+\langle a,c \rangle^4+\langle a,d \rangle^4\geq
		N^2/3,
	\]
	and 
	\[
		|\langle a,b+c+d \rangle|\leq \frac{1}{2}\left(|b+c+d|^2+1-\langle a,e
		\rangle^2\right),
	\]
	we get so that 
	\[
		\begin{aligned}
			|e_b+e_c+e_d|^2&\geq
			3+(1-N)^{-1/2}\Big(-3+(3/2-\sqrt{3}\tau)N+
			\frac{1}{12}(1/2-\sqrt{3}\tau)N^2\\
			&\quad + |b+c+d|^2-\langle a,b+c+d
			\rangle^2+\langle a,b \rangle\langle a,c \rangle\langle
			a,d \rangle\langle a,b+c+d\rangle\\
			&\quad +\frac{1}{4}\langle a,b \rangle\langle a,c \rangle\langle
			a,d \rangle\big( \langle a,b \rangle^3+\langle
			a,c \rangle^3+\langle a,d \rangle^3\big) \Big)\\
			&\geq (1-N)^{-1/2} \left((1-\tau^2)|b+c+d|^2-
			2\tau N-2\tau^3| \langle a,b+c+d \rangle|\right)\\
			&\geq (1-\tau)|b+c+d|^2-6\tau (1-\langle a,e \rangle^2).
		\end{aligned}
	\]

	Since $1/(1-x)= 1+x+x^2/(1-x)$ for $x\in [0,1)$, and 
	$\langle a,i \rangle^2\leq 1-\langle a,e \rangle^2$ for $i\in \{b,c,d\}$,
	we have that  
	\[
		\frac{\langle a,e \rangle^2}{1-\langle a,i \rangle^2}=\langle a,e
		\rangle^2+\frac{\langle a,e \rangle^2 \langle a,i \rangle^2}{1-\langle
		a,i \rangle^2}\leq \langle a,e \rangle^2 +\langle a,i \rangle^2
	\]
	and 
	\[
		\begin{aligned}
			\langle u_b,b_0 \rangle^2+\langle u_c,c_0 \rangle^2+
			\langle u_d,d_0 \rangle^2 
			&= \sum_{i\in \{b,c,d\}}  
			\frac{1-\langle a,e \rangle^2-\langle a,i \rangle^2}
			{1-\langle a,i \rangle^2}\\
			&=3(1- \langle a,e \rangle^2)-  N \\
			&\geq (1-\tau)(1-\langle a,e \rangle^2).
		\end{aligned}
	\]
	We get so that 
	\begin{equation}
		\HM^2(\varphi(X\cap B_1)) \leq \HM^2(X\cap B_1)-
		10^{-4}(1-10\tau)\left( |b+c+d|^2+1-\langle a,e \rangle^2 \right)
	\end{equation}

	Since $\arcsin x= x+\sum_{n\geq 1}C_nx^{2n+1}$ for $|x|\leq 1$, where 
	$C_n=\frac{(2n)!}{4^n(n!)^2(2n+1)}$, we have that $\arcsin \langle a,i
	\rangle \geq \langle a,i \rangle-\tau \langle a,i \rangle^2$, thus
	\begin{equation}\label{eq:CLLHY100}
		\begin{aligned}
			\HM^2(X\cap B_1)-\frac{3\pi}{4}&=-\frac{1}{2}\left(\arcsin \langle a,b
			\rangle+\arcsin \langle a,c \rangle+
			\arcsin \langle a,c \rangle\right)\\
			&\leq -\frac{1}{2}\langle a,b+c+d \rangle+\frac{\tau}{2}N\\
			&\leq \frac{1}{2} \left(|b+c+d|^2+ 1-\langle
			a,e\rangle^2 \right)+\tau \left(1- \langle a,e \rangle^2\right).
		\end{aligned}
	\end{equation}
	Thus 
	\[
		\HM^2(\varphi(X\cap B_1)) \leq(1- 10^{-4})
		\HM^2(X\cap B_1)-10^{-4}\cdot \frac{3\pi}{4}.
	\]
\end{proof}
Let $E\subseteq \Omega_0$ be a $2$-rectifiable
set satisfying \ref{AS1}, \ref{AS2} and \ref{AS3}.
We will denote by $\mathscr{R}_2$ the set 
\[
	\left\{ r\in \mathscr{R}_1 :
	\varepsilon(r)+j(r)^{1/2}\leq 10^{-6}(1-2 \cdot 10^{-4}) \right\}.
\]
\begin{lemma}\label{le:EsADB}
	For any $r\in (0,\mathfrak{r})\cap \mathscr{R}_2$, we have that
	\[
		\begin{aligned}
			\HM^2(E\cap B_r)&\leq (1-2\cdot 10^{-4})\frac{r}{2}\HM^1(E\cap \partial
			B_r)+(2\cdot 10^{-4}-\vartheta\kappa^2)\frac{r^2}{2}\HM^1(X\cap \partial
			B_1)\\
			&\quad +\vartheta\kappa^2r^2\Theta(0)+(2r)^2h(2r).
		\end{aligned}
	\]
\end{lemma}
\begin{proof}
	Let $\Sigma$, $\Sigma_r$, $\xi$, $\psi_{\xi}$, $\phi_{\xi}$ and $\{
	\varphi_t \}_{0\leq t\leq 1}$ be the same as in the proof of Lemma \ref{le:CotoCo}.
	We see that 
	\[
		\varphi_1(E\cap B(0,(1-\xi)r))=p(E\cap B(0,(1-\xi)r))\subseteq\Sigma_r,
	\]
	and that $\Sigma\cap B(0,2\kappa)=X\cap B(0,2\kappa)$,
	where $X$ is a cone defined in \eqref{eq:CoLip}. We see that if
	$\Theta(0)=\pi/2$, then $X$ satisfies the conditions in Lemma
	\ref{le:CLLHP}; if $\Theta(0)=3\pi/4$, then $X$ satisfies the conditions in 
	Lemma \ref{le:CLLHY}. Thus we can find a Lipschitz mapping $\Omega_0\to\Omega_0$
	with $\varphi(E\cap L)\subseteq L$, $|\varphi(z)|\leq 1$ when $|z|\leq 1$, and
	$\varphi(z)=z$ when $|z|>1$, such that
	\[
		\HM^2\left(\varphi(X)\cap \overline{B(0,1)}\right)\leq (1-\vartheta)\HM^2(X\cap
		B(0,1))+\vartheta\Theta(x).
	\]
	Let $\widetilde{\varphi}:\Omega_0\to\Omega_0$ be the mapping defined by
	$\widetilde{\varphi}(x)=r\varphi(x/r)$, then 
	\[
		\begin{aligned}
			\HM^2(E\cap B(0,r))&\leq \HM^2(\widetilde{\varphi}\circ\varphi_1(E)\cap
			\overline{B(0,r)})+(2r)^2h(2r)\\
			&\leq \HM^2(\widetilde{\varphi}\circ\varphi_1(E\cap B(0,(1-\xi)r)))+
			\HM^2(\varphi_1(E\cap A_{\xi}))\\
			&\leq \HM^2(\Sigma_r\setminus\overline{B(0,\kappa r)})+(1-\vartheta)(\kappa
			r)^2\HM^2(X\cap B(0,1))\\
			&\quad+\vartheta\cdot (\kappa r)^2\Theta(0)+\HM^2(\varphi_1(E\cap A_{\xi})).
		\end{aligned}
	\]
	But we see that $\Sigma_r=\{ rx:x\in \Sigma \}$, $\Sigma\cap
	B(0,2\kappa)=X\cap B(0,2\kappa)$, and 
	\[
		\lim_{\xi\to 0+}\HM^2(\varphi_1(E\cap A_{\xi}))\leq 2550\int_{E\cap\partial 
		B(0,r)}\dist(z,\Sigma_r)d\HM^1(z),
	\]
	we get so that 
	\[
		\HM^2(\Sigma_r\setminus \overline{B(0,\kappa
		r)})=r^2\left(\HM^2(\Sigma)-\HM^2(X\cap B(0,\kappa))\right),
	\]
	and  
	\[
		\begin{aligned}
			\HM^2(E\cap B(0,r))&\leq r^2\HM^2(\Sigma)-(\kappa r)^2\HM^2(X\cap
			B(0,1))\\
			&\quad+(1-\vartheta)(\kappa r)^2\HM^2(X\cap B(0,1))+(\kappa
			r)^2\vartheta\cdot\Theta(0)\\
			&\quad+2550\int_{E\cap\partial 
			B(0,r)}\dist(z,\Sigma_r)d\HM^1(z) +(2r)^2h(2r).
		\end{aligned}
	\]
	By \eqref{eq:CoToSu}, we get that 
	\[
		\begin{aligned}
			\HM^2(\Sigma)&\leq
			\HM^2(\mathcal{M})-10^{-4}(\HM^1(\Gamma_{\ast})-T)\\
			&= (1/2-10^{-4})\HM^1(\Gamma_{\ast})+ 10^{-4}\HM^1(X\cap \partial B(0,1)),
		\end{aligned}
	\]
	and then
	\[
		\begin{aligned}
			\HM^2(E\cap B_r)&\leq (1/2-10^{-4})r^2\HM^1(\Gamma_{\ast})+
			(10^{-4}-\vartheta\kappa^2/2)r^2\HM^1(X\cap \partial B_1)\\
			&\quad +\vartheta\kappa^2 r^2\Theta(0) +2550\int_{E\cap\partial 
			B_r}\dist(z,\Sigma_r)d\HM^1(z)+(2r)^2h(2r).
		\end{aligned}
	\]

	By \eqref{eq:XM} and Lemma \ref{le:EM}, we have that 
	\[
		d_{0,r}(E,\mathcal{M})\leq 5 \varepsilon(r)+ 10 j(r)^{1/2}.
	\]
	We get that for any $z\in E\cap \partial B(0,r)$, 
	\[
		\dist(\scale{1/r}(z),\mathcal{M})\leq 5\varepsilon(r)+10j(r)^{1/2}.
	\]
	Since $\Sigma\setminus B(0,9/10)=\mathcal{M}\setminus B(0,9/10)$, we have
	that 
	\[
		\begin{aligned}
			\dist(z,\Sigma_r)&=r\dist(\scale{1/r}(z),\Sigma)=
			r\dist(\scale{1/r}(z),\mathcal{M})\\
			&\leq 5r\varepsilon(r)+10rj(r)^{1/2}.
		\end{aligned}
	\]
	We get so that  
	\[
		\begin{aligned}
			\int_{E\cap\partial B(0,r)}\dist(z,\Sigma_r)d\HM^1(z)
			&\leq 5r(\varepsilon(r)+10j(r)^{1/2})
			\HM^1(E\cap\partial B(0,r)\setminus \Sigma_r)\\
			&\leq 10r(\varepsilon(r)+j(r)^{1/2})(\HM^1(E\cap\partial B_r)
			-r\HM^1( \Gamma_{\ast})). 
		\end{aligned}
	\]
	By Lemma \ref{le:SmCuAp}, we have that 
	\[
		\HM^1(\Gamma_{\ast}\setminus \Gamma)\leq \HM^1(\Gamma\setminus
		\Gamma_{\ast})\leq C\eta^2 (\HM^1(\Gamma)-\HM^1(X\cap \partial B(0,1))),
	\]
	so that 
	\[
		\HM^1(X\cap \partial B(0,1))\leq \HM^1(\Gamma_{\ast})\leq \HM^1(\Gamma)\leq
		\HM^1(\scale{1/r}(E\cap\partial B_r)),
	\]
	thus
	\[
		\begin{aligned}
			\HM^2(E\cap B_r)&\leq (1/2-10^{-4})r^2\HM^1(\Gamma_{\ast})+
			(10^{-4}-\vartheta\kappa^2/2)r^2\HM^1(X\cap \partial B_1)\\
			&\quad +10^5(\varepsilon(r)+j(r)^{1/2})r
			( \HM^1(E\cap\partial B_r)-r\HM^1(\Gamma_{\ast}))\\
			&\quad +\vartheta\kappa^2 r^2\Theta(0) +(2r)^2h(2r).
		\end{aligned}
	\]

	Since $r\in (0,\mathfrak{r})\cap \mathscr{R}_2$, we have that 
	\[
		10^5\left(\varepsilon(r)+10j(r)^{1/2}\right)\leq 
		\frac{1}{10}(1-2\cdot 10^{-4})
	\]
	thus
	\[
		\begin{aligned}
			\HM^2(E\cap B_r)&\leq (1-2\cdot 10^{-4})\frac{r}{2}\HM^1(E\cap \partial
			B_r)+(2\cdot 10^{-4}-\vartheta\kappa^2)\frac{r^2}{2}\HM^1(X\cap \partial
			B_1)\\
			&\quad +\vartheta\kappa^2r^2\Theta(0)+(2r)^2h(2r).
		\end{aligned}
	\]

\end{proof}
\begin{theorem}\label{thm:EsADB}
	There exist $\lambda,\mu\in (0,10^{-3})$ and $\mathfrak{r}_1>0$ such that, for any
	$0<r<\mathfrak{r}_1$,
	\[
		\HM^2(E\cap B_r)\leq (1-\mu-\lambda)\frac{r}{2}\HM^1(E\cap \partial
		B_r)+\mu \frac{r^2}{2}\HM^1(X\cap \partial B_1)+\lambda \Theta(0)r^2+4r^2h(2r).
	\]
\end{theorem}
\begin{proof}
	We put $\tau_1=\min\{ \tau_0,10^{-12}(1-\vartheta\kappa^2)^2\}$, and take
	$\delta$ such that 
	\begin{equation}\label{eq:cdelta1}
		\kappa<\delta<\kappa+(8\vartheta)^{-1}(1-2\cdot 10^{-4})\Theta(0)\tau_1.
	\end{equation}
	We see that $\varepsilon(r)\to 0$ as $r\to 0+$, there exist
	$\mathfrak{r}_1\in (0,\mathfrak{r})$ such that, for any $r\in
	(0,\mathfrak{r}_1)$,
	\begin{equation}\label{eq:ctau1}
		\varepsilon(r)\leq
		10^{-1}\min\{\tau_1,\vartheta(\delta^2-\kappa^2)\}. 
	\end{equation}

	If $r\in (0,\mathfrak{r}_1)$ and $j(r)\leq \tau_1$, then $r\in
	\mathscr{R}_2$, then by Lemma \ref{le:EsADB}, we have that 
	\[
		\begin{aligned}
			\HM^2(E\cap B_r)&\leq (1-2\cdot 10^{-4})\frac{r}{2}\HM^1(E\cap \partial
			B_r)+(2\cdot 10^{-4}-\vartheta\kappa^2)\frac{r^2}{2}\HM^1(X\cap \partial
			B_1)\\
			&\quad +\vartheta\kappa^2r^2\Theta(0)+(2r)^2h(2r).
		\end{aligned}
	\]
	We only need to consider the case $r\in (0,\mathfrak{r}_1)$,
	$j(r)>\tau_1$ and $\HM^1(E\cap \partial B_r)<+\infty$, thus 
	\begin{equation}\label{eq:AppEBr}
		\HM^1(X\cap \partial B_1)+\tau_1\leq \frac{1}{r}\HM^1(E\cap B(0,r)).
	\end{equation}

	By the construction of $X$, we see that $X\cap B(0,1)$ is Lipschitz
	neighborhood retract, let $U$ be a neighborhood of $X\cap B(0,1)$ and
	$\varphi_0:U\to X\cap B(0,1)$ be a retraction such that $|\varphi_0(x)-x|\leq
	r/2$. We put $U_1=\scale{8r/9}(U)$, $\varphi_1=\scale{8r/9}\circ 
	\varphi_0\circ \scale{9/(8r)}$, and let $s:[0,\infty)\to [0,1]$ be a function
	given by 
	\[
		s(t)=\begin{cases}
			1,&0\leq t\leq 3r/4,\\
			-(8/r)(t-7r/8),& 3r/4<t \leq 7r/8,\\
			0,&t>7r/8.
		\end{cases}
	\]

	We see that there exist sliding minimal cone $Z$ such that $d_{0,1}(X,Z)\leq
	\varepsilon(r)$, thus $d_{0,r}(E,X)\leq 2\varepsilon(r)$, then for any
	$x\in E\cap B(0,r)\setminus B(0,3r/4)$, 
	\[
		\dist(x,X)\leq 2\varepsilon(r) r\leq \frac{8\varepsilon(r)}{3}|x|.
	\]
	We consider the mapping $\psi:\Omega_0\to\Omega_0$ defined by 
	\[
		\psi(x)=s(|x|)\varphi_1(x)+(1-s(|x|))x,
	\]
	then $\psi(L)=L$ and $\psi(x)=x$ for $|x|\geq 8r/9$.

	We take $\mathfrak{r}_1>0$ such that, for any $r\in (0,\mathfrak{r}_1)$,
	\[
		\{ x\in \Omega_0\cap B(0,1):\dist(x,X)\leq 3\varepsilon(r) \}\subseteq  U.
	\]
	Then we get that $\psi(x)\in X$ for any $x\in E\cap B(0,3r/4)$; 
	\[
		\dist(\psi(x),X)\leq 3\varepsilon(r)|x|\text{ for any }x\in E\cap
		B(0,r)\setminus B(0,3r/4);
	\]
	and $\Psi(E\cap B_r)\cap B(0,r/4)=X\cap B(0,r/4)$.

	We now consider the mapping $\Pi_1:\Omega_0\to\Omega_0$ defined by
	\[
		\Pi_1(x)=s(4|x|)x+(1-s(4|x|))\Pi(x),
	\]
	and the mapping $\psi_1:\Omega_0\to\Omega_0$ defined by
	\[
		\psi_1(x)=\begin{cases}
			\Pi_1\circ\psi(x),& |x|\leq r,\\
			x,&|x|\geq r.
		\end{cases}
	\]
	We have that $\psi_1$ is Lipschitz, $\psi_1(L_0)=L_0$ and $\psi_1(B(0,r))\subseteq
	\overline{B(0,r)}$, 
	\[
		\psi_1(E\cap B(0,r))\subseteq X\cap B(0,r)\cup \{ x\in \partial B_r:
		\dist(x,X) \leq 3r\varepsilon(r)\}.
	\]

	Let $\varphi$ be the same as in Lemma \ref{le:CLLHP} and Lemma \ref{le:CLLHY}, and
	let $\psi_2=\scale{\delta}\circ \varphi\circ \scale{1/\delta}\circ \psi_1$.
	Then we have that 
	\begin{equation}\label{eq:AppEBr10}
		\begin{aligned}
			\HM^2(E\cap \overline{B(0,r)})&\leq \HM^2(\psi_2(E\cap
			\overline{B(0,r)}))+(2r)^2h(2r)\\
			&\leq (1-\vartheta\delta^2)\HM^2(X\cap B(0,r))+\vartheta\delta^2
			\Theta(0)r^2+4r^2h(2r)\\
			&\quad +\HM^2(\{ x\in \partial B_r:
			\dist(x,X) \leq 3r\varepsilon(r)\})\\
			&\leq (1-\vartheta\delta^2)\HM^2(X\cap B(0,r))+\vartheta\delta^2
			\Theta(0)r^2 \\
			&\quad + 4r\varepsilon(r)\HM^1(X\cap \partial B_r)+4r^2h(2r)\\
			&\leq (1-\vartheta\delta^2+8\varepsilon(r))\frac{r^2}{2}\HM^1(X\cap \partial 
			B_1)+\vartheta\delta^2 \Theta(0)r^2 +4r^2h(2r)\\
		\end{aligned}
	\end{equation}

	We take $\mu=2\cdot 10^{-4}-\vartheta\kappa^2$ and $\lambda=\vartheta\kappa^2$,
	then by \eqref{eq:cdelta1} and \eqref{eq:ctau1}, we have that   
	\[
		8\varepsilon(r)<\vartheta(\delta^2-\kappa^2)
	\]
	and 
	\[
		\vartheta(\delta^2-\kappa^2)\Theta(0)\leq (1-2\cdot 10^{-4})\frac{\tau_1}{2}.
	\]
	We get from \eqref{eq:AppEBr} and \eqref{eq:AppEBr10} that 
	\[
		\begin{aligned}
			\HM^2(E\cap \overline{B_r})&\leq (1-2\cdot 10^{-4})
			\frac{r^2}{2}(\HM^1(X\cap \partial B_1)+\tau_1)-(1-2\cdot
			10^{-4})\frac{\tau_1r^2}{2}\\
			&\quad +\mu \frac{r^2}{2}\HM^1(X\cap \partial B_1)+ 
			\vartheta\kappa^2\Theta(0)r^2 +4r^2h(2r)\\
			&\quad +(8\varepsilon(r)-\vartheta\delta^2+\vartheta\kappa^2)\frac{r^2}{2}\HM^1(X\cap \partial
			B_1)+(\vartheta\delta^2-\vartheta\kappa^2)\Theta(0)r^2\\
			&\leq (1-\lambda-\mu)\frac{r}{2}\HM^1(E\cap \partial B_r)+\mu \frac{r^2}{2}\HM^1(X\cap \partial B_1)+ 
			\lambda\Theta(0)r^2 +4r^2h(2r).
		\end{aligned}
	\]

\end{proof}
For convenient, we put $\lambda_0=\lambda/(1-\lambda)$,
$f(r)=\Theta(0,r)-\Theta(0)$ and $u(r)=\HM^1(E\cap B(0,r))$ for $r>0$.
Since $f(r)=r^{-2}u(r)-\Theta(0)$ and $u$ is a nondecreasing function, we have
that, for any $\lambda_1\in \mathbb{R}$ and $0<r\leq R<+\infty$, 
\[
	R^{\lambda_1}f(R)-r^{\lambda_1}f(r)\geq \int_{r}^{R}\left(
	t^{\lambda_1}f(t) \right)' dt,
\]
thus
\begin{equation}\label{eq:refd}
	f(r)\leq r^{-\lambda_1}R^{\lambda_1}f(R)+r^{-\lambda_1}\int_{r}^{R}\left(
	t^{\lambda_1}f(t) \right)' dt.
\end{equation}
\begin{corollary}\label{co:dendecay}
	If the gauge function $h$ satisfy 
	\[
		h(t)\leq C_ht^{\alpha},\ 0<t\leq \mathfrak{r}_1\text{ for some }C_h>0,\
		\alpha>0,
	\]
	then for any $0<\beta<\min\{ \alpha, 2\lambda_0\}$, there is a constant
	$C=C(\lambda_0,\alpha,\beta,\mathfrak{r}_1,C_h)>0$ such that  
	\begin{equation}\label{eq:dendecay}
		|\Theta(0,\rho)-\Theta(0)|\leq C\rho^{\beta}
	\end{equation}
	for any $0<\rho\leq \mathfrak{r}_1$.
\end{corollary}
\begin{proof}
	For any $r>0$, we put $u(r)=\HM^2(E\cap B(0,r))$. Then $u$ is differentiable
	for $\HM^1$-a.e. $r\in (0,\infty)$.

	By Theorem \ref{thm:EsADB} and Lemma \ref{le:DE}, we have that for any
	$r\in (0,\mathfrak{r}_1)\cap \mathscr{R}$,  
	\[
		\begin{aligned}
			u(r)&\leq (1-\lambda)\frac{r}{2}\HM^1(E\cap \partial B(0,r))+\lambda
			\Theta(0)r^2+4r^2h(2r)\\
			&\leq (1-\lambda)\frac{r}{2}u'(r)+\lambda \Theta(0)r^2+4r^2h(2r),
		\end{aligned}
	\]
	thus 
	\[
		rf'(r)\geq
		\frac{2\lambda}{1-\lambda}f(r)-\frac{8}{1-\lambda}h(2r)=2\lambda_0
		f(r)-8(1+\lambda_0)h(2r),
	\]
	and
	\[
		\left( r^{-2\lambda_0}f(r) \right)'=
		r^{-1-2\lambda_0}\left( rf'(r)-2\lambda_0 \right)\geq
		-8(1+\lambda_0)r^{-1-2\lambda_0}h(2r).
	\]

	Recall that $\HM^1((0,\infty)\setminus \mathscr{R})=0$.
	We get, from \eqref{eq:refd}, so that, for any $0<r<R\leq \mathfrak{r}_1$,
	\begin{equation}\label{eq:denapp10}
		f(r)\leq r^{2\lambda_0}R^{-2\lambda_0}f(R)
		+8(1+\lambda_0)r^{2\lambda_0}
		\int_r^{R}t^{-1-2\lambda_0}h(2t)dt.
	\end{equation}
	Since $h(t)\leq C_ht^{\alpha}$, we have that 
	\[
		f(r)\leq (r/R)^{-2\lambda_0}f(R)
		+2^{3+\alpha}(1+\lambda_0)C_hr^{2\lambda_0}
		\int_r^{R}t^{\alpha-2\lambda_0-1}dt.
	\]
	If $\alpha>2\lambda_0$, then 
	\begin{equation}\label{eq:denapp20}
		f(r)\leq \left( f(R) +2^{3+\alpha}(1+\lambda_0)(1+\lambda_0)
		(\alpha-2\lambda_0)^{-1}C_h 
		R^{\alpha}\right)(r/R)^{2\lambda_0};
	\end{equation}
	if $\alpha=2\lambda_0$, then
	\[
		f(r)\leq f(R)(r/R)^{\alpha} +2^{\alpha+3}(1+\lambda_0)C_h 
		r^{\alpha} \ln (R/r),
	\]
	thus, for any $\beta\in (0,\alpha)$, 
	\begin{equation}\label{eq:denapp30}
		\begin{aligned}
			f(r)&\leq f(R)r^{\alpha} +2^{\alpha+3}(1+\lambda_0)C_h 
			r^{\beta}R^{\alpha-\beta} \frac{\ln (R/r)}{(R/r)^{\alpha-\beta}}\\
			&\leq \left(f(R) + 2^{\alpha+3}(1+\lambda_0)C_h
			(\alpha-\beta)^{-1}e^{-1} R^{\alpha} \right)(r/R)^{\beta};
		\end{aligned}
	\end{equation}
	if $\alpha<2\lambda_0$, then
	\begin{equation}\label{eq:denapp40}
		\begin{aligned}
			f(r)&\leq f(R)(r/R)^{2\lambda_0}
			+2^{\alpha+3}(1-\lambda_0)C_h r^{2\lambda_0} \cdot
			(2\lambda_0-\alpha)^{-1}\left( r^{\alpha-2\lambda_0} -
			R^{\alpha-2\lambda_0}\right)\\
			&\leq \left((r/R)^{2\lambda_0-\alpha}f(R) + 
			2^{\alpha+3}(1-\lambda_0)C_h (2\lambda_0-\alpha)^{-1}R^{\alpha}
			\right)(r/R)^{\alpha}.
		\end{aligned}
	\end{equation}
	Hence \eqref{eq:dendecay} follows from \eqref{eq:denapp20},
	\eqref{eq:denapp30}, \eqref{eq:denapp40} and Theorem \ref{thm:ANDD}. 
	Indeed, there is a constant
	$C_1(\alpha,\beta,\lambda_0)>0$ such that 
	\begin{equation}\label{eq:denapp50}
		r^{2\lambda_0}\int_{r}^{R}t^{\alpha-2\lambda_0-1}dt\leq 
		C_1(\alpha,\beta,\lambda_0) R^{\alpha}\cdot (r/R)^{\beta},
	\end{equation}
	and there is a constant $C_2(\alpha,\beta,\lambda_0)>0$ such that 
	\[
		f(r)\leq \left( f(R)+C_2(\alpha,\beta,\lambda_0)C_h \cdot R^{\alpha}
		\right)(r/R)^{\alpha}.
	\]
\end{proof}
\begin{remark}\label{re:largegaugefun}
	If the gauge function $h$ satisfy that 
	\[
		h(t)\leq C\left( \ln\left( \frac{A}{t} \right) \right)^{-b}
	\]
	for some $A,b,C>0$, then \eqref{eq:denapp10} implies that there exist $R>0$
	and constant $C(R,\lambda,b)$ such that 
	\[
		f(r)\leq C(R,\lambda,b)\left( \ln \left( \frac{A}{r} \right)
		\right)^{-b}\text{ for }0<r\leq R.
	\]
\end{remark}

\section{Approximation of $E$ by cones at the boundary}
In this section, we also assume that $E\subseteq \Omega_0$ is a $2$-rectifiable
set satisfying \ref{AS1}, \ref{AS2} and \ref{AS3}.
We let $\varepsilon(r)=\varepsilon_P(r)$ if $E$
is locally $C^0$-equivalent to a sliding minimal cone of type $\mathbb{P}_+$;
and let $\varepsilon(r)=\varepsilon_Y(r)$ if $E$ is locally $C^0$-equivalent 
to a sliding minimal cone of type $\mathbb{Y}_+$.

For any $r>0$, we put 
\[
	f(r)=\Theta(0,r)-\Theta(0),\ F(r)=f(r)+8h_1(r),\ F_1(r)=F(r)+8h_1(r),
\]
and for $r\in \mathscr{R}$, we put
\[
	\Xi(r)=rf'(r)+2f(r)+16h(2r)+32h_1(r).
\]

We denote by $X(r)$ and $\Gamma(r)$,
respectively, the cone $X$ and the set $\Gamma$ which are defined in
\eqref{eq:CoLip}, and by $\gamma(r)$ the set $\scale{r}(\Gamma(r))$.  
For any $r_2>r_1>0$, we put 
\[
	A(r_1,r_2)=\{ x\in \mathbb{R}^{3}: r_1\leq |x|\leq r_2 \}.
\]
\begin{lemma}\label{le:eap}
	For any $0<r<R<\infty$ with $\HM^2(E\cap \partial B_r)=\HM^2(E\cap \partial
	B_R)=0$, we have that 
	\begin{equation}\label{eq:eap1}
		\int_{E\cap A(r,R)} \frac{1-\cos\theta(x)}{|x|^2}d\HM^2(x)\leq F(R)-F(r),
	\end{equation}
	and 
	\begin{equation}\label{eq:eap2}
		\HM^2\left( \Pi(E\cap A(r,R)) \right)\leq \int_{E\cap
		A(r,R)}\frac{\sin\theta(x)}{|x|^2}d\HM^2(x).
	\end{equation}
\end{lemma}
\begin{proof}
	We see that for $\HM^2$-a.e. $x\in E$, the tangent plane $\Tan(E,x)$
	exists, we will denote by $\theta(x)$, the angle between the line $[0,x]$ 
	and the plane $\Tan(E,x)$. For any $t>0$, we put $u(t)=\HM^2(E\cap
	B(0,t))$, then $u:(0,\infty)\to [0,\infty]$ is a nondecreasing function. By
	Lemma \ref{le:DI}, we have that 
	\[
		u(t)\leq \frac{t}{2}\HM^1(E\cap \partial B(0,t))+4t^2h(2t),
	\]
	for
	$\HM^1$-a.e. $t\in (0,\infty)$. Considering the mapping
	$\phi:\mathbb{R}^{3}\to
	[0,\infty)$ given by $\phi(x)=|x|$, we have, by \eqref{eq:DE10}, that
	\[
		\ap J_1(\phi|_{E})(x)=\cos\theta(x)
	\]
	for $\HM^2$-a.e. $x\in E$.

	Apply Theorem 3.2.22 in \cite{Federer:1969}, we get that
	\[
		\begin{aligned}
			&\int_{E\cap A(r,R)}\frac{1}{|x|^2}\cos\theta(x)d\HM^2(x)
			=\int_{r}^{R}\frac{1}{t^2}\HM^1(E\cap \partial B(0,t)d dt\\
			&\geq 2\int_{r}^{R}\frac{u(t)}{t^3} d t-8\int_{r}^{R}\frac{h(2t)}{t}dt\\
			&=2\int_{r}^{R}\frac{1}{t^3}\int_{E\cap B(0,t)}d\HM^2(x) dt
			-8(h_1(R)-h_1(r))\\
			&=2\int_{E\cap B(0,R)}\int_{\max\{ r,|x| \}}^{R}\frac{1}{t^3} dt d\HM^2(x)
			-8(h_1(R)-h_1(r))\\
			&=\int_{E\cap A(r,R)}\frac{1}{|x|^2}d\HM^2(x)+r^{-2}u(r)-R^{-2}u(R)-
			8(h_1(R)-h_1(r)),
		\end{aligned}
	\]
	thus \eqref{eq:eap1} holds.

	By a simple computation, we get that 
	\[
		\ap J_2\Pi(x)=\frac{\sin\theta(x)}{|x|^2},
	\]
	we now apply Theorem 3.2.22 in \cite{Federer:1969} to get \eqref{eq:eap2}.
\end{proof}
We get from above Lemma that 
\[
	\HM^2(\Pi(E\cap A(r,R)))\leq \frac{r_2}{r_1}\left(2\Theta(0,R)
	\right)^{1/2}\left( F(R)-F(r) \right)^{1/2}
\]

\begin{lemma}\label{le:HDCC}
	For any $r\in (0,\mathfrak{r}_1)\cap \mathscr{R}$, if $\Xi(r)\leq
	\mu\tau_0$, then 
	\[
		d_H(\Gamma(r), X(r)\cap \partial B(0,1))\leq 10\mu^{-1/2}\Xi(r)^{1/2}. 
	\]
\end{lemma}
\begin{proof}	
	By lemma \ref{le:DE}, we get that 
	\[
		\frac{1}{r}\HM^1(E\cap \partial B(0,r))\leq
		2\Theta(0)+rf'(r)+2f(r),
	\]
	By Theorem \ref{thm:EsADB}, we get that 
	\[ 
		\begin{aligned}
			r^2\Theta(0,r)&\leq (1-\lambda-\mu)\frac{r}{2}\HM^1(E\cap \partial
			B_r)+\mu\frac{r^2}{2}\HM^1(X\cap \partial B_1)+\lambda\Theta(0)r^2+4r^2h(2r)\\
			&\leq \frac{1}{2}(1-\lambda-\mu)r^2(2\Theta(0)+rf'(r)+2f(r))+
			\mu\frac{r^2}{2}\HM^1(X\cap \partial B_1)\\
			&\quad+\lambda\Theta(0)r^2+4r^2h(2r),
		\end{aligned}
	\]
	thus 
	\[
		\HM^1(X\cap \partial B_1)\geq
		2\Theta(0)+\frac{2(\lambda+\mu)}{\mu}f(r)-\frac{1-\lambda-\mu}{\mu}rf'(r)-
		\frac{\mu}{8}h(2r).
	\]
	Hence  
	\begin{equation}\label{eq:HDCC5}
		\begin{aligned}
			j(r)&=\frac{1}{r}\HM^1(E\cap B_r)-\HM^1(X\cap \partial B_1)\\
			&\leq \frac{1-\lambda}{\mu}rf'(r)-\frac{2\lambda}{\mu}f(r)+
			\frac{8}{\mu}h(2r)\\
			&\leq \frac{1}{\mu}(rf'(r)+16h_1(r)+16h(2r)).
		\end{aligned}
	\end{equation}
	Since 
	\[
		\HM^1(X\cap \partial B_1)\leq \HM^1(\Gamma_{\ast}(r))\leq
		\HM^1(\Gamma(r))\leq \HM^1(\scale{1/r}(E\cap \partial B_r)),
	\]
	we have that 
	\[
		0\leq \HM^1(\Gamma(r))-\HM^1(X\cap B_1)\leq j(r)\leq \frac{1}{\mu}\Xi(r),
	\]
	by Lemma \ref{le:smallh}, we get that for any $z\in \Gamma(r)$, 
	\[
		\dist\left( z, X\cap \partial B(0,1)\right)\leq
		10 \left(\frac{\Xi(r)}{\mu}\right)^{1/2}.
	\]

\end{proof}
\begin{lemma}\label{le:cinp}
	For any $0<r_1<r_2< (1-\tau)\mathfrak{r}$, if $P$ is a plane
	such that $\HM^1(E\cap P\cap B_{\mathfrak{r}})<\infty$ and $P\cap
	\mathcal{X}_r=\emptyset$ for any $r\in [r_1,r_2]$, then there is a compact 
	path connected set 
	\[
		\mathcal{C}_{P,r_1,r_2}\subseteq E\cap P\cap A(r_2,r_1)
	\]
	such that 
	\[
		\mathcal{C}_{P,r_1,r_2}\cap \gamma(t)\neq \emptyset\text{ for }r_1\leq
		t\leq r_2.
	\]
\end{lemma}
\begin{proof}
	We let $\varrho$ be the same as in \ref{eq:apps}.
	Since $\|\Phi-\id\|_{\infty}\leq \tau\varrho$, we get that 
	\[
		\Phi^{-1}\left( E\cap \overline{B(0,r_2)} \right)\subseteq 
		Z_{0,\varrho}\cap \overline{B(0,r_2+\tau\varrho)}.
	\]
	We put 
	\[
		\begin{gathered}
			\mathbb{X}=Z_{0,\varrho}\cap \overline{B(0,r_2+\tau\varrho)},\\ 
			F=\mathbb{X}\cap \Phi^{-1}(E\cap P_z).
		\end{gathered}
	\]
	We take $x_1,x_2\in \mathcal{X}_r$, $x_2\neq x_1$, such that $\Phi^{-1}(x_1)$
	and $\Phi^{-1}(x_2)$ are contained in two different connected components of
	$\mathbb{X}\setminus F$. By Lemma \ref{le:CoSe}, there is a connected closed
	subset $F_0$ of $F$ such that $\Phi^{-1}(x)$ and $\Phi^{-1}(x_2)$ are still
	contained in two different connected components of $\mathbb{X}\setminus F_0$.
	Then $F_0\cap \phi^{-1}(\gamma(t))\neq \emptyset$ for $0<t\leq r_2$;
	otherwise, if $F_0\cap \phi^{-1}(\gamma(t_0))=\emptyset$, then $x_1$ and
	$x_2$ are in the same connected component of $\Phi(\mathbb{X})\setminus
	\Phi(F_0)$, thus $\Phi^{-1}(x_1)$ and $\Phi^{-1}(x_2)$ are in the same
	connected component of $\mathbb{X}\setminus F_0$, absurd! 

	Since $\HM^1(\Phi(F_0))\leq \HM^1(E\cap P_z\cap B_{\varrho})<\infty$, we
	get that $\Phi(F_0)$ is path connected. We take $z_1\in
	\Phi(F_0)\cap\gamma(r_1)$ and $z_2\in \Phi(F_0)\cap \gamma(r_2)$, and let
	$g:[0,1]\to \Phi(F_0)$ be a path such that $g(0)=z_1$ and $g(1)=z_2$. We
	take $t_1=\sup\{ t\in [0,1]: |g(t)|\leq r_1 \}$ and $t_2=\inf\{ t\in
	[t_1,1]:|g(t)|\geq r_2 \}$. Then $\mathcal{C}_{z,r_1,r_2}=g([t_1,t_2])$ is
	our desire set. 
\end{proof}
\begin{lemma}\label{le:smallband}
	Let $T\in [\pi/4,3\pi/4]$ and $\varepsilon\in (0,1/2)$ be given. Suppose that
	$F$ a 2-rectifiable set satisfying 
	\[
		F\subseteq \partial B(0,1)\cap \{ (t\cos\theta,t\sin\theta,x_3)\in
		\mathbb{R}^{3}\mid t\geq 0, |\theta| \leq T/2, |x_3|\leq \varepsilon \}.
	\]
	Then we have, by putting $\mathcal{P}_{\theta}=\{ (t\cos\theta,t\sin\theta, x_3)\mid
	t\geq 0, x_3\in \mathbb{R} \}$, that 
	\[
		\int_{-T/2}^{T/2}\HM^1(F\cap \mathcal{P}_{\theta})d\theta \leq 
		(1+\varepsilon) \HM^2(F)
	\]
\end{lemma}
\begin{proof}
	For any $x=(x_1,x_2,x_3)\in F$, we have that $x_1^2+x_2^2+x_3^2=1$ and
	$|x_3|\leq \varepsilon$, thus $x_1^2+x_2^2\geq 1-\varepsilon^2$. Since
	$|\theta|\leq T/2\leq 3\pi/8$, we get that the mapping $\phi:F\to
	\mathbb{R}$ given by 
	\[
		\phi(x_1,x_2,x_3)=\arctan \frac{x_2}{x_1}
	\]
	is well defined and Lipschitz. Moreover, we have that 
	\[
		\ap J_1\phi(x)=(x_1^2+x_2^2)^{-1/2}\leq (1-\varepsilon^2)^{-1/2}\leq
		1+\varepsilon.
	\]
	Hence 
	\[
		\int_{-T/2}^{T/2}\HM^1(F\cap \mathcal{P}_{\theta})d\theta=\int_{F}\ap
		J_1\phi(x)d\HM^2(x)\leq (1+\varepsilon)\HM^2(F).
	\]
\end{proof}

For any $0< t_1\leq t_2$, we put
\[
	E_{t_1,t_2}=\Pi\left(\{ x\in E:t_1\leq |x|\leq t_2 \}\right).
\]
For any $t>0$, we put 
\[
	\bar{\varepsilon}(t)=\sup\{ \varepsilon(r):r\leq t \}.
\]
\begin{lemma}
	If $r_2> r_1>0$ satisfy that $ 10(1+r_2/r_1)\bar{\varepsilon}(r_2)<1/2$,
	then we have that 
	\begin{equation}\label{le:Essa}
		\int_{X(t)\cap \partial B(0,1)}\HM^1\left( P_z\cap E_{r_1,r_2}\right)
		d\HM^1(z)\leq 2\HM^2\left(E_{r_1,r_2}\right),\ \forall r_1\leq t\leq r_2.
	\end{equation}
\end{lemma}
\begin{proof}
	By Lemma \ref{le:EM}, we have that, for any $r>0$, if $\varepsilon(r)<1/2$,
	then 
	\[
		d_{0,r}(E,X(r))\leq 5\varepsilon(r).
	\]
	We get so that 
	\[
		\begin{aligned}
			d_{0,1}(X(t),X(r_2))&=d_{0,t}(X(t),X(r_2))\leq d_{0,t}(E,X(t))+d_{0,t}(E,X(r_2))\\
			&\leq5\bar{\varepsilon}(r_2)+5\frac{r_2}{t}\bar{\varepsilon}(r_2).
		\end{aligned}
	\]
	Since  
	\[
		\dist(x,X(r_2))\leq 5r_2\varepsilon(r_2),
		\text{ for any }x\in E\cap B(0,r_2),
	\]
	we have that 
	\[
		\dist(\Pi(x),X(r_2))\leq \frac{5r_2\varepsilon(r_2)}{|x|}, 
		\text{ for any }x\in E\cap A(r_1,r_2),
	\]
	we get so that 
	\[
		\dist(\Pi(x),X(t))\leq
		\frac{5r_2\varepsilon(r_2)}{|x|}+5\bar{\varepsilon}(r_2)+5\frac{r_2}{t}\bar{\varepsilon}(r_2)\leq
		10(r_2/r_1+1)\bar{\varepsilon}(r_2)<\frac{1}{2}.
	\]
	We now apply Lemma \ref{le:smallband} to get the result.
\end{proof}

\begin{lemma}
	Let $\varepsilon\in (0,1/2)$ be given. Let $A\subseteq \partial B(0,1)$ be
	an arc of a great circle such that $0<\HM^1(A)\leq \pi$ and  
	\[
		\dist(x,L_0)\leq \varepsilon, \forall x\in A.
	\]
	Then   
	\[
		\dist(x,L_0)\leq \frac{\pi^2}{2\HM^1(A)^2}\int_{A}\dist(x,L_0)d\HM^1(x),\ 
		\forall x\in A.
	\]
\end{lemma}
\begin{proof}
	We let $P$ be the plane such that $A\subseteq P$, let $v_0\in P\cap L_0 \cap
	\partial B(0,1)$ and $v_2\in P\cap \partial B(0,1)$ be two vectors such that
	$v_0$ is perpendicular to $v_1$. Then $A$ can be parametrized as
	$\gamma:[\theta_1,\theta_2]\to A$ given by 
	\[
		\gamma(t)=v_0\cos t+ v_1\sin t,
	\]
	where $\theta_2-\theta_1=\HM^1(A)$. We write $v_1=w+w^{\perp}$ with $w\in
	L_0$ and $w^{\perp}$ perpendicular to $L_0$. Since $\ap J_{1}\gamma(t)=1$ 
	for any $t\in [\theta_1,\theta_2]$, by Theorem 3.2.22 in \cite{Federer:1969}, 
	we have that 
	\[
		\begin{aligned}
			\int_{A}\dist(x,L_0)\HM^1(x)&=\int_{\theta_1}^{\theta_2}\dist(\gamma(t),L_0)dt=
			\int_{\theta_1}^{\theta_2}|w^{\perp}\sin t|dt\\
			&\geq 2|w^{\perp}|\left(1-\cos \frac{\theta_2-\theta_1}{2}\right)\geq
			\frac{2(\theta_2-\theta_1)^2}{\pi^2}|w^{\perp}|,
		\end{aligned}
	\]
	and that 
	\[
		\dist(x,L_0)\leq |w^{\perp}|\leq \frac{\pi^2}{2\HM^1(A)^2}\int_{A}
		\dist(x,L_0)d\HM^1(x).
	\]
\end{proof}

\begin{lemma}\label{le:CompCon}
	Let $r_1$ and $r_2$ be the same as in Lemma \ref{le:cinp}. If $\Xi(r_i)\leq
	\mu\tau_0$, $10(1+r_2/r_1)\bar{\varepsilon}(r_2)\leq 1$, then
	we have that 
	\[
		d_{0,1}(X(r_1),X(r_2))\leq \frac{30r_2}{r_1}\Theta(0,r_2)^{1/2}\cdot 
		F(r_2)^{1/2}+ 20\pi\mu^{-1/2}\cdot \left( \Xi(r_1)^{1/2}+\Xi(r_2)^{1/2}
		\right).
	\]
\end{lemma}
\begin{proof}
	For $z\in X(r_2)\cap \partial B_1$, if $z\notin \{ y_r \}\cup
	\mathcal{X}_r$, we will denote by $P_z$ the plane which is
	through $0$ and $z$ and perpendicular to $\Tan(X(r_2)\cap \partial B_1, z)$.
	By Lemma \ref{le:HDCC}, we have that
	\[
		|z-a|\leq 10\mu^{-1/2}\Xi(r_1)^{1/2},\forall a\in \Gamma(r_2)\cap P_z.
	\]
	Since $\mathcal{C}_{P_z,r_1,r_2}\cap \gamma(r_i)\neq \emptyset$, $i=1,2$,
	we take $b_i\in\mathcal{C}_{P_z,r_1,r_2}\cap \gamma(r_i)$, then 
	\[
		|\Pi(b_1)-\Pi(b_2)|\leq \HM^1(\Pi(\mathcal{C}_{P_z,r_1,r_2}))\leq 
		\HM^1(P_z\cap E_{r_1,r_2}), 
	\]
	thus 
	\[
		\begin{aligned}
			\dist(z,X(r_1)\cap \partial B_1)&\leq
			|z-\Pi(b_2)|+|\Pi(b_2)-\Pi(b_1)|+\dist(\Pi(b_1),X(r_1)\cap \partial B_1)\\ 
			&\leq \HM^1(P_z\cap E_{r_1,r_2})+10\mu^{-1/2}\left( \Xi(r_1)^{1/2}+
			\Xi(r_2)^{1/2} \right).
		\end{aligned}
	\]
	For any $x\in \mathcal{X}_r$, we
	let $A_x$ be the arc in $\partial B(0,1)$ which join $\Pi(x)$ and $\Pi(y_r)$,
	We see that $X(r_2)\cap \partial B(0,1)=\cup_{x\in \mathcal{X}_r}A_x$, and 
	$\HM^1(A_x)\geq (1/2-\bar{\varepsilon}(r_2))\pi\geq \pi/4$. Suppose $z\in
	A_x$, then 
	\[
		\begin{aligned}
			\dist(z,X(r_1))&\leq
			\frac{\pi^2}{2\HM^1(A_x)^2}\int_{A_x}\dist(z,X(r_1))d\HM^1(x)\\
			&\leq \frac{2\pi}{\HM^1(A_x)} \int_{A_x}\HM^1(P_z\cap E_{r_1,r_2})d\HM^1(x)+
			20\pi\mu^{-1/2}\left( \Xi(r_1)^{1/2}+ \Xi(r_2)^{1/2} \right)\\
			&\leq 16 \HM^2(E_{r_1,r_2})+20\pi\mu^{-1/2}\left( \Xi(r_1)^{1/2}+ 
			\Xi(r_2)^{1/2} \right)\\
			&\leq \frac{16r_2}{r_1}\left(2\Theta(0,r_2)
			\right)^{1/2}F(r_2)^{1/2}+
			20\pi\mu^{-1/2}\left( \Xi(r_1)^{1/2}+ \Xi(r_2)^{1/2} \right)
		\end{aligned}
	\]
\end{proof}
\begin{remark}
	For any cones $X_1$ and $X_2$, we see that 
	\[
		d_H(X_1\cap \partial B(0,1),X_2\cap \partial B(0,1))\leq 
		2d_{0,1}(X_1,X_2).
	\]
\end{remark}

Since $\Xi(r)=[rF_1(r)]'$ for any $r\in \mathscr{R}$, we get that  
\[
	\int_{r_1}^{r_2}\Xi(t)dt\leq r_2F_1(r_2)-r_1F_1(r_1),
\]
For any $\zeta>2$, if $r_1\leq r_2\leq r$, then by
Chebyshev's inequality, we get that,
\[ 
	\HM^1\left( \left\{ t\in [r_1,r_2]\midd \Xi(t)\leq \zeta F_1(r)^{2/3} \right\}
	\right)\geq r_2-r_1-\frac{1}{\zeta}rF_1(r)^{1/3},
\]
thus $\left\{ t\in [r_1,r_2]\midd \Xi(t)\leq \zeta F_1(r)^{2/3} \right\}\neq
\emptyset$ when $r_2-r_1>(1/\zeta) rF_1(r)^{1/3}$.

\begin{lemma}\label{le:1appEC}
	Let $R_0<(1-\tau)\mathfrak{r}$ be a positive number such that $F(R_0)\leq 
	\mu\tau_0/4$ and $\bar{\varepsilon}(R_0)\leq 10^{-4}$. For any 
	$r\in\mathscr{R}\cap (0,R_0)$, if $\Xi(r)\leq \mu\tau_0$, then there is a 
	constant $C=C(\mu,\Theta(0))$ such that
	\[
		\dist(x, E)\leq C r\left(F_1(r)^{1/3}+\Xi(r)^{1/2}\right), 
		\ x\in X(r)\cap B_r.
	\]
\end{lemma}
\begin{proof}
	For any $k\geq 0$, we take $r_k=2^{-k}r$. Then there exists $t_k\in
	[r_k,r_{k-1}]$ such that 
	\[
		\Xi(t_k)\leq \frac{\int_{r_k}^{r_{k-1}}\Xi(t)dt}{r_{k-1}-r_k}\leq
		\frac{r_{k-1}F_1(r_{k-1})}{r_{k-1}/2}=2F_1(r_{k-1}).
	\]
	We let $X_k=X(t_k)$, then for any $j>i\geq 1$, we have that 
	\begin{equation}\label{eq:sequ}
		\begin{aligned}
			d_{0,1}(X_i,X_j)&\leq \sum_{k=i}^{j-1}d_{0,1}(X_k,X_{k+1})\\
			&\leq 60 \left( \Theta(0)+\mu\tau_0/4 \right)^{1/2}
			\sum_{k=i}^{j-1}F_1(t_k)^{1/2}+ 20\pi\mu^{-1/2}\sum_{k=i}^{j-1}
			\left( \Xi(t_k)^{1/2}+\Xi(t_{k+1})^{1/2} \right) \\
			&\leq \left(  60 \left( \Theta(0)+\mu\tau_0/4
			\right)^{1/2}+40\pi\mu^{-1/2}
			\right)\sum_{k=i}^{j-1}2F_1(t_k)^{1/2}+F_1(t_{k-1})^{1/2}\\
			&\leq C_1(\mu,\Theta(0))(j-i)F_1(r_{i-1})^{1/2}=C_1(\mu,\Theta(0))
			F_1(r_{i-1})^{1/2}\log_2 (r_i/r_j),
		\end{aligned}
	\end{equation}
	where $C_1(\mu,\Theta(0))=3\left(  60 \left( \Theta(0)+\mu\tau_0/4
	\right)^{1/2}+40\pi\mu^{-1/2} \right)$.

	For any $x\in X(r)\cap B_r$ with $\Xi(|x|)\leq \mu\tau_0$, we assume that $t_{k+1}\leq
	|x|< t_k$, then 
	\[
		\begin{aligned}
			\dist(x,E)&\leq d_H(X(r)\cap B_{|x|},X(|x|)\cap B_{|x|})+d_H(X(|x|)\cap
			B_{|x|},\gamma(|x|))\\
			&\leq 2|x|d_{0,1}(X(r),X(|x|))+10\mu^{-1/2}|x|\Xi(|x|)^{1/2} \\
			&\leq 2|x|(d_{0,1}(X(|x|),X_k)+d_{0,1}(X_k,X_1)+d_{0,1}(X_1,X(r)))
			+10\mu^{-1/2}|x|\Xi(|x|)^{1/2}\\
			&\leq (40\pi+10)\mu^{-1/2}|x| \left( \Xi(|x|)^{1/2}+\Xi(r)^{1/2}\right)+
			C_2(\mu,\Theta(0))|x|F_1(r)^{1/2}\log_{2}(r/|x|)\\
			&\leq (40\pi+10)\mu^{-1/2} |x|\Xi(|x|)^{1/2}+C_3(\mu,\Theta(0))r \left( 
			\Xi(r)^{1/2}+ F_1(r)^{1/2} \right)
		\end{aligned}
	\]

	For any $0\leq a\leq b\leq r$, we put 
	\[
		I(a,b)=\left\{ t\in [a,b]\midd \Xi(t)\leq F_1(r)^{2/3}
		\right\},
	\]
	then $I(a,b)\neq\emptyset$ when $b-a> rF_1(r)^{1/3}$. If $|x|\in I(0,r)$,
	then
	\[
		\dist(x,E)\leq C_4(\mu,\Theta(0))r\left(F_1(r)^{1/3}
		+\Xi(r)^{1/2}\right).
	\]

	We let $\{ s_i \}_{i=0}^{m+1}\subseteq [0,r]$ be a sequence such that
	\[
		0=s_0<s_1<\cdots<s_m<s_{m+1}=r,\ s_i\in I(0,r),
	\]
	and 
	\[
		s_{i+1}-s_i\leq 2rF_1(r)^{1/3}.
	\]

	For any $x\in X(r)\cap B_r$, if $s_i\leq |x|<s_{i+1}$ for 
	some $0\leq i\leq m$, we have that 
	\[
		\begin{aligned}
			\dist(x,E)&\leq \left|x-\frac{s_i}{|x|}x\right| +\dist\left(
			\frac{s_i}{|x|}x,E\right)\\
			&\leq (s_{i+1}-s_i)+ C_4(\mu,\Theta(0))r\left(F_1(r)^{1/3}
			+\Xi(r)^{1/2}\right)\\
			&\leq (C_4(\mu,\Theta(0))+2) r \left(F_1(r)^{1/3} +\Xi(r)^{1/2}\right).
		\end{aligned}
	\]
\end{proof}

\begin{definition}
	Let $U\subseteq\mathbb{R}^{3}$ be an open set, $E\subseteq \mathbb{R}^{3}$ be a
	set of Hausdorff dimension $2$. $E$ is called Ahlfors-regular in $U$ if
	there is a $\delta>0$ and $\xi_0\geq 1$ such that, for any $x\in E\cap U$, if
	$0<r<\delta$ and $B(x,r)\subseteq U$, we have that 
	\[
		\xi_0^{-1} r^2\leq \HM^2(E\cap B(x,r))\leq \xi_0 r^2.
	\]
\end{definition}
\begin{lemma}\label{le:2appEC}
	Let $R_0$ be the same as in Lemma \ref{le:1appEC}. If $E$ is Ahlfors-regular,
	and $r\in\mathscr{R}\cap (0,R_0)$ satisfies $\Xi(r)\leq \mu\tau_0$, then 
	there is a constant $C=C(\mu,\xi_0,\Theta(0))$ such that 
	\[
		\dist(x,X(r))\leq Cr\left(F_1(r)^{1/4}+\Xi(r)^{1/2}\right),\ x\in E\cap
		B(0,9r/10).
	\]
\end{lemma}
\begin{proof}
	Let $\{X_k\}_{k\geq 1}$ be the same as in \eqref{eq:sequ}. For any $t\in
	\mathscr{R}$ with $t_{k+1}\leq t< t_k$, $\Xi(t)\leq \mu\tau_0$ and $x\in \gamma(t)$, we have that 
	\[
		\begin{aligned}
			\dist(x,X(r))&\leq d_H(\gamma(t),X(|x|)\cap B_{|x|})+d_H(X(|x|)\cap
			B_{|x|},X(r))\\
			&\leq (40\pi+10)\mu^{-1/2} |x|\Xi(|x|)^{1/2}+C_3(\mu,\Theta(0))r 
			\left( \Xi(r)^{1/2}+ F_1(r)^{1/2} \right)		
		\end{aligned}
	\]
	We put 
	\[
		J(0,r)=\{ t\in [0,r]:\Xi(t)> F_1(r)^{1/2} \}.
	\]
	For any $x\in \gamma(t)$ with $t\in (0,r)\setminus J(0,r)$, we have that 
	\[
		\dist(x,X(r))\leq C_5(\mu,\Theta(0))r\left(
		\Xi(r)^{1/2} + F_1(r)^{1/4}\right).
	\]
	We put 
	\[
		E_1=\bigcup_{t\in J(0,r)}(E\cap \partial B_t),\ E_2= \bigcup_{t\in
		(0,r)\setminus J(0,r)}(E\cap B_t\setminus \gamma(t)),
	\]
	and
	\[
		E_3=E\cap B_r\setminus (E_1\cup E_2)=\bigcup_{t\in
		(0,r)\setminus J(0,r)}\gamma(t).
	\]
	Then 
	\begin{equation}\label{eq:sm}
		\begin{aligned}
			\HM^2(E_1\cup E_2)&=\int_{E\cap B_r}d \HM^2(x)-\int_{E_3}d\HM^2(x)\\
			&\leq \int_{E\cap B_r}d \HM^2(x)-\int_{E_3}\cos\theta(x)d\HM^2(x)\\
			&=\int_{E\cap B_r}(1-\cos\theta(x))d \HM^2(x)+\int_{
			E_1\cup E_2}\cos\theta(x)d\HM^2(x)\\
			&\leq r^2F(r)+\int_{0}^{r}\HM^1(E_1\cap \partial B_t)d t
			+\int_{0}^{r}\HM^1(E_2\cap \partial B_t) dt\\
			&\leq r^2F(r)+\int_{J(0,r)}(2\Theta(0)+tf'(t)+2f(t))td t
			+ \mu^{-1}\int_{0}^rt\Xi(t) d t\\
			&\leq (2+\mu^{-1})r^2F_1(r)+2\Theta(0)\int_{\{ t\in [0,r]:\Xi(t)> F_1(r)^{1/2} \}} t dt\\
			&\leq (2+\mu^{-1})r^2F_1(r)+\frac{2\Theta(0)}{ F_1(r)^{1/2}}\int_0^rt\Xi(t) d	t\\
			&\leq C_6(\mu,\Theta(0))r^2F_1(r)^{1/2},
		\end{aligned}
	\end{equation}
	where $C_6(\mu,\Theta(0))=(2+\mu^{-1})(\mu\tau_0/4)^{1/2}+2\Theta(0)$.

	We see that, for any $x\in  E_3$,
	\[
		\dist(x,X(r))\leq C_5(\mu,\Theta(0)) r\left(\Xi(r)^{1/2} + F_1(r)^{1/4} \right). 
	\]
	If $x\in E\cap B(0,9r/10)$ with
	\[
		\dist(x,X(r))>C_5(\mu,\Theta(0)) r\left(\Xi(r)^{1/2} + F_1(r)^{1/4} \right)+s
	\]
	for some $s\in (0,r/10)$, then $E\cap B(x,s)\subseteq E_1\cup E_2$, thus 
	\[
		\HM^2(E\cap B(x,s))\leq C_6(\mu,\Theta(0))r^2F_1(r)^{1/2}.
	\]
	But on the other hand, by Ahlfors-regular property of $E$, we have that 
	\[
		\HM^2(E\cap B(x,s))\geq \xi_0^{-1}s^2.
	\]
	We get so that 
	\[
		s\leq C_6(\mu,\Theta(0))^{1/2}\cdot\xi_0^{1/2}\cdot r F_1(r)^{1/4} .
	\]
	Therefore, for $x\in E\cap B(0,9r/10)$,
	\[
		\dist(x,X(r))\leq
		\left(C_6(\mu,\Theta(0))^{1/2}\cdot\xi_0^{1/2}+C_5(\mu,\Theta(0))\right)
		\left(\Xi(r)^{1/2} + F_1(r)^{1/4} \right) .
	\]

\end{proof}
For any $k\geq 0$, we take $R_k=2^{-k}R_0$ and $s_k\in
[R_{k+1},R_{k}]$ such that  
\[
	\Xi(s_k)\leq \frac{\int_{R_{k+1}}^{R_{k}}\Xi(t) d t}{R_{k}-R_{k+1}}\leq 2F_1(R_{k}).
\]
We put $X_k=X(s_k)$. Then for any $j\geq i\geq 2$, we have that 
\[
	\begin{aligned}
		d_{0,1}(X_i,X_j)&\leq \frac{C_1(\mu,\Theta(0))}{3}\sum_{k=i}^{j-1}
		\left(2F_1(s_k)^{1/2}+F_1(s_{k-1})^{1/2}\right)\\
		&\leq C_1(\mu,\Theta(0))\sum_{k=i-1}^{j-1} F_1(R_{k})^{1/2}\\
		&\leq  \frac{C_1(\mu,\Theta(0))}{\ln
		2}\sum_{k=i-1}^{j-1}\int_{R_k}^{R_{k-1}}\frac{F_1(t)^{1/2}}{t}dt\\
		&=\frac{C_1(\mu,\Theta(0))}{\ln
		2} \int_{R_{i-2}}^{R_{j-1}}\frac{F_1(t)^{1/2}}{t}dt.
	\end{aligned}
\]
If the gauge function $h$ satisfy that 
\begin{equation}\label{eq:smallF}
	\int_0^{R_0}\frac{F_1(t)^{1/2}}{t}dt < +\infty,
\end{equation}
then $X_k$ converges to a cone $X(0)$, and 
\[
	d_{0,1}(X(0),X_k)\leq \frac{C_1(\mu,\Theta(0))}{\ln
	2}\int_0^{R_{k-2}}\frac{F_1(t)^{1/2}}{t}dt.
\]
\begin{remark}
	If $h(r)\leq C(\ln (A/r))^{-b}$, $0<r\leq
	R_0$, for some $A>R_0$, $C>0$ and $b>3$, then \eqref{eq:smallF} holds.
\end{remark}
Indeed,
\[
	h_1(r)=\int_{0}^{r}\frac{h(2t)}{t}dt\leq \frac{C}{b-1}\left( \ln\left(
	\frac{A}{r} \right)\right)^{-b+1},
\]
and then Remark \ref{re:largegaugefun} implies that
\[
	F(r)\leq C_1\left( \ln\left( \frac{A}{r}\right)\right)^{-b}+
	\frac{C}{b-1}\left( \ln\left( \frac{A}{r} \right)\right)^{-b+1}\leq 
	C_2\left( \ln\left( \frac{A}{r} \right)\right)^{-b+1},
\]
thus \eqref{eq:smallF} holds.
\begin{lemma}
	If \eqref{eq:smallF} holds, then $X(0)$ is a minimal cone. 
\end{lemma}
\begin{proof}
	By Lemma \ref{le:EM}, for any $r\in (0,\mathfrak{r})\cap \mathscr{R}$, there
	exist sliding minimal cone $Z(r)$ such that $d_{0,1}(X(r),Z(r))\leq 4
	\varepsilon(r)$. But $\varepsilon(r)\to 0$ as $r\to 0+$, we get that 
	\[
		d_{0,1}(Z(s_k),X(0))\to 0. 
	\]
	Since $Z(s_k)$ is sliding minimal for any $k$, we get that $X(0)$ is also
	sliding minimal.
\end{proof}
For any $r\in \mathscr{R}\cap (0,R_0)$ with $\Xi(r)\leq \mu\tau_0$, 
we assume $R_{k+1}\leq r<R_k$, by Lemma \ref{le:CompCon}, we have that 
\begin{equation}\label{eq:appCT}
	\begin{aligned}
		d_{0,1}(X(0),X(r))&\leq d_{0,1}(X(0),X_{k+3})+d_{0,1}(X_{k+3},X(r))\\
		&\leq \frac{C_1(\mu,\Theta(0))}{\ln 2}\int_0^{R_{k+1}}\frac{F_1(t)^{1/2}}{t}dt\\
		&\quad +\frac{30r}{s_{k+3}}\Theta(0,r)^{1/2}F_1(r)^{1/2}+20\pi\mu^{-1/2}\left( \Xi(s_{k+3})^{1/2}+\Xi(r)^{1/2} \right)\\
		&\leq 10C_1(\mu,\Theta(0))\left(\Xi(r)^{1/2}+F_1(r)^{1/2}+
		\int_0^{r}\frac{F_1(t)^{1/2}}{t}dt\right).
	\end{aligned}
\end{equation}
\begin{theorem}\label{thm:distTC}
	If \eqref{eq:smallF} holds, and $E$ is Ahlfors-regular, then $E$ has unique 
	blow-up limit $X(0)$ at $0$, and there is a constant $C=C_{10}(\mu,\Theta,
	\xi_0)$ such that 
	\begin{equation}\label{eq:distTC0}
		d_{0,9r/10}(E,X(0))\leq C \left(F_1(r)^{1/4}+
		\int_0^{r}\frac{F(t)^{1/2}}{t}dt \right),\ 0<r< \mathfrak{r}.
	\end{equation}
	In particular, 
	\begin{itemize}[nolistsep]
		\item if $h(r)\leq C_h(\ln(A/r))^{-b}$ for some $A,C_h>0$, $b>3$ and
			$0<r\leq R_0<A$, then 
			\[
				d_{0,r}(E,X(0))\leq C'(\ln(A_1/r))^{-(b-3)/4},\ 0<r\leq 9R_0/10, \
				A_1\leq 10A/9;
			\]
		\item if $h(r)\leq C_hr^{\alpha_1}$ for some $C_h,\alpha_1>0$, and 
			$0<r\leq r_0$, $0<r_0\leq \min\{ 1,R_0 \}$, then 
			\[
				d_{0,r}(E,X(0))\leq C (r/r_0)^{\beta},\ 0<r\leq 9r_0/10,\
				0<\beta<\alpha_1,
			\]
			where 
			\[
				C\leq C_{11}(\mu,\lambda_0,\alpha_1,\beta,C_h,\xi_0,\Theta(0))\left(
				F(r_0)^{1/4}+r_0^{\alpha_1/4}\right).
			\]
	\end{itemize}
\end{theorem}
\begin{proof}
	From \eqref{eq:appCT} and Lemma \ref{le:1appEC}, we get that, for any 
	$x\in X(0)\cap B_r$ where $r\in \mathscr{R}\cap (0,R_0)$ such that $\Xi(r)\leq
	\mu\tau_0$,
	\[
		\dist(x,E)\leq C_7(\mu,\xi_0,\Theta(0))r\left( \Xi(r)^{1/2}+F_1(r)^{1/4}+
		\int_0^{r}\frac{F_1(t)^{1/2}}{t}dt \right).
	\]
	Similarly to the proof of Lemma \ref{le:1appEC}, we still consider 
	\[
		I(a,b)=\left\{ t\in [a,b]\midd \Xi(t)\leq F_1(r)^{2/3} \right\},\ 0\leq
		a\leq b\leq r,
	\]
	we have that $I(a,b)\neq \emptyset$ whenever $b-a>rF_1(r)^{1/3}$. We let
	$\{ s_i \}_{0}^{m+1}\subseteq [0,r]$ be a sequence such that 
	\[
		0=s_0<s_1<\cdots<s_m<s_{m+1}=r,\ s_i\in I(0,r),
	\]
	and 
	\[
		s_{i+1}-s_i\leq 2rF_1(r)^{1/3}.
	\]
	For any $r\in (0,R_0)$, we assume that $s_i\leq r<s_{i+1}$, $x\in X(0)\cap
	\partial B_r$. 
	\begin{equation}\label{eq:distTC1}
		\begin{aligned}
			\dist(x,E)&\leq \left| x-\frac{s_i}{|x|}x\right|+\dist\left(
			\frac{s_i}{|x|}x,E \right)\\
			&\leq C_8(\mu,\xi_0,\Theta(0)) r \left( F_1(r)^{1/4}+
			\int_0^{r}\frac{F_1(t)^{1/2}}{t}dt\right)
		\end{aligned}
	\end{equation}
	From 	\eqref{eq:appCT} and Lemma \ref{le:2appEC}, we have that, for any 
	$x\in X(0)\cap B(0,9r/10)$ where $r\in \mathscr{R}\cap (0,R_0)$ such that
	$\Xi(r)\leq \mu\tau_0$,
	\[
		\dist(x,X(0))\leq C_9(\mu,\xi_0,\Theta(0))\left( \Xi(r)^{1/2}+F_1(r)^{1/4} +
		\int_0^{r}\frac{F_1(t)^{1/2}}{t}dt\right).
	\]
	Similarly to the proof of Lemma \ref{le:2appEC}, we can get that
	\begin{equation}\label{eq:distTC2}
		\dist(x,X(0))\leq C_{10}(\mu,\xi_0,\Theta(0)) \left( F_1(r)^{1/4} +
		\int_0^{r}\frac{F_1(t)^{1/2}}{t}dt \right).
	\end{equation}

	We get, from \eqref{eq:distTC1} and \eqref{eq:distTC2}, that
	\eqref{eq:distTC0} holds.

	If $h(r)\leq C_h(\ln(A/r))^{-b}$ for some $A,C_h>0$ and $b>3$ and $0<r\leq
	R_0<A$, then 
	\[
		h_1(r)=\int_{0}^{r}\frac{h(2t)}{t}dt \leq \frac{C_h}{b-1}\left( \ln\left(
		\frac{A}{r} \right) \right)^{-b+1},
	\]
	and by Remark \ref{re:largegaugefun} we have that 
	\[
		F(r)\leq C''\left( \ln \frac{A}{r} \right)^{-b+1}
	\]
	where 
	\[
		C''\leq  C(R_0,\lambda,b)\left(\ln\frac{A}{r}\right)^{-1}+
		\frac{C_1}{b-1} \leq C(R_0,\lambda,b)\left(\ln\frac{A}{R_0}\right)^{-1}+
		\frac{C_1}{b-1}
	\]
	is bounded, thus
	\[
		\int_{0}^{r}\frac{F_1(t)^{1/2}}{t}dt \leq C'''
		\left( \ln \frac{A}{r} \right)^{(-b+3)/2}
	\] 
	Hence we get that 
	\[
		\begin{aligned}
			d_{0,9r/10}(E,X(0))&\leq C_{10}(\mu,\xi_0,\Theta(0)) \left(
			F_1(r)^{1/4}+\int_{0}^{r}\frac{F_1(t)^{1/2}}{t}dt \right)\\
			&\leq C' \left( \ln \frac{A}{r} \right)^{-(b-3)/4}.
		\end{aligned}
	\]

	If $h(r)\leq C_hr^{\alpha_1}$ for some $C_h,\alpha_1>0$ and $0<r\leq
	r_0$, then 
	\[
		h_1(r)=\int_{0}^{r}\frac{h(2t)}{t}dt\leq
		\frac{C_h}{\alpha_1}(2r)^{\alpha_1}.
	\]
	We see, from the proof of Corollary \ref{co:dendecay}, that 
	\[
		f(r)\leq \left( f(r_0)+C_2(\alpha_1,\beta,\lambda_0)C_hr_0^{\alpha_1}
		\right)(r/r_0)^{\beta},\ \forall 0<\beta<\alpha_1,
	\]
	thus 
	\[
		F_1(r)=f(r)+16h_1(r)\leq (f(r_0)+C_2'(\alpha_1,\beta,\lambda_0)
		C_hr_0^{\alpha_1})(r/r_0)^{\beta}.
	\]
	Then 
	\[
		\begin{aligned}
			d_{0,9r/10}(E,X(0))&\leq C_{10}(\mu,\xi_0,\Theta(0)) 
			\left(F_1(r)^{1/4}+\int_{0}^{r}\frac{F_1(t)^{1/2}}{t}dt\right)\\
			& \leq C (r/r_0)^{\beta/4},
		\end{aligned}
	\]
	where 
	\[
		C\leq C_{10}'(\mu,\xi_0,\Theta(0))
		(F(r_0)^{1/4}+C_{2}''(\alpha_1,\beta,\lambda_0,C_h)r_0^{1/4}).
	\]
\end{proof}

\section{Parameterization of well approximate sets}
Recall that a cone in $\mathbb{R}^3$ is called of type $\mathbb{P}$ if it is a
plane; a cone is called of type $\mathbb{Y}$ if it is the union of three half 
planes with common boundary line and that make $120^{\circ}$ angles along the
boundary line; a cone of type $\mathbb{T}$ if it is the cone over the union of
the edges of a regular tetrahedron.
\begin{theorem}\label{thm:RPWAPS}
	Let $E\subseteq \Omega_0$ be a set with $0\in E$. Suppose that there exist $C>0$,
	$r_0>0$, $\beta>0$ and $0<\eta\leq 1$ such that, for any $x\in E\cap B(0,r_0)$
	and $0<r\leq 2r_0$, we can find cone $Z_{x,r}$ through $x$ such that 
	\[
		d_{x,r}(E,Z_{x,r})\leq Cr^{\beta},
	\]
	where $Z_{x,r}$ is a minimal cone in $\mathbb{R}^3$ of type $\mathbb{P}$ or
	$\mathbb{Y}$ when $x\notin \partial \Omega_0$ and $0<r<\eta\dist(x,\partial
	\Omega_0)$, and otherwise, $Z_{x,r}$ is a sliding minimal cone of type
	$\mathbb{P}_+$ or $\mathbb{Y}_+$ in $\Omega_0$ with sliding boundary
	$\partial \Omega_0$ centered at some point in $\partial \Omega_0$. Then 
	there exist a radius $r_1\in (0,r_0/2)$, a sliding minimal cone $Z$ centered 
	at $0$ and a mapping $\Phi:\Omega_0\cap B(0,r_1)\to \Omega_0$, which is a 
	$C^{1,\beta}$-diffeomorphism between its domain and image, such that
	$\Phi(0)=0$, $\Phi(\partial\Omega_0\cap B(0,2r_1))\subseteq\partial\Omega_0$,
	$\|\Phi-\id\|_{\infty}\leq 10^{-2}r_1$ and 
	\[
		E\cap B(0,r_1)=\Phi(Z)\cap B(0,r_1).
	\]
\end{theorem}
\begin{proof}
	Let $\sigma:\mathbb{R}^3\to \mathbb{R}^3$ be given by
	$\sigma(x_1,x_2,x_3)=(x_1,x_2,-x_3)$. By setting $E_1=E\cup \sigma(E)$, we
	have that, for any $x\in E_1\cap B(0,r_0)$ and $0<r\leq 2r_0$, there exist
	minimal cone $Z(x,r)$ in $\mathbb{R}^3$ centered at $x$ of type
	$\mathbb{P}$ or $\mathbb{Y}$ such that $Z(\sigma(x),r)=\sigma(Z(x,r))$ and  
	\[
		d_{x,r}(E,Z(x,r))\leq Cr^{\beta}.
	\]
	By Theorem 4.1 in \cite{Fang:2015}, there exist $r_1\in (0,r_0)$, 
	$\tau\in (0,1)$, a cone $Z$ centered at $0$ of type $\mathbb{P}$ or
	$\mathbb{Y}$, and a mapping $\Phi_1:B(0,3r_1/2)\to B(0,2r_1)$ such that 
	\[
		\begin{gathered}
			\sigma(Z)=Z,\ \sigma\circ \Phi_1=\Phi_1\circ \sigma,\ \|\Phi_1-\id\|\leq
			r_0\tau,\\
			C_1|x-y|^{1+\tau}\leq |\Phi(x)-\Phi(y)|\leq C_1^{-1}|x-y|^{1/(1+\tau)},\\
			E_1\cap B(0,r_1)\subseteq \Phi_1(Z\cap B(0,3r_1/2))\subseteq E_1\cap
			B(0,2r_1).
		\end{gathered}
	\]

	Using the same argument as in Section 10 in \cite{DDT:2008}, we get 
	that $\Phi_1$ is of class $C^{1,\beta}$. 
\end{proof}
\section{Approximation of $E$ by cones away from the boundary}
In this section, we let $\Omega\subseteq \mathbb{R}^3$ be a closed set.
Let $E\in SAM(\Omega,\partial \Omega,h)$ be a sliding almost minimal
set, $x_0\in E\setminus \partial \Omega$. Then $E\cap B(x,r)$ is almost minimal 
with gauge function $h$ for any $0<r<\dist(x_0,\partial \Omega)$. We put
\[
	F(x,r)=\Theta(x,r)-\Theta(x)+8h_1(r).
\]
We see from Theorem \ref{thm:ANDD} that $F(x,r)\geq 0$ and $F(x,\cdot)$ is
nondecreasing for $0<r<\dist(x_0,L)$.

\begin{theorem} \label{thm:ine}
	If $\int_{0}^{R_0}r^{-1}F(x,r)^{1/3}dr<\infty$ for some $R_0>0$, then $E$ 
	has unique blow-up limit $T$ at $x$. Moreover there is a constant $C>0$ and
	a radius $\rho_0=\rho_0(x)>0$ such that  
	\begin{equation}\label{eq:ine0}
		d_{x,r}(E,T)\leq C \int_{0}^{200r}\frac{F(x,t)^{1/3}}{t}dt,\
		0<r\leq \rho_0.
	\end{equation}
	In particular, if the gauge function $h$ satisfies that 
	\[
		h(t)\leq C_ht^{\alpha_1} \text{ for some }\alpha_1>0 \text{ and }0<t\leq
		R_0,
	\]
	then there is a $\beta_0>0$ such that, for any $0<\beta<\beta_0$, 
	\[
		d_{x,r}(E,T)\leq C(\alpha_1,\beta)\left(F(x,\rho_0)+
		C_h\rho_0^{\alpha_1}\right)^{1/3}(r/\rho_0)^{\beta/3}. 
	\]
\end{theorem}
\begin{proof}
	Let $\varrho$ be the radius defines as in \eqref{eq:imcone}.
	We take $\rho_0=10^{-3}\min\{R_0,\dist(x_0,\partial \Omega),\varrho\}$. 
	By Theorem 11.4 in \cite{David:2008}, there is a constant $C>0$ and cone 
	$Z_r$ for each $0<r<\rho_0$ such that 
	\[
		d_{x,r}(E,Z_r)+\alpha_+(Z_r)\leq CF(x,110r)^{1/3}.
	\]
	We put $\rho_k=2^{-k}\rho_0$, and $Z_k=Z_{\rho_k}$. Then 
	\[
		\begin{aligned}
			d_{x,1}(Z_k,Z_{k+1})&=d_{x,\rho_{k+1}}(Z_{k},Z_{k+1})\leq
			d_{x,\rho_{k+1}}(Z_k,E)+d_{x,\rho_{k+1}}(E,Z_{k+1})\\
			&\leq CF(x,110\rho_{k+1})^{1/3}+2CF(x,110\rho_{k})^{1/3}.
		\end{aligned}
	\]
	For any $1\leq i<j$, we have that 
	\[
		\begin{aligned}
			d_{x,1}(Z_i,Z_j)&\leq 2C \sum_{k=i}^{j-1}F(x,110\rho_{k})^{1/3}+C
			\sum_{k=i+1}^{j}F(x,110\rho_{k})^{1/3} \leq 3C
			\sum_{k=i}^{j}F(x,110\rho_{k})^{1/3}\\
			&\leq \frac{3C}{\ln	2}
			\int_{\rho_j}^{\rho_{i-1}}\frac{F(x,110t)^{1/3}}{t}dt.
		\end{aligned}
	\]
	Let $Z_0$ be the limit of $\{Z_k\}_{k=1}^{\infty}$. Then we have that 
	\[
		d_{x,1}(Z_0,Z_i)\leq \frac{3C}{\ln	2}
		\int_{0}^{\rho_{i-1}}\frac{F(x,110t)^{1/3}}{t}dt.
	\]
	For any $0<r<\rho_0$, we assume that $\rho_{k+1}\leq r< \rho_k$, then 
	\[
		\begin{aligned}
			d_{x,1}(Z_r,Z_0)&\leq
			d_{x,\rho_{k+1}}(Z_r,Z_{k+1})+d_{x,1}(Z_{k+1},Z_0)\\
			&\leq d_{x,1}(Z_{k+1},Z_0)+d_{x,\rho_{k+1}}(Z_r,E)+
			d_{x,\rho_{k+1}}(E,Z_{k+1})\\
			&\leq d_{x,1}(Z_{k+1},Z_0)+\frac{r}{\rho_{k+1}}d_{x,r}(Z_r,E)+
			d_{x,\rho_{k+1}}(E,Z_{k+1})\\
			&\leq 3CF(x,110r)^{1/3}+\frac{3C}{\ln	2}
			\int_{0}^{\rho_k}\frac{F(x,110t)^{1/3}}{t}dt.
		\end{aligned}
	\]
	Hence 
	\begin{equation}\label{eq:ine10}
		d_{x,r}(E,Z_0)\leq d_{x,r}(E,Z_r)+ d_{x,r}(Z_r,Z_0)\leq \frac{10C}{\ln	2}
		\int_{0}^{200r}\frac{F(x,t)^{1/3}}{t}dt
	\end{equation}
	and $T=\trans{x}(Z_0)$ is the only blow up limit of $E$ at $x$, which is 
	a minimal cone.

	By Theorem 4.5 in \cite{David:2008}, we have that 
	\[
		\Theta_E(x,r)\leq \left(\frac{1}{2}-\alpha_0\right)\frac{\HM^1(E\cap
		B(x,r))}{r} +2 \alpha_0\Theta_E(x)+4h(r),
	\]
	where we take $\alpha_0$ the constant $\alpha$ in Theorem 4.5 
	in \cite{David:2008}. For our convenient, we denote
	$u(r)=\HM^2(E\cap B(x,r))$ and $f(r)=\Theta_E(x,r)- \Theta_E(x)$, then we
	have $\HM^1(E\cap \partial B(x,r))\leq u'(r)$ and 
	\[
		\begin{aligned}
			f(r)+\Theta_E(x)&\leq \left(\frac{1}{2}-\alpha_0\right)\frac{u'(r)}{r}+
			2 \alpha_0 \Theta_E(x)+4h(r)\\
			&= \left(\frac{1}{2}-\alpha_0\right)(2f(r)+rf'(r)+2 \Theta_E(x))+
			2 \alpha_0 \Theta_E(x)+4h(r),
		\end{aligned}
	\]
	thus
	\[
		rf'(r)\geq \frac{4 \alpha_0}{1-2 \alpha_0}f(r)-\frac{8}{1-2 \alpha_0}h(r),
	\]
	and 
	\[
		\left( r^{-\frac{4 \alpha_0}{1-2 \alpha_0}}f(r)\right)'\geq -\frac{8}{1-2
		\alpha_0}r^{-\frac{1+2 \alpha_0}{1-2 \alpha_0}}h(r).
	\]
	We take $\beta_0=\min\{4 \alpha_0/(1-2 \alpha_0),\alpha_1\}$. Then for any
	$0<\beta<\beta_0$, we have that 
	\[
		\begin{aligned}
			f(r)&\leq (r/\rho_0)^{\frac{4 \alpha_0}{1-2 \alpha_0}} f(\rho_0)+\frac{8}{1-2
			\alpha_0}r^{\frac{4 \alpha_0}{1-2 \alpha_0}}\int_{r}^{\rho_0}t^{-\frac{1+2
			\alpha_0}{1-2 \alpha_0}}h(t) dt\\
			&\leq (r/\rho_0)^{\frac{4 \alpha_0}{1-2 \alpha_0}} f(\rho_0) +
			C_1'(\alpha_1,\beta,\alpha_0) \rho_0^{\alpha_1} \cdot (r/\rho_0)^{\beta}.
		\end{aligned}
	\]
	We get so that 
	\[
		F(x,r)\leq C(\alpha_1,\beta,\alpha_0)(F(x,\rho_0)+C_h\rho_0^{\alpha_1})
		(r/\rho_0)^{\beta},
	\]
	combine this with \eqref{eq:ine10}, we get the conclusion.
\end{proof}

\section{Parameterization of sliding almost minimal sets}
Let $n$, $d\leq n$ and $k$ be nonnegative integers, $\alpha\in (0,1)$. 
By a $d$-dimensional submanifold of class $C^{k,\alpha}$ of $\mathbb{R}^{n}$ 
we mean a subset $M$ of $\mathbb{R}^{n}$ satisfying that for each $x\in M$
there exist s neighborhood $U$ of $x$ in $\mathbb{R}^{n}$, a mapping 
$\Phi:U\to \mathbb{R}^{n}$ which is a diffeomorphism of class $C^{k,\alpha}$ 
between its domain and image, and a $d$ dimensional vector subspace $Z$ of
$\mathbb{R}^{n}$ such that 
\[
	\Phi(M\cap U)= Z\cap \Phi(U).
\]

In this section, we assume that $\Omega\subseteq \mathbb{R}^{3}$ is a closed
set whose boundary $\partial \Omega$ is a 2-dimensional submanifold of
class $C^{1,\alpha}$ for some $\alpha\in (0,1)$, and suppose that $\Omega$ has tangent
cone a half space at any point in $\partial \Omega$. 
Let $E\subseteq\Omega$ be a closed set such that $E\in SAM(\Omega,\partial
\Omega,h)$ and $\partial \Omega\subseteq E$, $x_0\in \partial\Omega$. We always 
assume that the gauge function $h$ satisfies that 
\begin{equation}\label{eq:gaugesmall1}
	\int_{0}^{R_0}\frac{1}{r}\left(\int_{0}^{r}\frac{h(2t)}{t}dt\right)^{1/2}dr
	<+\infty
\end{equation}
and
\begin{equation}\label{eq:gaugesmall2}
	\int_{0}^{R_0}r^{-1+\frac{\lambda}{1-\lambda}} \left(\int_{r}^{R_0}
	t^{-1-\frac{2\lambda}{1-\lambda}}h(2t)dt\right)^{1/2}dr<+\infty,
\end{equation}
for some $R_0>0$. It is easy to see that if $h(t)\leq Ct^{\alpha_1}$ for some
$\alpha_1>0$, $C>0$ and $0<t\leq R_0$, then \eqref{eq:gaugesmall1} and
\eqref{eq:gaugesmall2} hold.
%Then by \eqref{eq:denapp10}, we have that \eqref{eq:smallF} holds. 
For our convenient, we put $\lambda_0=\lambda/(1-\lambda)$,  
\begin{equation}\label{eq:sgauge1}
	h_2(\rho)=\int_{0}^{\rho}\frac{1}{r}\left( \int_{0}^{r}\frac{h(2t)}{t}
	dt\right)^{1/2}dr
\end{equation}
and 
\begin{equation}\label{eq:sgauge2}
	h_3(\rho)=\int_{0}^{\rho}r^{-1+\lambda_0}\left(
	\int_{r}^{R_0}t^{-1-2\lambda_0}h(2t) dt\right)^{1/2}dr.
\end{equation}

We see, from Proposition 4.1 in \cite{David:2014}, that $E$ is Ahlfors-regular
in $B(x_0,R_0)$, i.e. there exist $\delta_1>0$ and $\xi_1\geq 1$ such that for
any $x\in E\cap B(x_0,R_0)$, if $0<r<\delta_1$ and $ B(x,r)\subseteq B(x_0,R_0)$,
we have that 
\[
	\xi_1^{-1}r^2\leq \HM^2(E\cap B(x,r))\leq \xi_1 r^{2} .
\]
We see from Theorem 3.10 in \cite{Fang:2015} that there only there kinds of
possibility for the blow-up limits of $E$ at $x_0$, they are the plane 
$\Tan(\partial\Omega,x_0)$, cones of type $\mathbb{P}_+$ union
$\Tan(\partial\Omega,x_0)$, and cones of type $\mathbb{Y}_+$ union
$\Tan(\partial\Omega,x_0)$. By Proposition 29.53 in \cite{David:2014}, we get 
so that 
\[
	\Theta_E(x_0)=\pi,\ \frac{3\pi}{2}, \text{ or } \frac{7\pi}{4}.
\]
If $\Theta_E(x_0)=\pi$, then there is a neighborhood $U_0$ of $x_0$ in
$\mathbb{R}^{3}$ such that $E\cap U_0=\partial \Omega\cap U_0$. In the next
content of this section, we put ourself in the case 
$\Theta_E(x_0)=3\pi/2$ or $7\pi/4$. 

\begin{lemma}\label{le:diffeodo}
	There exist $r_0=r_0(x_0)>0$ and a mapping 
	$\Psi=\Psi_{x_0}:B(0,r_0)\to\mathbb{R}^{3}$,
	which is a diffeomorphism of class $C^{1,\alpha}$ from $B(0,r_0)$ to
	$\Psi(B(0,r_0))$, such that  
	\[
		\Psi(0)=x_0, \Psi(\Omega_0\cap B_{r_0})\subseteq \Omega\cap B(x_0,R_0), 
		\Psi(L_0\cap B_{r_0})\subseteq \partial \Omega\cap B(x_0,R_0),
	\]
	and that $D\Psi(0)$ is a rotation satisfying that 
	\[
		D\Psi(0)(\Omega_0)=\Tan(\Omega,x_0)\text{ and
		}D\Psi(0)(L_0)=\Tan(\partial\Omega,x_0).
	\]
\end{lemma}
\begin{proof}
	By definition, there are an open set $U,V \subseteq \mathbb{R}^{3}$ and a
	diffeomorphism $\Phi:U\to V$ of class $C^{1,\alpha}$ such that $x_0\in U$,
	$0=\Phi(x_0)\in V$ and 
	\[
		\Phi(U\cap \partial \Omega)=Z\cap V,
	\]
	where $Z$ is a plane through $0$. Indeed, we have that
	\[
		Z=D\Phi(x_0)\Tan(\partial\Omega,x_0)
	\]
	and 
	\[
		\Phi(U\cap \Omega)=V\cap D\Phi(x_0)\Tan(\Omega,x_0).
	\]

	We will denote by $A$ the linear mapping given by
	$A(v)=D\Phi(x_0)^{-1}v$, and assume that $A(V)=B(0,r)$ is a ball. Let 
	$\Phi_1$ be a rotation such that $\Phi_1(\Tan(\partial\Omega,x_0))=L_0$ and
	$\Phi_1(\Tan(\Omega,x_0))=\Omega_0$. Then we get that $\Phi_1\circ A\circ 
	\Phi$ is also $C^{1,\alpha}$ mapping which is a diffeomorphism between $U$ 
	and $B(0,r)$, 
	\[
		\begin{gathered}
			D(\Phi_1\circ A\circ \Phi)(x_0)\Tan(
			\Omega,x_0)=\Phi_1(\Tan(\Omega,x_0))=\Omega_0,\\
			D(\Phi_1\circ A\circ \Phi)(x_0)\Tan(\partial
			\Omega,x_0)=\Phi_1(\Tan(\partial \Omega,x_0))=L_0,
		\end{gathered}
	\]
	and 
	\[
		\begin{gathered}
			\Phi_1\circ A\circ \Phi(U\cap \partial \Omega)=\Phi_1\circ A(Z\cap V)=
			L_0\cap B(0,r),\\
			\Phi_1\circ A\circ \Phi(U\cap \partial \Omega)=\Phi_1\circ 
			A(V\cap D\Phi(x_0)\Tan(\Omega,x_0))=\Omega_0\cap B(0,r).
		\end{gathered}
	\]
	We now take $r_0=r$ and $\Psi=(\Phi_1\circ A\circ \Phi)^{-1}\vert_{B(0,r)}$
	to get the result.

\end{proof}

Let $U\subseteq \mathbb{R}^{n}$ be an open set. For any mapping $\Psi:U\to
\mathbb{R}^{n}$ of class $C^{1,\alpha}$, we will denote by $C_{\Psi}$ the
constant  $C_{\Psi}=\sup\left\{\|D\Psi(x)-D\psi(y)\|/|x-y|^{\alpha}:
x,y\in U,x\neq y\right\}$.
%\begin{equation}\label{eq:diffeapp}
%	C_{\Psi}=\sup\left\{\frac{\|D\Psi(x)-D\psi(y)\|}{|x-y|^{\alpha}}:
%	x,y\in U,x\neq y\right\}.
%\end{equation}
Then we have that 
\[
	\Psi(x)-\Psi(y)=\left\langle x-y,\int_{0}^{1}D\Psi(y+t(x-y))dt\right\rangle,
\]
and thus
\[
	|\Psi(x)-\Psi(y)-D\Psi(y)(x-y)|\leq 
	|x-y|\int_{0}^{1}C_{\Psi}(t|x-y|)^{\alpha}dt \leq 
	\frac{C_\Psi}{\alpha+1}|x-y|^{1+\alpha}.
\]

For any $0<\rho\leq r_0$, we set $U_{\rho}=\Psi(B_{\rho})$, 
$M_{\rho}=\Psi^{-1}(E\cap U_{\rho})$ and
\begin{equation}\label{eq:defL}
	\Lambda(\rho)=\max\left\{ \Lip\left(\Psi_{B_{\rho}}\right),
	\Lip\left(\Psi^{-1}_{U_{\rho}}\right)\right\}.
\end{equation}
Then 
\[
	\|D\Psi(0)\|-\|D\Psi(x)-D\Psi(0)\|\leq \|D\Psi(x)\|\leq
	\|D\Psi(0)\|+\|D\Psi(x)-D\Psi(0)\|,
\]
thus $1-C_{\Psi}\rho^{\alpha}\leq \|D\Psi(x)\|\leq 1+C_{\Psi}\rho^{\alpha}$
for $x\in B_{\rho}$, and we have that 
\begin{equation}\label{eq:diffeapp}
	\Lambda(\rho)\leq 1/(1-C_{\Psi}\rho^{\alpha}) \text{ whenever }
	C_{\Psi}\rho^{\alpha}<1.
\end{equation}
\begin{lemma}\label{le:localmini}
	For any $1<\rho\leq \min\{r_0, C_{\Psi}^{-1/\alpha}\}$, $M_{\rho}$ is local
	almost minimal in $B_{\rho}$ at $0$ with gauge function $H$ satisfying that  
	\[
		H(2r)\leq 4 \Lambda(r)^2h(2\Lambda(r)r)+ 
		4\xi_1 C_{\Psi} \Lambda(\rho)r^{\alpha} \text{ for }0<r<(1-C_{\Psi}\rho^{\alpha})\delta_1.
	\]
\end{lemma}
\begin{proof}
	For any open set $U\subseteq \mathbb{R}^3$, $M\geq 1$, $\delta>0$ and
	$\epsilon>0$, we let $GSAQ(U,M,\delta,\epsilon)$ be the collection of
	generalized sliding Almgren quasiminimal sets which is defined in Definition
	2.3 in \cite{David:2014}.
	We see that 
	\[
		\diam(U_{\rho})\leq 2\rho \Lip\left( \Psi\vert_{B_{\rho}} \right)\leq
		2\rho \Lambda(\rho)
	\]
	and 
	\[
		E\cap U_{\rho}\in GSAQ(U_{\rho},1,\diam(U_{\rho}),h(2\diam(U_{\rho}))),
	\]
	By Proposition 2.8 in \cite{David:2014}, we have that 
	\[
		M_{\rho}\in GSAQ\left(B_{\rho},\Lambda(\rho)^4,2\rho,\Lambda(\rho)^4
		h\left(2\rho\Lambda(\rho)\right)\right)
	\]
	By Proposition 4.1 in \cite{David:2014}, we get that $M_\rho$ is Ahlfors-regular
	in $B_{\rho}$. Indeed, we can get a little more, that is, for any $x\in
	M_{\rho}$ with $0<r\Lambda(\rho)<\delta_1$ and $B(x,r)\subseteq 
	B(0,\rho)$, we have that 
	\begin{equation}\label{eq:localmini9}
		\left( \xi_1 \Lambda(\rho) \right)^{-1}r^2\leq \HM^2(M_{\rho}\cap
		B(x,r))\leq \left( \xi_1 \Lambda(\rho) \right)r^2.
	\end{equation}

	Let $\{ \varphi_t \}_{0\leq t\leq 1}$ be any sliding deformation of $M_{\rho}$ 
	in $B_{r}$. Then  
	\[
		\left\{ \Psi\circ\varphi_t\circ\Psi^{-1} \right\}_{0\leq t\leq 1}
	\]
	is a sliding deformation of $E$ in $U_{r}$. Hence we get that 
	\begin{equation}\label{eq:localmini10}
		\HM^2(E\cap U_{r})\leq \HM^2(\Psi\circ\varphi_1\circ\Psi^{-1}(E\cap
		U_{r}))+ h(2\diam (U_{r}))^2\diam(U_{r})^2
	\end{equation}

	For any $2$-rectifiable set $A\subseteq B_{\rho}$, by Theorem 3.2.22 in
	\cite{Federer:1969}, we have that 
	\[
		\ap J_2(\Psi\vert_{A})(x)=
		\left\|\wedge_2\left( D\Psi(x)\vert_{\Tan(A,x)} \right)\right\|
	\]
	and
	\[
		\HM^2(\Psi(A\cap B_{r}))=\int_{A\cap B_{r}}\ap
		J_2(\Psi\vert_{A})(x) d\HM^2(x)
	\]
	By \eqref{eq:diffeapp}, we get that 
	\begin{equation}\label{eq:localmini11}
		\int_{A\cap B_{r}}(1-C_{\Psi}|x|^{\alpha})^2 d\HM^2\leq 
		\HM^2(\Psi(A\cap B_{r}))\leq
		\int_{A\cap B_{r}}(1+C_{\Psi}|x|^{\alpha})^2 d\HM^2.
	\end{equation}
	Thus, by taking $A=M_{\rho}$, we have that $M_r=M_{\rho}\cap B_r$,
	$\Psi(M_r)=E\cap U_r$ and 
	\[
		\HM^2(\Psi(M_{r}))\geq (1-C_{\Psi}\rho^{\alpha})^2
		\HM^2(M_{r}) ;
	\]
	by taking $A=\varphi_1(M_{\rho})$, we have that  
	\[
		\HM^2(\Psi(\varphi_1(M_{\rho})\cap B_r))\leq
		(1+C_{\Psi}r^{\alpha})^2\HM^2(\varphi_1(M_{\rho})\cap B_r) .
	\]
	Combine these two equations
	with \eqref{eq:localmini10} and \eqref{eq:localmini9}, we get that 
	\[
		\begin{aligned}
			\HM^2(\varphi_1(M_{\rho})\cap B_r)&\geq
			(1+C_{\Psi}r^{\alpha})^{-2}\HM^2(\Psi(\varphi_1(M_{\rho})\cap B_r)) \\
			&\geq (1+C_{\Psi}r^{\alpha})^{-2} \left(\HM^2(E\cap U_{r}) - h(4r
			\Lambda(r))(2 r\Lambda(r))^2\right) \\
			&\geq
			\left(\frac{1-C_{\Psi}\rho^{\alpha}}{1+C_{\Psi}r^{\alpha}}\right)^2\HM^2(M_r)
			-\left(\frac{2 r\Lambda(r)}{1+C_{\Psi}r^{\alpha}}\right)^2h(4 r\Lambda(r))\\
			&\geq \HM^2(M_r)-H(2r)r^2.
		\end{aligned}
	\]
\end{proof}
\begin{lemma}\label{le:diffeohd}
	Let $E_1\subseteq \Omega_0$ be a $2$-rectifiable set, $x\in E_1$,
	$X$ a cone centered at $0$, $\Phi:\mathbb{R}^{3}\to\mathbb{R}^{3}$ a
	diffeomorphism of class $C^{1,\alpha}$. Then there exist $C>0$
	such that, for any $r>0$ and $\rho>0$ with $B(\Phi(x),\rho)\subseteq
	\Phi(B(x,r))$,
	\[
		d_{\Phi(x),\rho}\left(\Phi(E_1),\Phi(x)+D\Phi(x)X\right)\leq 
		\left(Cr^{\alpha}+\|D\Phi(x)\| d_{x,r}(E_1,x+X)\right)\frac{r}{\rho}.
	\]
\end{lemma}
\begin{proof}
	Since $\Phi$ is of class $C^{1,\alpha}$, we have that
	\[
		|\Phi(y)-\Phi(x)-D\Phi(x)(y-x)|\leq \frac{C_{\Phi}}{\alpha+1}|x-y|^{1+\alpha},
	\]
	by putting $C_1=C_{\Phi}/(\alpha+1)$, we get that 
	\[
		\dist(\Phi(y),\Phi(x)+D\Phi(x)X)\leq C_1 |y-x|^{1+\alpha} \text{ for }
		y\in x+X.
	\]
	For any $z\in E_1\cap B_r$ and $y\in x+X$, we have that 
	\[
		\begin{aligned}
			|\Phi(z)-\Phi(y)|&\leq
			|\Phi(z)-\Phi(y)-D\Phi(x)(z-y)|+\|D\Phi(x)\|\cdot|z-y|\\
			&\leq \|D\Phi(x)\|\cdot|z-y| + C_1|z-x|^{1+\alpha}+C_{1}|y-x|^{1+\alpha},
		\end{aligned}
	\]
	thus
	\[
		\dist(\Phi(z),\Phi(x+X))\leq
		\|D\Phi(x)\|rd_{x,r}(E_1,x+X)+2C_1r^{1+\alpha},
	\]
	hence 
	\begin{equation}\label{eq:diffeohd10}
		\dist(\Phi(z),\Phi(x)+D\Phi(x)X)\leq 
		\|D\Phi(x)\|rd_{x,r}(E_1,x+X)+3C_1r^{1+\alpha}.
	\end{equation}

	For any $z\in X\cap B_r$, $\Phi(x)+D\Phi(x)z\in \Phi(x)+D\Phi(x)X$, and 
	\begin{equation}\label{eq:diffeohd11}
		\begin{aligned}
			\dist(\Phi(x)+D\Phi(x)z, \Phi(E_1)) &=\inf
			\{ |\Phi(y)-\Phi(x)-D\Phi(x)z|:y\in E_1 \}\\
			&\leq \inf\{C_1r^{1+\alpha}+\|D\Phi(x)\|\cdot|y-x-z|:y\in E_1\}\\
			&\leq \|D\Phi(x)\|rd_{x,r}(x+X,E_1)+C_1r^{1+\alpha}.
		\end{aligned}
	\end{equation}

	We get from \eqref{eq:diffeohd10} and \eqref{eq:diffeohd11} that 
	\[
		d_{\Phi(x),\rho}(\Phi(E_1),\Phi(x)+D\Phi(x)X)\leq \frac{r}{\rho}\left(
		3C_1r^{\alpha}+\|D\Phi(x)\|\cdot d_{x,r}(E_1,x+X) \right)
	\]

\end{proof}
\begin{theorem}\label{thm:ut}
	Let $\Omega$, $E\subseteq \Omega$, $x_0\in \partial\Omega$ and $h$ be the same as in
	the beginning of this section. Then there is a unique blow-up limit $X$ of $E$
	at $x_0$; moreover, if the gauge function $h$ satisfy that 
	\begin{equation}\label{eq:ut0}
		h(t)\leq C_h t^{\alpha_1} \text{ for some }C_h>0, \alpha_1>0
		\text{ and }0<t<t_0,
	\end{equation}
	then there exists $\rho_0>0$ such that, for any
	$0<\beta<\min\{ \alpha,\alpha_1,2\lambda_0 \}$,
	\[
		d_{x_0,\rho}(E,x_0+X)\leq C(\rho/\rho_0)^{\beta/4},\ 0<\rho\leq 9\rho_0/20,
	\]
	where $C$ is a constant satisfying that 
	\[
		C\leq C_{20}(\mu,\lambda_0,\alpha,\alpha_1,\beta,\xi_1)
		(F_E(x_0,2\rho_0)+ C_{\Psi}\rho_0^{\alpha}+C_h\rho_0^{\alpha_1})^{1/4},
	\]
	and $F_E(x_0,r)=r^{-2}\HM^2(E\cap B(x_0,r))-\Theta_E(x_0)+16h_1(r)$.
\end{theorem}
\begin{proof}
	Let $r\in (0,r_0)$ be such that $C_{\Psi}r^{\alpha}\leq 1/2$ and
	$2r\leq R_0$. Then $\Lambda(r)\leq 2$. By Lemma \ref{le:localmini}, we have
	that $M_r$ is local almost minimal at 0 with gauge function $H$ satisfying 
	that 
	\begin{equation}\label{eq:newgauge2}
		H(t)\leq 16 h(2t)+C_rt^{\alpha} ,\ 0<t<r,
	\end{equation}
	where $C_r\in (0,	2^{3-\alpha}\xi_1C_{\Psi})$ is a constant.

	We put $f_{M_r}(\rho)=\Theta_{M_r}(0,\rho)-\Theta_{M_r}(0)$. Then we get, 
	from \eqref{eq:denapp10}  and \eqref{eq:denapp50}, that 
	\[
		\begin{aligned}
			f_{M_r}(\rho)&\leq \left( r^{-2\lambda_0}f_{M_r}(r) \right)
			\rho^{2\lambda_0}+8(1+\lambda_0)\rho^{2\lambda_0}\int_{\rho}^{r}
			t^{-1-2\lambda_0}H(2t)dt\\
			&\leq \left(r^{-2\lambda_0}f_{M_r}(r)\right) \rho^{2\lambda_0}+
			2^{7+2\lambda_0}(1+\lambda_0) \rho^{2\lambda_0}
			\int_{2\rho}^{2 r} \frac{h(2t)}{t^{1+2\lambda_0}}dt \\
			&\quad + 2^{\alpha+3}(1+\lambda_0)C_r\cdot
			C_1(\alpha,\beta,\lambda_0)r^{\alpha}\cdot (\rho/r)^{\beta},
		\end{aligned}
	\]
	where $C_1(\alpha,\beta,\lambda_0)$ is the constant in \eqref{eq:denapp50}.

	We get from \eqref{eq:newgauge2} that 
	\[
		H_1(\rho)=\int_{0}^{\rho}\frac{H(2s)}{s}ds\leq 
		16h_1(2\rho)+\frac{C_r}{\alpha}(2\rho)^{\alpha},
	\]
	by setting $F_1(\rho)=f_{M_r}(\rho)+16H_1(\rho)$, we have that
	\[
		\begin{aligned}
			F_1(\rho)&\leq C_{12}(\lambda_0,\alpha,\beta,r)(\rho/r)^{\beta} +
			2^8h_1(2\rho)+2^{4+\alpha}C_r\alpha^{-1}\rho^{\alpha}\\
			&\quad + 2^{7+2\lambda_0}(1+\lambda_0)
			\rho^{2\lambda_0}
			\int_{2\rho}^{2 r} \frac{h(2t)}{t^{1+2\lambda_0}}dt,
		\end{aligned}
	\]
	where 
	\[
		C_{12}(\lambda_0,\alpha,\beta,r)\leq  f_{M_r}(r)+
		2^{\alpha+3}(1+\lambda_0)C_rC_1(\alpha,\beta,\lambda_0)
		r^{\alpha}.
	\]
	Hence 
	\[
		\begin{aligned}
			\int_{0}^{t}\frac{F_1(\rho)^{1/2}}{\rho}d\rho &\leq 
			C_{12}(\lambda_0,\alpha,\beta,r)^{1/2}(2/\beta) (t/r)^{\beta} +
			16 h_2(2t)
			+C_{13}(\alpha,r)t^{\alpha/2}\\
			& + 2^{4+\lambda_0}(1+\lambda_0)^{1/2}\int_{0}^{t} 
			\rho^{-1+\lambda_0}\left(\int_{2\rho}^{2r}\frac{h(2s)}
			{s^{1+2\lambda_0}}ds \right)^{1/2}d\rho ,
		\end{aligned}
	\]
	where $C_{13}(\alpha,r)\leq 2^{3+\alpha/2}\alpha^{-3/2}C_r^{1/2}$,
	thus 
	\[
		\int_{0}^{t}\frac{F_1(\rho)^{1/2}}{\rho}d\rho <+\infty,\text{ for
		}0<t\leq r.
	\]

	We now apply Theorem \ref{thm:distTC}, there is a unique tangent cone $T$ of
	$M_r$ at $0$, thus there is a unique tangent cone $X$ of $E$ at $x_0$.

	For any $R\in (0,R_0)$, we put 
	\[
		f_E(x_0,R)=R^{-2}\HM^2(E\cap B(x_0,R))-\Theta_E(x_0)
	\]
	and 
	\[
		F_E(x_0,R)=f_E(x_0,R)+16h_1(R).
	\]
	We see, from \eqref{eq:localmini10}  and $B(x_0,\rho/\Lambda(\rho))\subseteq
	U_{\rho}\subseteq B(x_0,\rho\Lambda(\rho))$, that  
	\[
		(1-C_{\Psi}\rho^{\alpha})^2(f_{M_r}(\rho)+\Theta_E(x_0))\leq 
		\rho^{-2}\HM^2(E\cap U_{\rho})\leq
		(1+C_{\Psi}\rho^{\alpha})^2(f_{M_r}(\rho)+\Theta_E(x_0)),
	\]
	so that 
	\begin{equation}\label{eq:cME10}
		f_{M_r}(\rho)\leq (1-C_{\Psi}\rho^{\alpha})^{-4}f_E(x_0,\rho\Lambda(\rho))+
		4\Theta_E(x_0)C_{\Psi}\rho^{\alpha},
	\end{equation}
	and 
	\begin{equation}\label{eq:cME20}
		f_{M_r}(\rho)\geq (1-C_{\Psi}^2\rho^{2\alpha})^2f_E(x_0,\rho/\Lambda(\rho))+
		2\Theta_E(x_0)C_{\Psi}^2\rho^{2\alpha}.
	\end{equation}
	Thus we get that 
	\[
		C_{12}(\lambda_0,\alpha,\beta,r)\leq  16 f_E(x_0,2r)+ (9\xi_1\cdot 2^{\alpha+3}
		(1+\lambda_0) C_1(\alpha,\beta,\lambda_0)+4\Theta_E(0))C_{\Psi}r^{\alpha}.
	\]

	If $h$ satisfy \eqref{eq:ut0}, we take $0<\rho_0\leq \min\{ r,t_0\}$, then 
	\[
		h_1(\rho)\leq \frac{C_h}{\alpha_1}(2\rho)^{\alpha_1},\ H_1(\rho)\leq
		\frac{2^{4+2\alpha_1}C_h}{\alpha_1}\rho^{\alpha_1}+
		\frac{2^{\alpha}C_r}{\alpha}\rho^{\alpha},\ 0<\rho\leq \rho_0,
	\]
	and 
	\begin{equation}\label{eq:ut10}
		F_1(\rho)\leq C_{13}(\lambda_0,\alpha,\beta,\rho_0,C_h)(\rho/\rho_0)^{\beta}+
		2^{8+\alpha_1}\alpha_1^{-1}C_h\rho^{\alpha_1}+
		C_{14}(\alpha,\xi_1,C_{\Psi})\rho^{\alpha},
	\end{equation}
	where $C_{13}(\lambda_0,\alpha_1,\beta,\rho_0,C_h)$ and $C_{14}(\alpha,\xi_1,C_{\Psi})$ 
	are constant satisfying that 
	\[
		C_{13}(\lambda_0,\alpha_1,\beta,\rho_0,C_h)\leq
		C_{12}(\lambda_0,\alpha,\rho_0)+ 2^{7+4\alpha_1}(1+\lambda_0)
		C_1(\alpha_1,\beta,\lambda_0)C_h\rho_0^{\alpha_1}
	\]
	and 
	\[
		C_{14}(\alpha,\xi_1,C_{\Psi})\leq 2^{8+\alpha}\alpha^{-1}\xi_1C_{\Psi}.
	\]
	We get so that \eqref{eq:ut10} can be rewrite as 
	\[
		F_{1}(\rho)\leq C_{15}(\lambda_0,\alpha,\alpha_1,\beta,\xi_1)
		(F_E(x_0,2\rho_0)+C_{\Psi}\rho_0^{\alpha}+C_h\rho_0^{\alpha_1})
		(\rho/\rho_0)^{\beta/4}.
	\]
	By Theorem \ref{thm:distTC}, we have that 
	\[
		\begin{aligned}
			d_{0,9\rho/10}(M_r,T)&\leq C_{16}(\mu,\xi_0) \left(F_{1}(\rho)^{1/4}+
			\int_{0}^{\rho}\frac{F_{1}(t)^{1/2}}{t}dt\right)\\
			&\leq C_{17}(\mu,\lambda_0,\alpha,\alpha_1,\beta,\xi_1)
			G_E(x_0,\rho_0) (\rho/\rho_0)^{\beta/4},
		\end{aligned}
	\]
	where
	\[
		G_E(x_0,\rho_0)=
		(F_E(x_0,2\rho_0)+C_{\Psi}\rho_0^{\alpha}+C_h\rho_0^{\alpha_1})^{1/4}.
	\]
	Apply Lemma \ref{le:diffeohd}, and by setting $X=D\Psi(0)T$, we get that, 
	for any $\rho\in (0,9\rho_0/10)$, 
	\[
		\begin{aligned}
			d_{x_0,\rho/2}(E,x_0+X)&\leq
			d_{x_0,\rho/\Lambda(\rho)}(E,x_0+D\Psi(0)T)\\
			&\leq 6C_{\Psi}\rho^{\alpha}+2d_{x,\rho}(M_r,T)\\
			&\leq 6C_{\Psi}\rho^{\alpha}+
			C_{18}(\mu,\lambda_0,\alpha,\alpha_1,\beta,\xi_1)G_E(x_0,\rho_0)
			(\rho/\rho_0)^{\beta/4}\\
			&\leq C_{19}(\mu,\lambda_0,\alpha,\alpha_1,\beta,\xi_1)
			G_E(x_0,\rho_0) (\rho/\rho_0)^{\beta/4}.
		\end{aligned}
	\]
	The radius $\rho_0$ is chosen to be such that
	\begin{equation}\label{eq:ut100}
		0<\rho_0\leq \min\left\{ 1,t_0,r_0(x_0),R_0/2,(2C_{\Psi})^{-1/\alpha} \right\}
	\end{equation}
	and $R_0>0$ is chosen to be such that 
	\[
		F_{M_r}(R_0)\leq \mu\tau_0/4,\ \bar{\varepsilon}(R_0)\leq 10^{-4},\
		R_0<(1-\tau)\mathfrak{r}.
	\]
\end{proof}

\begin{lemma} \label{le:smalldendecay}
	For any $\tau>0$ small enough, there exists $\varepsilon_2=\varepsilon_2(\tau)>0$
	such that the following hold: $E$ is an sliding almost minimal set in
	$\Omega$ with sliding boundary $\partial\Omega$ and gauge function $h$,
	$x_0\in E\cap \partial\Omega$, $\Psi$ is a mapping as in Lemma
	\ref{le:diffeodo} and $C_{\Psi}$ is the constant as in \eqref{eq:diffeapp},
	if $r_1>0$ satisfy that $C_{\Psi}r_1^{\alpha}\leq \varepsilon_2$, 
	$h(2r_1)\leq \varepsilon_2$ and $F_E(x_0,r_1)\leq \varepsilon_2$, then for any 
	$r\in (0,9r_1/10)$, we can find sliding minimal cone $Z_{x_0,r}$ in 
	$\Tan(\Omega,x_0)$ with sliding boundary $\Tan(\partial\Omega,x_0)$ 
	such that 
	\[
		\begin{gathered}
			\dist(x,Z_{x_0,r})\leq \tau r,\ x\in E\cap B(x_0,(1-\tau)r)\\
			\dist(x,E)\leq \tau r,\ x\in Z_{x_0,r}\cap B(x_0,(1-\tau)r),
		\end{gathered}
	\]
	and for any ball $B(x,t)\subseteq B(x_0,(1-\tau)r)$,
	\[
		|\HM^2(Z_{x_0,r}\cap B(x,t))-\HM^2(E\cap B(x,t))|\leq \tau r^2.
	\]
	Moreover, if $E\supseteq \partial\Omega$, then $Z_{x_0,r}\supseteq 
	\Tan(\partial\Omega,x_0)$. 
\end{lemma}
\begin{proof}
	It is a consequence of Proposition 30.19 in \cite{David:2014}.
\end{proof}
\begin{corollary}\label{cor:appbytangentcone}
	Let $\Omega$, $E\subseteq \Omega$, $x_0\in \partial\Omega$, $h$ and $F_E$ be 
	the same as in Theorem \ref{thm:ut}. Suppose that the gauge function $h$
	satisfying
	\begin{equation}\label{eq:ut0}
		h(t)\leq C_h t^{\alpha_1} \text{ for some }C_h>0, \alpha_1>0
		\text{ and }0<t<t_0.
	\end{equation}
	Then there exists $\delta>0$ and constant $C=C_{20}(\mu,\lambda_0,\alpha,\alpha_1,
	\beta,\xi_1)>0$ for $0<\beta<\min\{ \alpha,\alpha_1,2\lambda_0 \}$ such 
	that, whenever $0<\rho_0\leq \min\{ 1,t_0,r_0(x_0),\mathfrak{r}\}$ satisfying 
	\[
		F_E(x_0,2\rho_0)+ C_{\Psi}\rho_0^{\alpha}+C_h\rho_0^{\alpha_1}\leq \delta,
	\]
	we have that, for $0<\rho\leq 9\rho_0/20$,
	\[
		d_{x_0,\rho}(E,x_0+\Tan(E,x_0))\leq C(F_E(x_0,2\rho_0)+ C_{\Psi}
		\rho_0^{\alpha}+C_h\rho_0^{\alpha_1})^{1/4}(\rho/\rho_0)^{\beta/4}.
	\]
\end{corollary}
\begin{proof}
	By Theorem \ref{thm:ut}, there exist $\rho_0>0$ such that 	
	\[
		d_{x_0,\rho}(E,x_0+\Tan(E,x_0))\leq C(\rho/\rho_0)^{\beta/4},
		\ 0<\rho\leq 9\rho_0/20,
	\]
	where $\rho_0>0$ is chosen to be such that 
	\begin{equation}\label{eq:utcor}
		0<\rho_0\leq \min\left\{ 1,t_0,r_0(x_0),R_0/2,(2C_{\Psi})^{-1/\alpha} \right\}
	\end{equation}
	and $R_0>0$ is chosen to be such that 
	\[
		F_{M_r}(R_0)\leq \mu\tau_0/4,\ \bar{\varepsilon}(R_0)\leq 10^{-4},\
		R_0<(1-\tau)\mathfrak{r}.
	\]
	By Lemma \ref{le:smalldendecay}, there exists $\delta>0$ such that if 
	$F_E(x_0,2\rho_0)+C_{\Psi}\rho_0^{\alpha}+C_h\rho_0^{\alpha_1}\leq \delta$,
	then \eqref{eq:utcor} holds, and we get the result.
\end{proof}

\begin{lemma}\label{le:np1}
	Let $\Omega,E$ and $h$ be the same as in Theorem \ref{thm:ut}. We have that
	\[
		\overline{E\setminus \partial \Omega}\in SAM(\Omega, \partial \Omega,h).
	\]
\end{lemma}
\begin{proof}
	We will put $E_1=\overline{E\setminus \partial \Omega}$ for convenient.
	We first show that $\HM^2(E_1\cap \partial \Omega)=0$. Indeed, for any
	$x\in E_1\cap \partial \Omega$, $\Theta_{E}(x)\geq 3\pi/2$. It follows
	from the fact that for $\HM^2$-a.e. $x\in E$, $\Theta_E(x)=\pi$ that 
	$\HM^2(E_1\cap \partial \Omega)=0$.

	Let $\{\varphi_t\}_{0\leq t\leq 1}$ be any sliding deformation in some ball
	$B=B(y,r)$. Since $E\supseteq \partial \Omega$ and $E\in SAM(\Omega,
	\partial \Omega,h)$, we have that  
	\[
		\begin{aligned}
			\HM^2(E_1)&=\HM^2(E\setminus \partial \Omega)\leq
			\HM^2(\varphi_1(E)\setminus \partial \Omega)+4h(2r)r^2\\
			&= \HM^2(\varphi_1(E_1)\setminus \partial \Omega)+4h(2r)r^2\\
			&\leq \HM^2(\varphi_1(E_1))+4h(2r)r^2.
		\end{aligned}
	\]
	Thus $E_1\in SAM(\Omega, \partial \Omega,h)$.
\end{proof}
\begin{lemma}\label{le:np100}
	Let $\Omega,E$, $x_0$ and $h$ be the same as in Theorem \ref{thm:ut}. For
	ant $\varepsilon>0$ small enough, there exists a $\rho_0>0$ such that for
	any $0<\rho<\rho_0$ and  $x\in E\cap B(x_0,\rho)$, there exists $x_1\in
	B(x_0,5\rho)\cap \partial\Omega$ with $x_1\in \overline{E\setminus\Omega}$ 
	such that  
	\[
		|x-x_1|\leq (1+\varepsilon)\dist(x,\partial\Omega).
	\]
\end{lemma}
\begin{proof}
	If $\Theta_E(x_0)=\pi$, then there is an open ball $B=B(x_0,r)$ such 
	that $E\cap B=\partial \Omega\cap B$, and we have nothing to prove.

	We assume that $\Theta_E(x_0)=3\pi/2$ or $7\pi/4$. We put
	$E_1=\overline{E\setminus \partial \Omega}$. Then $x_0\in E_1$ and
	$\Theta_E(x_0)=\pi/2$ or $3\pi/4$, and by Lemma \ref{le:np1}, we have that
	$E_1\in SAM(\Omega,\partial \Omega,h)$. By Lemma \ref{le:smalldendecay}, 
	for any $\varepsilon\in (0,10^{-3})$, there exists $\rho_0\in (0,r_0)$ such 
	that, for any $0<\rho<\rho_0$, we can find sliding minimal cone $Z_{\rho}$ 
	centered at $x_0$ of type $\mathbb{P}_+$ or $\mathbb{Y}_+$ satisfying that 
	\[
		d_{x_0,\rho}(E_1,Z_{\rho})\leq \varepsilon.
	\]
	Let $\Psi:B(0,r_0)\to \mathbb{R}^3$ be the mapping defined in Lemma
	\ref{le:diffeodo}, and let $\Lambda$ be the same as in \eqref{eq:defL}.
	We put $U_{\rho}=\Psi(B_{\rho})$, $A_1=\Psi^{-1}(E_1\cap
	U_{\rho_0})$. By Lemma \ref{le:diffeohd}, for any $0<r\leq 
	\rho/\Lambda(\rho)$, there exist sliding minimal cone $X_{r}$ in
	$\Omega_0$ such that  
	\[
		d_{0,r}(A_1,X_{r})\leq (C\rho^{\alpha}+\varepsilon)\frac{\rho}{r}.
	\]
	Thus there exists $\rho_1>0$ such that for any $0<r\leq \rho_1$, we can find
	sliding minimal cone $X_r$ of type $\mathbb{P}_+$ or $\mathbb{Y}_+$ such
	that 
	\[
		d_{0,r}(A_1,X_r)\leq 2 \varepsilon.
	\]

	Using the same argument as in the proof Lemma 5.4 in \cite{Fang:2015}, we
	get that there exists $\rho_2>0$ such that for any $x\in A_1\cap
	B(0,\rho)$ with $0<\rho\leq \rho_2$, we can find $a\in A_1\cap L_0\cap 
	B(0,3\rho)$ such that 
	\[
		|P_{L_0}(x)-a|\leq 8 \varepsilon |x-a|,
	\]
	where we denote by $P_{L_0}$ the orthogonal projection from 
	$\mathbb{R}^3$ to $L_0$. Thus 
	\[
		|x-a|\leq |x-P_{L_0}(x)|+|P_{L_0}(x)-a|\leq \dist(x,L_0)+8
		\varepsilon|x-a|,
	\]
	and we get that 
	\[
		\dist(x,A_1\cap L_0\cap B(0,3\rho))\leq \frac{1}{1-8 \varepsilon}
		\dist(x,L_0\cap B(0,3\rho)).
	\]

	We take $\rho_3=\dist(x_0, \mathbb{R}^3\setminus U_{\rho_2})/10$. Then, for
	any $0<\rho\leq \rho_3$ and $z\in E_1\cap B(x_0,\rho)$,
	\[
		\begin{aligned}
			\dist(z,E_1\cap \partial \Omega\cap B(x_0,5\rho))&\leq
			\Lip\left(\Psi\vert_{B(0,3\rho_2)}\right)\dist(\Psi^{-1}(z),A_1\cap L_0
			\cap B(0,3\rho))\\
			&\leq (1-8 \varepsilon)^{-1}\Lambda (3\rho) \dist(\Psi^{-1}(z),A_1\cap
			L_0\cap B(0,3\rho)) \\
			&\leq (1-8 \varepsilon)^{-1}\Lambda (3\rho)^2\dist(z,\partial
			\Omega \cap B(x_0,5\rho)).
		\end{aligned}
	\]
	We assume $\rho_2$ to be small enough such that $(1-8
	\varepsilon)^{-1}\Lambda (3\rho_2)^2<1+10 \varepsilon$, then 
	\[
		\dist(z,E_1\cap\partial \Omega\cap B(x_0,5\rho) )\leq (1+10 \varepsilon)
		\dist(z,\partial \Omega \cap B(x_0,5\rho)).
	\]

\end{proof}
\begin{lemma}\label{le:appmcp}
	Let $\Omega,E$, $x_0$ and $h$ be the same as in Theorem \ref{thm:ut}.
	Suppose that $\Theta_E(x_0)=3\pi/2$. Then, by putting $E_1=
	\overline{E\setminus \partial \Omega}$, there exist a radius $r>0$, a number
	$\beta>0$ and a constant $C>0$ such that, for any $x\in B(x_0,r)\cap E_1$
	and $0<\rho<2r$, we can find cone $Z_{x,\rho}$ such that 
	\begin{equation}\label{eq:mainappp}
		d_{x,\rho}(E_1,Z_{x,\rho})\leq C\rho^{\beta},
	\end{equation}
	where $Z_{x,\rho}=y+\Tan(E_1,y)$, $y\in E_1\cap B(x,C\rho)$, and $y\in
	E_1\cap \partial \Omega\cap B(x,C\rho)$ 
	in case $\rho\geq \dist(x,\partial \Omega)/10$.
\end{lemma}
\begin{proof}
	We see that $E=E_1\cup \partial \Omega$, and
	$F_E(x_0,\rho)=F_{E_1}(x,\rho)+F_{\partial \Omega}(x_0,r)$.
	By Corollary \ref{cor:appbytangentcone}, there exist $\delta>0$ and $C>0$
	such that whenever  $0<\rho_0\leq \min\{ 1,t_0,r_0(x_0) \}$ satisfying 
	\[
		F_{E_1}(x_0,2\rho_0)+ C_{\Psi_{x_0}}\rho_0^{\alpha}+C_h\rho_0^{\alpha_1}\leq \delta.
	\]
	we have that, for $0<\rho\leq 9\rho_0/20$,
	\[
		d_{x_0,\rho}(E_{1},x_0+\Tan(E_{1},x_0))\leq C\delta^{1/4}(\rho/\rho_0)^{\beta},
	\]
	where $0<\beta<\min\{ \alpha,\alpha_1,2\lambda_0 ,\beta_0\}/4$. We take $\rho_1\in
	(0,\rho_0)$ such that 
	\[
		F_{E_1}(x_0,2\rho)+ C_{\Psi_{x_0}}\rho^{\alpha}+C_h\rho^{\alpha_1}\leq
		\min\{\delta/2,\varepsilon_2(\tau)\}, \forall 0<\rho\leq \rho_1.
	\]

	If $x\in \partial\Omega\cap B(x_0,\rho_1/10)$, we take $t=\rho_1/2$,
	then apply Lemma \ref{le:smalldendecay} with $r=|x-x_0|+t$ to get that 
	\[
		\HM^2(E_{1}\cap B(x,t))\leq \HM^2(Z_{x,r}\cap B(x,t))+\tau r^2,
	\]
	thus
	\[
		\Theta_{E_1}(x,t)\leq \frac{1}{t^2}\HM^2(Z_{x,r}\cap B(x,t))+4\tau\leq
		\frac{\pi}{2}+C_{\Psi_{x_0}}r^{\alpha}+4\tau,
	\]
	and 
	\[ 
		F_{E_1}(x,t)\leq C_{\Psi_{x_0}}r^{\alpha}+4\tau +16 h_1(t). 
	\]
	We get that $F_{E_1}(x,2\rho)+
	C_{\Psi_{x}}\rho^{\alpha}+C_h\rho^{\alpha_1}\leq \delta$ for $0<\rho\leq
	t/2$. Thus 
	\begin{equation}\label{eq:bappbtp10}
		d_{x,r}(E_{1},x+\Tan(E_{1},x))\leq C\delta^{1/4} (r/t)^{\beta},\ 0<r<9t/20. 
	\end{equation}

	By Lemma \ref{le:np100}, we assume that for 
	any $x\in E_{1}\cap B(x_0,\rho_1/10)$, there exists $x_1\in E_1\cap 
	B(x_0,\rho_1/2)\cap \partial\Omega$ such that  
	\[
		|x-x_1|\leq 2\dist(x,\partial\Omega).
	\]

	If $x\in E_{1}\cap B(x_0,\rho_1/10)\setminus \partial\Omega$, we take
	$t=t(x)=10^{-3}\dist(x,\partial\Omega)$, then apply Lemma 
	\ref{le:smalldendecay} with $r=|x-x_1|+t$ to get that 
	\[
		\HM^2(E_{1}\cap B(x,t))\leq \HM^2(Z_{x_1,r}\cap B(x,t))+\tau r^2,
	\]
	thus
	\[
		\Theta_{E_1}(x,t)\leq \frac{1}{t^2}\HM^2(Z_{x_1,r}\cap B(x,t))+
		(1+2\cdot 10^3)^2\tau\leq \pi/2+(1+2\cdot 10^3)^2\tau,
	\]
	and 
	\[
		F(x,t)\leq (1+2\cdot 10^3)^2\tau +8h_1(t).
	\]
	By Theorem \ref{thm:ine}, there is a constent $C_1>0$ such that 
	\begin{equation}\label{eq:bappbtp11}
		d_{x,r}(E_{1},x+\Tan(E_{1},x))\leq C_1 (r/t)^{\beta},\ 0<r<t.
	\end{equation}
	Hence we get that 
	\begin{equation}\label{eq:bappbtp12}
		d_{x,r}(E_{1},x+\Tan(E_{1},x))\leq C_2 (r/t_0)^{\beta},\forall x\in E_{1}\cap
		B(x_0,\rho_1/10), 0<r<t_0,
	\end{equation}
	where 
	\[
		t_0=\begin{cases}
			\rho_1/10,& x\in \partial \Omega,\\
			10^{-3}\dist(x,\partial \Omega),&x\notin \partial \Omega.
		\end{cases}
	\]

	We take $0<a<\beta/(1+\beta)$. For any $x\in
	B(x_0,\rho_1/10)\setminus \partial\Omega$, if $r\leq C_3 t_0^{1/(1-a)}$, then 
	we get from \eqref{eq:bappbtp12} that 
	\[
		d_{x,r}(E_{1},x+\Tan(E_{1},x))\leq C_2C_3^{\beta(a-1)}r^{a\beta};
	\]
	if $C_3 t_0^{1/(1-a)}<r<\rho_1/5$, then by \eqref{eq:bappbtp12}, we have that
	\[
		\begin{aligned}
			d_{x,r}(E_{1},x_1+\Tan(E_{1},x_1))&\leq \frac{|x-x_1|+r}{r}
			d_{x_1,|x-x_1|+r}(E_{1},x_1+\Tan(E_{1},x_1))\\
			&\leq C_4\left( 1+\frac{2\cdot 10^3t_0}{r} \right)\left(
			\frac{r+2\cdot 10^3t_0}{\rho_1/2} 
			\right)^{\beta}\\
			&\leq C_5(1+C_6r^{-a})^{\beta+1}r^{\beta}\leq C_7 r^{\beta-a\beta-a}.
		\end{aligned}
	\]
	We get so that, for any $0<\beta_1<\min\{a\beta, \beta-a\beta-a \}$ there is
	a constant $C_8$ such that for any $x\in E_{1}\cap B(x_0,\rho_1/10)$ and
	$0<\rho<\rho_1/5$, we can find cone $Z_{x,\rho}$ such that 
	\[
		d_{x,\rho}(E_{1},Z_{x,\rho})\leq C_8\rho^{\beta_1},
	\]
	where $Z_{x,\rho}=y+\Tan(E_{1},y)$, $y\in E_{1}\cap B(x,C_8\rho)$, and $y\in
	E_{1}\cap \partial \Omega\cap B(x,C_8\rho)$ 
	in case $\rho\geq C_3t_0^{1/(1-a)}$.
\end{proof}
\begin{lemma}\label{le:appmcy}
	Let $\Omega,E$, $x_0$ and $h$ be the same as in Theorem \ref{thm:ut}.
	Suppose that $\Theta_E(x_0)=7\pi/4$. Then, by putting $E_1=
	\overline{E\setminus \partial \Omega}$, there exist a radius $r>0$, a number
	$\beta>0$ and a constant $C>0$ such that, for any $x\in B(x_0,r)\cap E_1$ and
	$0<\rho<2r$, we can find a cone $Z_{x,\rho}$ such that 
	\begin{equation}\label{eq:mainappy}
		d_{x,\rho}(E_1,Z_{x,\rho})\leq C\rho^{\beta},
	\end{equation}
	where $Z_{x,\rho} = y+\Tan(E_1,y) $, $y\in E_1\cap B(x_0,C\rho)$, and
	$y\in E_1\cap\partial
	\Omega\cap B(x_0,C\rho)$ in case $\rho\geq \dist(x,\partial \Omega)/10$.
\end{lemma}
\begin{proof}
	By Corollary \ref{cor:appbytangentcone}, there exist $\delta>0$ and $C>0$
	such that whenever  $0<\rho_0\leq \min\{ 1,t_0,r_0(x_0) \}$ satisfying 
	\[
		F_{E_1}(x_0,2\rho_0)+ C_{\Psi_{x_0}}\rho_0^{\alpha}+C_h\rho_0^{\alpha_1}\leq \delta,
	\]
	we have that, for $0<\rho\leq 9\rho_0/20$,
	\[
		d_{x_0,\rho}(E_1,x_0+\Tan(E_1,x_0))\leq C\delta^{1/4}(\rho/\rho_0)^{\beta},
	\]
	where $0<\beta<\min\{ \alpha,\alpha_1,2\lambda_0 \}/4$. We take $\rho_1\in
	(0,\rho_0)$ such that 
	\[
		F_{E_1}(x_0,2\rho)+ C_{\Psi_{x_0}}\rho^{\alpha}+C_h\rho^{\alpha_1}\leq
		\min\{\delta/2,\varepsilon_2(\tau)\}, \forall 0<\rho\leq \rho_1.
	\]
	If $x\in \partial\Omega\cap B(x_0,\rho_1/10)$, we take $t=|x-x_0|/2$,
	then apply Lemma \ref{le:smalldendecay} with $r=|x-x_0|+t$ to get that 
	\[
		\HM^2(E_1\cap B(x,t))\leq \HM^2(Z_{x,r}\cap B(x,t))+\tau r^2,
	\]
	thus 
	\[
		\Theta_{E_1}(x,t)\leq \frac{1}{t^2}\HM^2(Z_{x,r}\cap B(x,t))+9\tau\leq
		\frac{\pi}{2}+C_{\Psi_{x_0}}r^{\alpha}+9\tau,
	\]
	and 
	\[ 
		F_{E_1}(x,t)\leq C_{\Psi_{x_0}}r^{\alpha}+9\tau +16 h_1(t). 
	\]
	We get that $F_{E_1}(x,2\rho)+
	C_{\Psi_{x}}\rho^{\alpha}+C_h\rho^{\alpha_1}\leq \delta$ for $0<\rho\leq
	t/2$. Thus 
	\begin{equation}\label{eq:bappbty10}
		d_{x,r}(E_1,x+\Tan(E_1,x))\leq C\delta^{1/4} (r/t)^{\beta},\ 0<r<9t/20. 
	\end{equation}

	By Lemma \ref{le:np100}, we assume that for 
	any $x\in E_1\cap B(x_0,\rho_1/10)$, there exists $x_1\in E_1\cap 
	B(x_0,\rho_1/5)\cap \partial\Omega$ such that  
	\[
		|x-x_1|\leq 2\dist(x,\partial\Omega).
	\]

	If $x\in E_1\cap B(x_0,\rho_1/10)\setminus \partial\Omega$, then
	$\Theta_{E_1}(x)=\pi$ or $3\pi/2$. We put $t(x)=\dist(x,\partial\Omega)$.
	If $\Theta_{E_1}(x)=3\pi/2$, we take $t=10^{-3}t(x)$, then apply Lemma 
	\ref{le:smalldendecay} with $r=|x-x_1|+t$ to get that 
	\[
		\HM^2(E_1\cap B(x,t))\leq \HM^2(Z_{x_1,r}\cap B(x,t))+\tau r^2,
	\]
	thus 
	\begin{equation}\label{eq:bappbty30}
		\Theta_{E_1}(x,t)\leq \frac{1}{t^2}\HM^2(Z_{x_1,r}\cap
		B(x,t))+(1+2\cdot 10^3)^2\tau\leq
		\frac{3\pi}{2}+(1+2\cdot 10^3)^2\tau.
	\end{equation}
	and
	\[
		F_{E_1}(x,t)\leq (1+2\cdot 10^3)^2\tau +8h_1(t).
	\]
	By Theorem \ref{thm:ine}, we have that
	\begin{equation}\label{eq:bappbty40}
		d_{x,\rho}(E_1,x+\Tan(E_1,x))\leq C_1 (\rho/t)^{\beta},\ 0<\rho<t.
	\end{equation}

	We put $E_Y=\{x_0\}\cup \{ x\in E\setminus \partial \Omega:\Theta_{E_1}(x)=\pi \}$. 
	If $\Theta_{E_1}(x)=\pi$ and $\dist(x,E_Y)\leq 10^{-2}\dist(x,\partial \Omega)$,
	we take $x_2\in E_Y$ such that $|x-x_2|\leq 2\dist(x,E_Y)$ and
	$t=10^{-1}\dist(x,E_Y)$, then apply Lemma
	7.24 in \cite{David:2009} with $r=|x-x_2|+t$ to get that 
	\[
		\HM^2(E_1\cap B(x,t))\leq \HM^2(Z_{x_2,r}\cap B(x,t))+\tau r^2,
	\]
	thus 
	\begin{equation}\label{eq:bappbty45}
		\Theta_{E_1}(x,t)\leq \frac{1}{t^2}\HM^2(Z_{x_2,r}\cap B(x,t))+400\tau\leq
		\pi+400\tau,
	\end{equation}
	and
	\[
		F_{E_1}(x,t)\leq 4\tau +8h_1(t).
	\]
	By Theorem \ref{thm:ine}, we have that
	\begin{equation}\label{eq:bappbty50}
		d_{x,\rho}(E_1,x+\Tan(E_1,x))\leq C_2 (\rho/t)^{\beta},\ 0<\rho<t.
	\end{equation}

	If $\Theta_{E_1}(x)=\pi$ and $\dist(x,E_Y)> 10^{-2}\dist(x,\partial \Omega)$,
	we take $t=10^{-3}\dist(x,\partial \Omega)$, then apply Lemma
	\ref{le:smalldendecay} with $r=|x-x_1|+t$ to get that 
	\[
		\HM^2(E_1\cap B(x,t))\leq \HM^2(Z_{x_1,r}\cap B(x,t))+\tau r^2,
	\]
	thus 
	\begin{equation}\label{eq:bappbty55}
		\Theta_{E_1}(x,t)\leq \frac{1}{t^2}\HM^2(Z_{x_1,r}\cap
		B(x,t))+(1+2\cdot 10^3)^2\tau\leq\pi+(1+2\cdot 10^3)^2\tau.
	\end{equation}
	and
	\[
		F_{E_1}(x,t)\leq (1+2\cdot 10^3)^2\tau +8h_1(t).
	\]
	By Theorem \ref{thm:ine}, we have that
	\begin{equation}\label{eq:bappbty60}
		d_{x,\rho}(E_1,x+\Tan(E_1,x))\leq C_3 (\rho/t)^{\beta},\ 0<\rho<t.
	\end{equation}

	We get, from \eqref{eq:bappbty10}, \eqref{eq:bappbty40},
	\eqref{eq:bappbty50} and \eqref{eq:bappbty60}, so that 
	\begin{equation}\label{eq:bappbty71}
		d_{x,\rho}(E_1,x+\Tan(E_1,x))\leq C_4 (\rho/t_0)^{\beta},\ x\in E_1\cap 
		B(x_0,\rho_1/10),\ 0<\rho<t_0,
	\end{equation}
	where 
	\[
		t_0=\begin{cases}
			\rho_1/2,& x=x_0,\\
			|x-x_0|/10,& x\in \partial \Omega\setminus \{x_0\},\\
			10^{-3}\dist(x,\partial\Omega),&x\notin \partial \Omega, \Theta_{E_1}(x)=3\pi/2\\
			10^{-1}\min\{10^{-2}\dist(x,\partial\Omega),\dist(x,E_Y)\},&x\notin \partial
			\Omega, \Theta_{E_1}(x)=\pi.\\
		\end{cases}
	\]

	Claim: $E_Y\cap B(x_0,\rho_1/2)$ is a $C^{1}$ curve which is perpendicular
	to $\Tan(\Omega,x_0)$.
	Indeed, by biH\"older regaurity at the boundary, we see that $E_Y\cap
	B(x_0,\rho_1/2)$ is a curve, and by J. Taylor's regularity, we get that 
	$E_Y\cap B(x_0,\rho_1/2)$ is of class $C^{1}$. 

	By the claim, we can assume that, there is a constant $\eta_3>0$ such that  
	\begin{equation}\label{eq:bappbty77}
		\dist(x,\partial\Omega)\geq \eta_3|x-x_0|, \ \forall x\in E_Y\cap
		B(x_0,\rho_1/10).
	\end{equation}

	We fix $0<\beta_1< \beta_2<\beta/(1+\beta)$ such that $\beta_1\leq
	\beta_2\beta/(1+\beta)$.

	By \eqref{eq:bappbty71}, we have that, for any $x\in \partial \Omega \cap B(x_0,\rho_1/10)\setminus \{x_0\}$, and any $0<\rho<|x-x_0|/10$, 
	\[
		d_{x,\rho}(E_1,x+\Tan(E_1,x))\leq C_4 (\rho/t_0)^{\beta}.
	\]
	If $0<\rho\leq C_5|x-x_0|^{1/(1-\beta_1)}$, then 
	\[
		d_{x,\rho}(E_1,x+\Tan(E_1,x))\leq C_4
		(10\rho/|x-x_0|)^{\beta}=C_6\rho^{\beta_1\beta};
	\]
	if $C_5|x-x_0|^{1/(1-\beta_1)}<\rho\leq \rho_1/5$, then
	\[
		\begin{aligned}
			d_{x,\rho}(E_1,x_0+\Tan(E_1,x_0))&\leq
			\frac{|x-x_0|+\rho}{\rho}d_{x_0,|x-x_0|+\rho}(E_1,x_0+\Tan(E_1,x_0))\\
			&\leq (1+C_5^{-1+\beta_1}\rho^{-\beta_1})C_4
			\left(\frac{C_5^{-1+\beta_1}\rho^{1-\beta_1}+\rho}{\rho_1/2}\right)^{\beta}\\
			&\leq C_7\rho^{\beta-\beta_1-\beta\beta_1}.
		\end{aligned}
	\]
	Thus we get that,  for any
	$0<\beta_3\leq\min\{\beta\beta_1,\beta-\beta_1-\beta\beta_1)\}$, there is 
	a constant $C_8$ such that for any $x\in \partial \Omega \cap B(x_0,\rho_1/10)$ and
	$ 0<\rho\leq \rho_1/5$ we can find cone $Z_{x,\rho}$ satisfying that
	\begin{equation}\label{eq:bappbty80}
		d_{x,\rho}(E_1,Z_{x,\rho})\leq C_8\rho^{\beta_3}.
	\end{equation}

	If $x\in E_1\cap B(x_0,\rho_1/10)\setminus \partial \Omega$ and
	$\Theta_{E_1}(x)=3\pi/2$, then  for $0<\rho\leq C_5|x-x_0|^{1/(1-\beta_1)}$,
	we get, from \eqref{eq:bappbty71}, that 
	\[
		d_{x,\rho}(E_1,x+\Tan(E_1,x))\leq C_4
		(10^3\rho/\dist(x,\partial\Omega))^{\beta}=C_9\rho^{\beta_1\beta};
	\]
	and for $C_5|x-x_0|^{1/(1-\beta_1)}<\rho\leq \rho_1/5$, we have that
	\[
		\begin{aligned}
			d_{x,\rho}(E_1,x_0+\Tan(E_1,x_0))&\leq
			\frac{|x-x_0|+\rho}{\rho}d_{x_0,|x-x_0|+\rho}(E_1,x_0+\Tan(E_1,x_0))\\
			&\leq (1+C_5^{-1+\beta_1}\rho^{-\beta_1})C_4
			\left(\frac{C_5^{-1+\beta_1}\rho^{1-\beta_1}+\rho}{\rho_1/2}\right)^{\beta}\\
			&\leq C_{10}\rho^{\beta-\beta_1-\beta\beta_1}.
		\end{aligned}
	\]
	Thus we get that,  for any
	$0<\beta_4\leq\min\{\beta\beta_1,\beta-\beta_1-\beta\beta_1)\}$, there is 
	a constant $C_{11}$ such that for any $x\in E_1 \cap B(x_0,\rho_1/10)\setminus
	\partial \Omega$ with $\Theta_{E_1}(x)=3\pi/2$, and
	$ 0<\rho\leq \rho_1/5$ we can find cone $Z_{x,\rho}$ satisfying that
	\begin{equation}\label{eq:bappbty90}
		d_{x,\rho}(E_1,Z_{x,\rho})\leq C_{11}\rho^{\beta_4}.
	\end{equation}

	If $x\in E_1\cap B(x_0,\rho_1/10)\setminus \Omega$,
	$\Theta_{E_1}(x)=\pi$ and $\dist(x,\partial \Omega)<100\dist(x,E_Y)$, then for any
	$0<\rho<C_9\dist(x,\partial \Omega)^{1/(1-\beta_1)}$, we get, from
	\eqref{eq:bappbty71}, that 
	\begin{equation}\label{eq:bappbty94}
		d_{x,\rho}(E_1,x+\Tan(E_1,x))\leq
		C_4(10^3\rho/\dist(x,\partial\Omega))^{\beta}=C_{12}\rho^{\beta_1\beta};
	\end{equation}
	and for $C_9\dist(x,\partial \Omega)^{1/(1-\beta_1)}\leq \rho\leq \rho_1/5$, in
	case $\rho\leq C_{13}|x-x_0|^{1/(1-\beta_2)}$, we get, from 
	\eqref{eq:bappbty71}, that
	\begin{equation}\label{eq:bappbty95}
		\begin{aligned}
			d_{x,\rho}(E_1,x_1+\Tan(E_1,x_1))&\leq
			\frac{|x-x_1|+\rho}{\rho}d_{x_1,|x-x_1|+\rho}(E_1,x_1+\Tan(E_1,x_1))\\
			&\leq (1+2C_9^{-1+\beta_1}\rho^{-\beta_1})C_4
			\left(\frac{2C_9^{-1+\beta_1}\rho^{1-\beta_1}+\rho}
			{|x_0-x_1|/10}\right)^{\beta}\\
			&\leq C_{14}\rho^{\beta\beta_2-\beta_1-\beta\beta_1};
		\end{aligned}
	\end{equation}
	in case $\rho> C_{13}|x-x_0|^{1/(1-\beta_2)}$, we have that 
	\begin{equation}\label{eq:bappbty96}
		\begin{aligned}
			d_{x,\rho}(E_1,x_0+\Tan(E_1,x_0))&\leq
			\frac{|x-x_0|+\rho}{\rho}d_{x_0,|x-x_0|+\rho}(E_1,x_0+\Tan(E_1,x_0))\\
			&\leq (1+C_{13}^{-1+\beta_2}\rho^{-\beta_2})C_4
			\left(\frac{C_{13}^{-1+\beta_2}\rho^{1-\beta_2}+\rho}{\rho_1/2}\right)^{\beta}\\
			&\leq C_{15}\rho^{\beta-\beta_2-\beta\beta_2}.
		\end{aligned}
	\end{equation}
	If $x\in E_1\cap B(x_0,\rho_1/10)\setminus \partial \Omega$,
	$\Theta_{E_1}(x)=\pi$ and $\dist(x,\partial \Omega)\geq 100\dist(x,E_Y)$, then for any
	$0<\rho<C_{16}\dist(x,E_Y)^{1/(1-\beta_1)}$, we get, from
	\eqref{eq:bappbty71}, that 
	\begin{equation}\label{eq:bappbty97}
		d_{x,\rho}(E_1,x+\Tan(E_1,x))\leq
		C_4(10\rho/\dist(x,E_Y))^{\beta}=C_{17}\rho^{\beta_1\beta},
	\end{equation}
	for $C_{16}\dist(x,E_Y)^{1/(1-\beta_1)}\leq \rho\leq \rho_1/5$, we can find
	$y\in E_Y$ such that $|x-y|\leq 2\dist(x,E_Y)$, in case
	$\rho\leq C_{18}\dist(y,\partial \Omega)^{1/(1-\beta_2)}$, we get, from 
	\eqref{eq:bappbty71}, that
	\begin{equation}\label{eq:bappbty98}
		\begin{aligned}
			d_{x,\rho}(E_1,y+\Tan(E_1,y))&\leq
			\frac{|x-y|+\rho}{\rho}d_{y,|x-y|+\rho}(E_1,y+\Tan(E_1,y))\\
			&\leq (1+2C_{16}^{-1+\beta_1}\rho^{-\beta_1})C_4
			\left(\frac{2C_{16}^{-1+\beta_1}\rho^{1-\beta_1}+\rho}{10^{-3}
			\dist(y,\partial \Omega)}\right)^{\beta}\\
			&\leq C_{19}\rho^{\beta\beta_2-\beta_1-\beta\beta_1};
		\end{aligned}
	\end{equation}
	and in case $\rho> C_{18}\dist(y,\partial \Omega)^{1/(1-\beta_2)}$,
	we have that 
	\[
		|x-x_0|\geq \dist(x,\partial \Omega)\geq 100\dist(x,E_Y)\geq 50 |x-y|,
	\]
	and by \eqref{eq:bappbty77}, 
	\[
		\dist(y,\partial \Omega)\geq \eta_3|y-x_0|\geq \eta_3(|x-x_0|-|x-y|)\geq
		\eta_3\cdot \frac{49}{50}|x-x_0|,
	\]
	thus by \eqref{eq:bappbty71},
	\begin{equation}\label{eq:bappbty99}
		\begin{aligned}
			d_{x,\rho}(E_1,x_0+\Tan(E_1,x_0))&\leq
			\frac{|x-x_0|+\rho}{\rho}d_{x_0,|x-x_0|+\rho}(E_1,x_0+\Tan(E_1,x_0))\\
			&\leq (1+C_{20}^{-1+\beta_2}\rho^{-\beta_2})C_4
			\left(\frac{C_{20}^{-1+\beta_2}\rho^{1-\beta_2}+\rho}{\rho_1/2}\right)^{\beta}\\
			&\leq C_{21}\rho^{\beta-\beta_2-\beta\beta_2}.
		\end{aligned}
	\end{equation}
	We get, from \eqref{eq:bappbty94}, \eqref{eq:bappbty95}, \eqref{eq:bappbty96},
	\eqref{eq:bappbty97},\eqref{eq:bappbty98} and \eqref{eq:bappbty99}, that for any
	$0<\beta_5\leq \min\{\beta\beta_1,\beta\beta_2-\beta_1-\beta\beta_1,
	\beta-\beta_2-\beta\beta_2\}$, there is a constant $C_{22}$ such that for any $x\in E_1\cap B(x_0,\rho_1/10)\setminus \partial \Omega$ with
	$\Theta_{E_1}(x)=\pi$, and $ 0<\rho\leq \rho_1/5$ we can find cone $Z_{x,\rho}$ 
	such that
	\begin{equation}\label{eq:bappbty100}
		d_{x,\rho}(E_1,Z_{x,\rho})\leq C_{22}\rho^{\beta_5}.
	\end{equation}

	Hence we get, from \eqref{eq:bappbty80}, \eqref{eq:bappbty90} and 
	\eqref{eq:bappbty100}, that for any $0<\beta_6\leq 
	\min\{\beta\beta_1,\beta\beta_2-\beta_1-\beta\beta_1,
	\beta-\beta_2-\beta\beta_2\}$, there is a constant $C_{23}>0$ and $C_{24}>0$
	such that for any $x\in E_1\cap B(x_0,\rho_1/10)$ and $ 0<\rho\leq \rho_1/5$ 
	we can find cone $Z_{x,\rho}$  such that
	\begin{equation}\label{eq:bappbty111}
		d_{x,\rho}(E_1,Z_{x,\rho})\leq C_{23}\rho^{\beta_6},
	\end{equation}
	where $Z_{x,\rho}=z+\Tan(E_1,z)$ for some $z\in E_1\cap B(x,C_{24}\rho)$, and
	$z\in E_1\cap \partial \Omega \cap B(x,C_{24}\rho)$ in case $\rho\geq \max\{ 
	C_5|x-x_0|^{1/(1-\beta_1)},C_9\dist(x,\partial
	\Omega)^{1/(1-\beta_1)},C_{18}\dist(y,\partial \Omega)^{1/(1-\beta_2)}\dist(x,\partial \Omega)\}$.
\end{proof}
\begin{corollary}\label{cor:appmc}
	Let $\Omega$, $E$ and $h$ be the same as in Theorem \ref{thm:ut}. Let
	$E_1=\overline{E\setminus \partial \Omega}$ and $x_0\in E_1\cap \partial
	\Omega$. Then there exist a radius $r>0$, a number $\beta>0$ and a 
	constant $C>0$ such that, for any $x\in E_1\cap B(x_0,r)$ and $0<\rho<2r$, 
	we can find cone $Z_{x,\rho}$ such that 
	\begin{equation}\label{eq:mainapp}
		d_{x,\rho}(E_1,Z_{x,\rho})\leq C\rho^{\beta},
	\end{equation}
	where $Z_{x,\rho} = y+\Tan(E_1,y) $, $y\in E_1\cap B(x,C\rho)$, and
	$y\in E_1\cap\partial
	\Omega\cap B(x,C\rho)$ in case $\rho\geq \dist(x,\partial \Omega)/10$.. 
\end{corollary}
\begin{proof}
	It is follow from Lemma \ref{le:appmcp} and Lemma \ref{le:appmcy}.
\end{proof}
\begin{lemma}\label{le:WAPPM}
	Let $\Omega,E$, $x_0$ and $h$ be the same as in Corollary \ref{cor:appmc}.
	Let $\Psi:B(0,r_0)\to \mathbb{R}^3$ be the mapping defined in Lemma
	\ref{le:diffeodo}. Let $R>0$ be such that $\Psi(B(0,R))\subseteq B(x_0,r)$,
	where $B(x_0,r)$ is the ball considered as in Corollary \ref{cor:appmc}.  
	By putting $U=\Psi(B(0,R))$, 
	$M_1=\Psi^{-1}(E_1\cap U)$, we have that there exist $\rho_3>0$, $\beta>0$, and
	constant $C>0$ such that for any $z\in M_1\cap B(0,\rho_3)$ and $0<t<2\rho_3$,
	we can find cone $Z(z,t)$ through $z$ such that 
	\[
		d_{z,t}(M_1,Z(z,t))\leq Ct^{\beta},
	\]
	where $Z(z,t)$ is a minimal cone of type $\mathbb{P}$ or $\mathbb{Y}$ in
	case $z\in M_1\setminus L_0$ and $0<t<\dist(z,L_0)$; and in case $t\geq
	\dist(z,L_0)$ or $z\in L_0$, $Z(z,t)$ is a sliding minimal cone in 
	$\Omega_0$ with sliding boundary $L_0$, if $Z(z,t)\setminus L_0\neq 
	\emptyset$, we can be written as $Z(z,t)=L_0\cup Z$, $Z$ is a slding minimal
	cone of type $\mathbb{P}_+$ or $\mathbb{Y}_+$.
\end{lemma}
\begin{proof}
	For any $x\in B(x_0,r)\cap E_1$ and $0<\rho<2r$, we let $Z_{x,\rho}$ be the
	same cone considered as in Corollary \ref{cor:appmc}.
	We put $\Phi=\Psi^{-1}\vert_{B(x_0,r)}$ and $X=\Tan(E_1,y)$ for convenient.  

	For any $x\in E_1\cap B(x_0,r)$, and any $z\in E_1\cap B(x,\rho)$, we have that  
	\[
		\dist(\Phi(z),\Phi(y+X))\leq
		\Lip(\Phi)\dist(z,y+X)\leq C \Lip(\Phi)\rho^{1+\beta}.
	\]
	Since 
	\[
		|\Phi(z_1)-\Phi(z_2)-D\Phi(z_2)(z_1-z_2)|\leq
		C_1|z_1-z_2|^{1+\alpha},
	\]
	we have that, for any $z_1\in y+X$, 
	\[
		\dist(\Phi(z_1),\Phi(y)+D\Phi(y)X)\leq C_1 |z_1-y|^{1+\alpha}.
	\]
	Hence 
	\begin{equation}\label{eq:pfmth20}
		\dist(\Phi(z),\Phi(y)+D\Phi(y)X)\leq C\Lip(\Phi)\rho^{1+\beta} +
		C_1(\rho+C\rho+C\rho^{1+\beta})^{1+\alpha}\leq C_2 \rho^{1+\beta}.
	\end{equation}

	For any $v\in X$, we see that $\Phi(y)+D\Phi(y)v\in \Phi(y)+D\Phi(y)X$, and
	we have that 
	\[
		\begin{aligned}
			\dist(\Phi(y)+D\Phi(y)v, M_1)&\leq \dist(\Phi(y)+D\Phi(y)v, \Phi(E_1\cap
			B(x,\rho)))\\
			&=\inf\{|\Phi(z)-\Phi(y)-D\Phi(y)v|:z\in
			E_1\cap B(x,\rho)\}\\
			&\leq \inf\{C_1|z-y|^{1+\alpha}+\Lip(\Phi)|z-y-v|:z\in E_1\cap
			B(x,\rho)\}\\
			&\leq C_1(\rho+C\rho)^{1+\alpha}+\Lip(\Phi)\dist(y+v,E_1).
		\end{aligned}
	\]
	Thus there exist $C_3>0$ such that, for any $v\in X$ with $|y+v-x|\leq \rho$,
	\begin{equation}\label{eq:pfmth30}
		\dist(\Phi(y)+D\Phi(y)v, M_1)\leq C_3\rho^{1+\beta}.
	\end{equation}
	We take $0<C_5<C_4<1$ small enough, for example $C_4<(10\Lip(\Phi))^{-1}$,
	then for any $C_5\rho \leq t\leq C_4\rho\leq \rho/\Lip(\Phi)-C_1(C\rho)^{1+\alpha}$, we have 
	that $M_1\cap B(\Phi(x),t)\subseteq \Phi(E_1\cap B(x,\rho))$ and 
	\[
		[\Phi(y)+D\Phi(y)X]\cap B(\Phi(x),t)\subseteq
		\{\Phi(y)+D\Phi(y)v:v\in X,y+v\in B(x,\rho)\}.
	\]
	We get, from \eqref{eq:pfmth20} and \eqref{eq:pfmth30}, so that 
	\begin{equation}\label{eq:pfmth40}
		d_{\Phi(x),t}(M_1,\Phi(y)+D\Phi(y)X)\leq C_6\rho^{\beta}\leq
		C_7t^{\beta},
	\end{equation}
	and 
	\begin{equation}\label{eq:pfmth41}
		|\Phi(x)-\Phi(y)|\leq \Lip(\Phi)|x-y|\leq (\Lip(\Phi)CC_5^{-1})t.
	\end{equation}
	Hence 
	\begin{equation}\label{eq:pfmth45}
		d_{\Phi(x),t}(M_1,\Phi(y)+D\Phi(y)X)\leq
		C_7t^{\beta},\ \text{ for any }0<t<C_4\rho_1,
	\end{equation}
	where $\rho_1\in (0,2r)$ satisfy that  $C_1C^{1+\alpha}\rho_1\leq
	\Lip(\Phi)^{-1}-C_4$. 

	We take $\rho_2>0$ such that, for any $x\in E_1\cap 
	\Phi(B(x_0,\rho_2))$ and $0<\rho<2\rho_2$, $Z_{x,\rho}$ can be expressed as 
	$Z_{x,\rho}=y+\Tan(E_1,y)$ with $y\in E_1\cap U$.  Since $D\Phi(y)X=D\Phi(y)
	\Tan(E_1,y)=\Tan(M_1,\Phi(y))$ in case $y\in E_1\cap U$, by putting 
	$\rho_3=\min\{\rho_2,C_4\rho_1/2,R\}$, we have that, for any $z\in M_1\cap
	B(0,\rho_3)$ and $0<t<2\rho_3$, there exist cone $Z'(z,t) $ in $\Omega_0$ 
	with sliding boundary $L_0=\partial \Omega_0$, such that
	\[
		d_{x,t}(M_1,Z'(z,t))\leq C_7t^{\beta}.	
	\]
	For such cone $Z'(z,t)$, we have that $Z'(z,t)=w+\Tan(M_1,w)$, $w\in M_1$,
	$|w-z|\leq C_8 t$, and $w\in L_0\cap B(z,C_8t)$ in case $t\geq
	\dist(z,L_0)/2$. $Z'(z,t)$ may not pass through $z$, but the cone
	$Z(z,t)=Z'(z,t)-w+z$  pass through $z$, and 
	\[
		d_{x,t}(M_1,Z(z,t))\leq C_7t^{\beta} + C_8t\leq C_9 t^{\beta}.	
	\]
\end{proof}

\begin{proof}[Proof of Theorem \ref{mainthm}]
	Let $M_1$ be the same as in Lemma \ref{le:WAPPM}, and let
	$M=\Psi^{-1}(E\cap U)$. Then by Lemma \ref{le:WAPPM}, we have that 
	for any $x\in M_1\cap B(0,\rho_3)$ and $0<r<2\rho_3$, there exist cone 
	$Z(x,r)$ such that 
	\[
		d_{x,r}(M_1,Z(x,r))\leq Cr^{\beta},
	\]
	where $Z(x,r)$ is a minimal cone in $\mathbb{R}^3$ of type $\mathbb{P}$ or
	$\mathbb{Y}$ in case $x\not\in L_0$ and $t\leq \dist(x,L_0)$; and $Z(x,r)$ 
	is a sliding minimal cone in $\Omega_0$ with sliding boundary $L_0$ of type
	$\mathbb{P}_+$ or $\mathbb{Y}_+$ in other case. We apply Theorem
	\ref{thm:RPWAPS} to get that there exist  $\rho_4>0$, a sliding minimal cone
	$Z'$ centered at 0, and a mapping $\Phi_1:\Omega_0\cap B(0,\rho_4)\to
	\Omega_0$, which is a $C^{1,\beta}$-differential, such that $\Phi_1(0)=0$,
	$\Phi_1(\partial \Omega_0\cap B(0,\rho_4))\subseteq L_0$, $\|\Phi-\id\|\leq
	10^{-1}\rho_4$ and 
	\[
		M_1\cap B(0,\rho_4)=\Phi(Z')\cap B(0,\rho_4).
	\]
	We take $Z=Z'\cup L_0$, then we get that 
	\[
		M\cap B(0,\rho_4)=\Phi(Z)\cap B(0,\rho_4).
	\]
\end{proof}
\section{Existence of the Plateau problem with sliding boundary conditions}
The Plateau Problem with sliding boundary conditions arise in
\cite{David:2014p}, due to Guy David. That is, given an initial set $E_0$, and
boundary $\Gamma$, to find the minimizers among all competitors. The author of
the paper \cite{David:2014p} also gives some hint to the existence in Section 6,
and later on in \cite{David:2014}, he pave the way. We will give an existence
result in case the boundary is nice enough.

Let $\Omega\subseteq\mathbb{R}^{3}$ be a closed domain such that the boundary
$\partial\Omega$ is a $2$-dimensional manifold of class $C^{1,\alpha}$ for
some $\alpha>0$. Let $E_0\subseteq \Omega$ be a closed set with $E_0\supseteq
\partial\Omega$. We denote by $\mathscr{C}(E_0)$ the collection of all
competitors of $E_0$. 
\begin{theorem} \label{thm:espp}
	If there is a bounded minimizing sequence of competitors.
	Then there exists $E\in \mathscr{C}(E_0)$ such that 
	\[
		\HM^{2}(E\setminus \partial\Omega)=\inf\{\HM^2(S\setminus
		\partial\Omega): S\in \mathscr{C}(E_0)\}
	\]
\end{theorem}
\begin{proof}
	We put 
	\[
		m_0=\inf\{\HM^2(S\setminus \partial\Omega): S\in \mathscr{C}(E_0)\}.
	\]
	If $m_0=+\infty$, we have nothing to do.
	We now assume that $0\leq m_0<+\infty$.

	Let $\{S_i\}\subseteq \mathscr{C}_0$ be a sequence of competitors bounded by
	$B(0,R)$ such that 
	\[
		\lim_{i\to\infty}\HM^2(S_i\setminus \partial\Omega) = m_0.
	\]
	Apply Lemme 5.2.6 in \cite{Feuvrier:2008}, we can fined a sequence of open
	sets $\{ U_i \}$ and a sequence of competitors $\{ E_i \} \subseteq
	\mathscr{C}(E_0)$ of $E_0$ bounded by $B(0,R+1)$ such that
	\begin{itemize}[itemsep=0pt]
		\item $U_i\subseteq U_{i+1}$, $\cup_{i\geq 1} U_i=B(0,R+2)\setminus \partial
			\Omega$;
		\item $E_i\cap U_i\in QM(U_i,M,\diam(U_i))$ for constant $M>0$;
		\item $\HM^2(E_i)\leq \HM^2(S_i)+2^{-i}$.
	\end{itemize}
	We assume that $E_i$ converge locally to $E$ in $B(0,R+2)$, pass to 
	subsequence if necessary, then by Corollary 21.15 in \cite{David:2014}, we 
	get that $E$ is sliding minimal.

	We get, from Theorem \ref{mainthm} and Theorem 1.15 in \cite{David:2008}, 
	that $E$ is a Lipschitz neighborhood retract. But we see that $E_i$
	converges to $E$, we get so that $E$ contains a competitor.
\end{proof}

\bigskip
\footnotesize

Yangqin FANG, \textsc{Max-Planck-Institut f\"ur Gravitationsphysik,
Am M\"uhlenberg 1, 14476 Potsdam, Germany}\par\nopagebreak
\textit{E-mail address}: \texttt{yangqin.fang@aei.mpg.de}

\end{document}